\documentclass[reqno]{amsart}

\usepackage{amssymb}
\usepackage[dvipdfmx]{graphicx,xcolor}
\usepackage{tikz}

\usepackage{amsfonts}
\usepackage{mathrsfs}
\usepackage{bbm}

\usepackage{xcolor}

\theoremstyle{plain}
    \newtheorem{theorem}{Theorem}[section]
    \newtheorem*{theorem*}{Theorem}
    \newtheorem{lemma}[theorem]{Lemma}
    \newtheorem{proposition}[theorem]{Proposition}
    \newtheorem{corollary}[theorem]{Corollary}

\theoremstyle{definition}
    \newtheorem{definition}{Definition}[section]
    \newtheorem{remark}{Remark}[section]
    
\numberwithin{equation}{section}

\newcommand{\even}{{\mathord{\rm even}}}

\newcommand{\ess}{{\mathord{\rm ess}}}

\newcommand{\R}{\mathbb{R}}

\DeclareMathOperator{\im}{Im}
\DeclareMathOperator{\re}{Re}

\DeclareMathOperator{\supp}{supp}

\DeclareMathOperator{\sech}{sech}

\DeclareMathOperator{\spn}{span}

\begin{document}

\title[Even threshold solutions to NLS with $\delta$]{Threshold even solutions to the nonlinear Schr\"{o}dinger equation with delta potential at high frequencies}
\author[S. Gustafson]{Stephen Gustafson}
\address[S. Gustafson]{University of British Columbia, 1984 Mathematics Rd., Vancouver, Canada V6T1Z2.}
\email{gustaf@math.ubc.ca}
\author[T. Inui]{Takahisa Inui}
\address[T. Inui]{Department of Mathematics, Graduate School of Science, Osaka University, Toyonaka, Osaka, Japan 560-0043.
\newline 
University of British Columbia, 1984 Mathematics Rd., Vancouver, Canada V6T1Z2.}
\email{inui@math.sci.osaka-u.ac.jp 
}
\date{\today}
\date{\today}
\maketitle

\begin{abstract}
We consider the nonlinear Schr\"{o}dinger equation with a repulsive Dirac delta potential in one dimensional Euclidean space. We classify the global dynamics of even solutions with the same action as the 
high-frequency ground state standing wave solutions. 
\end{abstract}

\tableofcontents


\section{Introduction}

\subsection{Main result}

We consider the following nonlinear Schr\"{o}dinger equation with Dirac delta potential:
\begin{align}
\label{NLS}
\tag{$\delta$NLS}
	\begin{cases}
	i \partial_{t} u + \partial_{x}^{2} u +\gamma \delta u +|u|^{p-1}u=0, & (t,x) \in I \times \mathbb{R},
	\\
	u(0,x)=u_{0}(x), & x \in \mathbb{R},
	\end{cases}
\end{align}
where $\gamma <0$, $\delta$ denotes the Dirac delta at the origin, $p>5$, and $I \ni 0$ is a time interval. 
Here the Schr\"{o}dinger operator $-\Delta_{\gamma}=-\partial_{x}^{2} -\gamma \delta $ is defined by 
\begin{align*}
	- \Delta_{\gamma}f &:=- \partial_{x}^{2}f \text{ for } f \in \mathcal{D}(-\Delta_{\gamma}),
	\\
	\mathcal{D}(- \Delta_{\gamma}) &:= \{ f \in H^{1}(\mathbb{R}) \cap H^{2}(\mathbb{R}\setminus\{0\}): f'(0+) - f'(0-)=- \gamma f(0)\}. 
\end{align*}
We assume here that $\gamma<0$, meaning that the potential is repulsive. In this case, $-\Delta_{\gamma}$ is a nonnegative self-adjoint operator on $L^{2}(\mathbb{R})$ (see \cite{AGHH88} for more details), which implies that the linear Schr\"{o}dinger propagator $e^{it\Delta_{\gamma}}$ is well-defined on $L^{2}(\mathbb{R})$ by Stone's theorem. The operator $-\Delta_{\gamma}$ may also be defined by its quadratic form:
\begin{align*}
	\langle -\Delta_{\gamma}f , g \rangle := \int_{\mathbb{R}} f(x)\overline{g(x)} dx - \gamma f(0)\overline{g(0)} 
\end{align*} 
for $f,g \in H^{1}(\mathbb{R})$. 

It is known that the equation \eqref{NLS} is locally well-posed in the energy space $H^{1}(\mathbb{R})$. Moreover, the energy and mass
\begin{align*}
	E_{\gamma}(f)&:= \frac{1}{2}\|f\|_{\dot{H}_{\gamma}^{1}}^{2} - \frac{1}{p+1}\|f\|_{L^{p+1}}^{p+1},
	\\
	M(f)&:= \|f\|_{L^{2}}^{2},
\end{align*}
are conserved (see \cite{GHW04,FOO08}), where 
\[
  \|f\|_{\dot{H}_{\gamma}^{1}}^{2}:= \|\partial_{x} f\|_{L^{2}}^{2} - \gamma |f(0)|^{2} =\langle -\Delta_{\gamma} f , f \rangle.
\]
 
We are interested in the global behavior of the solutions. For this, 
the ground state solution plays an important role, as we now explain.
For a frequency $\omega > 0$, define the action and virial functionals
on $H^1(\R)$:
\begin{align*}
	S_{\omega,\gamma}(f)&:= E_{\gamma}(f) + \frac{\omega}{2}M(f),
	\\
	K_{\gamma}(f)&:=\|\partial_{x} f\|_{L^{2}}^{2} +\frac{-\gamma}{2}|f(0)|^{2} - \frac{p-1}{2(p+1)}\|f\|_{L^{p+1}}^{p+1}.
\end{align*}
We consider the following two minimizing problems:
\begin{align*}
	n_{\omega,\gamma}&:=\inf\{ S_{\omega,\gamma}(f): f \in H^{1}(\mathbb{R})\setminus\{0\}, K_{\gamma}(f)=0\},
	\\
	r_{\omega,\gamma}&:=\inf\{ S_{\omega,\gamma}(f): f \in H_{\even}^{1}(\mathbb{R})\setminus\{0\}, K_{\gamma}(f)=0\}, 
\end{align*}
where $H_{\even}^{1}(\mathbb{R})$ denotes the even functions in $H^{1}(\mathbb{R})$. We have: 

\begin{proposition}[\cite{FuJe08,IkIn17}]
\label{prop1.1}
Let $\gamma<0$. The following hold.
\begin{enumerate}
\item For all $\omega>0$, $n_{\omega,\gamma}=n_{\omega,0}$ and $n_{\omega,\gamma}$ is not attained.
\item $n_{\omega,\gamma} < r_{\omega,\gamma}$ and 
\begin{align*}
	r_{\omega,\gamma}
	\begin{cases}
	= 2 n_{\omega,0} & \text{ if } 0< \omega \leq \frac{\gamma^{2}}{4},
	\\
	< 2 n_{\omega,0} & \text{ if } \omega > \frac{\gamma^{2}}{4}.
	\end{cases}
\end{align*}
\item If $0 < \omega \leq \gamma^{2}/4$, then $r_{\omega,\gamma}$ is not attained. 
While, if $\omega> \gamma^{2}/4$, then $r_{\omega,\gamma}$ is attained by 
\begin{align*}
	Q_{\omega,\gamma}(x):= \left[ \frac{(p+1)\omega}{2} \sech^{2}\left\{ \frac{(p-1)\sqrt{\omega}}{2}|x| + \tanh^{-1}\left( \frac{\gamma}{2\sqrt{\omega}} \right) \right\} \right]^{\frac{1}{p-1}},
\end{align*}
which is the unique (up to a complex phase) solution of 
\begin{align}
\label{eleq}
	-\partial_{x}^{2}Q -\gamma \delta Q + \omega Q =|Q|^{p-1}Q.
\end{align}
\end{enumerate}
\end{proposition}

If $\omega>\gamma^{2}/4$, the function $e^{i\omega t}Q_{\omega,\gamma}$ is a non-scattering global solution to \eqref{NLS}, which is called the ground state (standing wave) solution. In the non-even setting, 
$n_{\omega,\gamma}=n_{\omega,0} = S_{\omega,0}(Q_{\omega,0})$.
Though $e^{i\omega t}Q_{\omega,0}$ is not a solution to \eqref{NLS}, it gives a threshold for a scattering-blowup dichotomy result. 
Indeed, we have the following global dynamics result below $n_{\omega,0}$: 

\begin{theorem}[Global dynamics below $Q_{\omega,0}$ (Ikeda--Inui \cite{IkIn17})]
Let $\omega>0$. Assume that $S_{\omega,\gamma}(u_{0}) < n_{\omega,\gamma}(=n_{\omega,0})$. Then the following hold for~\eqref{NLS}.
\begin{enumerate}
\item If $K_\gamma(u_{0})>0$, the solution is global and scatters in both time directions.
\item If $K_\gamma(u_{0})<0$, the solution blows up or grows up in both time directions. 
\end{enumerate}
\end{theorem}

We also have a dichotomy result in the threshold case:

\begin{theorem}[Global dynamics on the threshold (Ardila--Inui \cite{ArIn21}, Inui \cite{Inu21pre})]
\label{thm1.3}
Let $\omega>0$. Assume that $S_{\omega,\gamma}(u_{0}) = n_{\omega,\gamma}(=n_{\omega,0})$. Then the following hold for~\eqref{NLS}.
\begin{enumerate}
\item If $K_\gamma(u_{0})>0$, the solution is global and scatters in both time directions.
\item If $K_\gamma(u_{0})<0$ and $\int_{\mathbb{R}} |xu_{0}(x)|^{2} dx<\infty$, the solution blows up in both time directions. 
\end{enumerate}
\end{theorem}

\begin{remark}
\begin{enumerate}
\renewcommand{\theenumi}{\alph{enumi}}
\item We say a solution $u$ scatters in the positive time direction if there exists $u_{+} \in H^{1}(\mathbb{R})$ such that 
\begin{align*}
	\|u(t) - e^{it\Delta_{\gamma}} u_{+}\|_{H^{1}} \to 0 \text{ as } t \to \infty.
\end{align*}
By Mizutani \cite{Miz20}, if the above holds, then there exists $\widetilde{u}_{+}$ such that 
\begin{align*}
	\|u(t) - e^{it\partial_{x}^{2}} \widetilde{u}_{+}\|_{H^{1}} \to 0 \text{ as } t \to \infty.
\end{align*}
\item We say a solution $u$ grows up in positive time if it exists at least on $[0,\infty)$ and 
$\limsup_{t\to \infty}\|\partial_{x}u(t)\|_{L^{2}}=\infty$. 
On the threshold $S_{\omega,\gamma}(u_{0}) = n_{\omega,\gamma}$, it was not known if the blow-up or grow-up result holds without finite variance. The recent work \cite{GuIn22pre} shows that it does. 
\item The action condition $S_{\omega,\gamma}(u_{0}) \leq n_{\omega,\gamma}$ can be rewritten into the mass-energy condition $M(u_{0})^{(1-s_{c})/s_{c}}E_{\gamma}(u_{0}) \leq M(Q_{1,0})^{(1-s_{c})/s_{c}}E_{\gamma}(Q_{1,0})$ by using the scaling structure of $S_{\omega,0}(Q_{\omega,0})$, where $s_{c}:=1/2 - 2/(p-1)$. Moreover, the functional condition $K_{\gamma}(u_{0}) >0$ can be also written as $\|u_{0}\|_{L^{2}}^{1-s_{c}}\|u_{0}\|_{\dot{H}_{\gamma}^{1}}^{s_{c}} < \|Q_{1,0}\|_{L^{2}}^{1-s_{c}}\|Q_{1,0}\|_{\dot{H}^{1}}^{s_{c}}$. See e.g. \cite[Lemma 2.6]{ArIn21}. This condition is also equivalent to $\|u_{0}\|_{L^{2}}^{1-s_{c}}\|\partial_x u_{0}\|_{L^2}^{s_{c}} < \|Q_{1,0}\|_{L^{2}}^{1-s_{c}}\|Q_{1,0}\|_{\dot{H}^{1}}^{s_{c}}$. See \cite{HIIS22pre}. 

\item It is worth emphasizing that, in the threshold case $S_{\omega,\gamma}(u_{0}) = n_{\omega,\gamma}$, there is no non-scattering global solution like the ground state because of the repulsive potential (see \cite{DLR22,MMZ21,GuIn22} for NLS with other repulsive effects). 
\end{enumerate}
\end{remark}

For even solutions, we have the following global dynamics result below $r_{\omega,\gamma}$. 

\begin{theorem}[Global dynamics of even solutions below $r_{\omega,\gamma}$
(Ikeda--Inui \cite{IkIn17})]
Let $\omega>0$. Assume that $u_{0}$ is even and $S_{\omega,\gamma}(u_{0}) < r_{\omega,\gamma}$. Then the following hold for~\eqref{NLS}. 
\begin{enumerate}
\item If $K_\gamma(u_{0})>0$, the solution is global and scatters in both time directions.
\item If $K_\gamma(u_{0})<0$, the solution blows up or grows up in both time directions. 
\end{enumerate}
\end{theorem}

Since $r_{\omega,\gamma}$ is strictly larger than $n_{\omega,\gamma}$ by Proposition \ref{prop1.1}, the above theorem shows that the symmetry assumption raises the threshold. When $\omega>\gamma^{2}/4$, it is clear that the ground state solution is on the threshold, and is a non-scattering global solution. Thus, at least in the high-frequency case $\omega>\gamma^{2}/4$, we cannot expect a dichotomy result like Theorem \ref{thm1.3}. 

In the present paper, we are interested in the global dynamics of solutions on the threshold in the high-frequency and radial (i.e., even) setting. 
In that case, we expect the situation to be similar to that of the nonlinear Schr\"{o}dinger equation without potential \cite{DuRo10,CFR20}. 
Indeed, we get the following main results. 

\begin{theorem}[Existence of special solutions]
\label{thm1.5}
Let $\omega>\gamma^{2}/4$ be fixed. There exist two even solutions $U^{\pm}$ to \eqref{NLS} such that 
\begin{itemize}
\item $M(U^{\pm})=M(Q_{\omega,\gamma})$, $E_{\gamma}(U^{\pm})=E_{\gamma}(Q_{\omega,\gamma})$, $U^{\pm}$ exist at least on $[0,\infty)$ and there exists $c>0$ such that 
\begin{align*}
	\|U^{\pm}(t)-e^{i\omega t}Q_{\omega,\gamma}\|_{H^{1}} \lesssim e^{-ct} \text{ for } t\geq 0.
\end{align*}
\item $K_{\gamma}(U^{+}(0))<0$ and $U^{+}$ blows up in finite negative time. 
\item $K_{\gamma}(U^{-}(0))>0$ and $U^{-}$ scatters backward in time.  
\end{itemize}
\end{theorem}

\begin{theorem}[Global dynamics on the threshold]
\label{thm1.6}
Let $\omega>\gamma^{2}/4$. Assume that $u_{0} \in H^1(\mathbb{R})$ is even and satisfies the mass-energy condition
\begin{align}
\label{ME}
\tag{ME}
	M(u_{0})=M(Q_{\omega,\gamma}) \text{ and } E_{\gamma}(u_{0})=E_{\gamma}(Q_{\omega,\gamma}).
\end{align}
Then the following are true for~\eqref{NLS}. 
\begin{enumerate}
\item If $K_{\gamma}(u_{0})>0$, the solution $u$ scatters in both time directions, or else $u=U^{-}$ up to symmetry.
\item If $K_{\gamma}(u_{0})=0$, then $u=e^{i\omega t}Q_{\omega,\gamma}$ up to symmetry.
\item If $K_{\gamma}(u_{0})<0$ and $\int_{\mathbb{R}}|xu_{0}(x)|^{2}dx<\infty$, the solution $u$ blows up in both time directions, or else $u=U^{+}$ up to symmetry. 
\end{enumerate}
\end{theorem}

\begin{remark}
\begin{enumerate}
\item The phrase ``up to symmetry'' means up to multiplication by a complex phase, time translation, and time reversibility. 
\item We need not consider $S_{\omega,\gamma}(f) = r_{\omega,\gamma}$ for $\omega>\gamma^{2}/4$ since we have the following statement: 
If $f$ satisfies $S_{\omega,\gamma}(f) = r_{\omega,\gamma}$ for some $\omega > \gamma^{2}/4$, then either $f$ satisfies $S_{\omega',\gamma}(f) < r_{\omega',\gamma}$ for some $\omega' > 0$ or else $M(f)=M(Q_{\omega,\gamma})$ and $E_{\gamma}(f)=E_{\gamma}(Q_{\omega,\gamma})$. We will give the proof in Appendix \ref{appA}. 
\item We assume finite variance to show finite time blow-up. This can be removed by allowing for grow-up, and using the method of \cite{GuIn22pre},
but we do not pursue it here. 
\end{enumerate}
\end{remark}


\subsection{Background and idea of proof}

Global dynamics below and at the ground state level for the nonlinear Schr\"{o}dinger equation without potential is well investigated. See \cite{FXC11,AkNa13,Gue14,DWZ16,Guo16} for the dichotomy result below the ground state. Duyckaerts and Roundenko \cite{DuRo10} clarified the global dynamics on the threshold for 3d cubic NLS. Campos, Farah, and Roudenoko \cite{CFR20} generalized it to any dimensions and all mass-supercritical and energy-subcritical powers (and also treat the energy-critical nonlinearity). See also \cite{GuIn22pre} for the blow-up or grow-up result on the threshold. 

Our proof relies heavily on those of Duyckaerts and Roudenko \cite{DuRo10} and Campos, Farah, and Roudenko \cite{CFR20}. First, we show the existence of a family of solutions going to the ground state exponentially in positive time. In this step, the eigenfunctions of the linearized operator around the ground state play an important role. In the potential-free case, the corresponding eigenfunctions are of Schwartz class on $\mathbb{R}$, while in the present case, they are not smooth at the origin due to the delta interaction. We must introduce a cut-off function to remove the origin, somewhat complicating the proof. In the next step, we show that a threshold solution tends to the ground state if $K_\gamma > 0$ and it it non-scattering, or if $K_\gamma < 0$ and it is global. A modulation argument and the virial identity are crucial to this argument, but here we meet a difficulty. To control the modulation parameters, we use the functional
$\mu(u):= \|Q_{\omega,\gamma}\|_{\dot{H}_{\gamma}^{1}}^{2}-\|u\|_{\dot{H}_{\gamma}^{1}}^{2}$,
while the functional appearing in the virial identity is $K_\gamma(u)$. 
When $\gamma=0$, these are same under the assumption \eqref{ME}, and so the modulation argument and the virial identity are compatible. However, these are different when $\gamma<0$. To overcome this, we show that $K_\gamma (u)- c\mu (u) >0$ (resp. $<0$) when $K_\gamma >0$ (resp. $<0$), by characterizing the ground state in terms of a functional $K_{\omega,\gamma}^{\alpha,\beta}$ generalizing the 
virial and Nehari functionals. Though such an argument can be also seen in Ikeda and Inui \cite{IkIn17}, we need to treat a more general functional than that used there. 
Finally, we show that the solutions tending to the ground state exponentially are unique up to symmetries, by a bootstrap argument,
and investigate the behavior of those solutions in negative time. To determine the behavior of a solution with $K_{\gamma}(u_0)<0$ in negative time, we must show the solution has finite variance. In the higher dimensional cases, this fact follows from the radial symmetry (see e.g. Remark 5.2 in \cite{DuRo10}) since we can use the radial Sobolev inequality. However, the argument does not work in one dimensional case. Instead of radiality, we use the convergence to the ground state, which itself is of finite variance. See \cite{GuIn22pre} for such an argument.

In the low frequency case $0<\omega \leq \gamma^{2}/4$, there is no ground state at the threshold $r_{\omega,\gamma}$. The global dynamics is an interesting open problem. 

Though the global dynamics for NLS without potential above the ground state has been studied (see e.g. \cite{NaSc12,NaRo16,Nak17,AIKN21}), the global dynamics of \eqref{NLS} above $n_{\omega,\gamma}$ in the non-even setting, and above $r_{\omega,\gamma}$ in the even setting are open problems. 

The paper is organized as follows. In Section \ref{sec2}, we prepare some lemmas for variational arguments, linearized operators, and such. In Proposition \ref{prop2.7}, we relate $K_\gamma$ and $\mu$, as described above. Section \ref{sec3} is devoted to the construction of the family of special solutions. Section \ref{sec4} shows the convergence of threshold solutions to the ground state. In particular, we set up the modulation in Section \ref{sec4.1}, we treat the case of $K_\gamma>0$ in Section \ref{sec4.2.1}, and the case of $K_\gamma<0$ in Section \ref{sec4.2.2}. In Section \ref{sec5}, we prove that the family is the unique solution up to symmetries. Appendix \ref{appA} discusses the envelope of the threshold condition's family in terms of $\omega$. In Appendix \ref{appB}, we give some remarks about Schwartz class with cut-off.

\subsection{Notations}
We define $\langle x \rangle := (1+|x|^{2})^{1/2}$. 
We use $\langle F , \varphi \rangle:=F(\varphi)$ 
for $F \in H^{-1}(\mathbb{R})$ and $\varphi \in H^{1}(\mathbb{R})$. We define the inner products
\begin{align*}
	(f,g)_{L^{2}}&:=\re \int_{\mathbb{R}} f(x) \overline{g(x)} dx,
	\\
	(u,v)_{H_{\omega,\gamma}^{1}}&:= \re \int_{\mathbb{R}} \partial_{x}u(x) \overline{\partial_{x}v(x)} dx - \gamma \re \{u(0)\overline{v(0)} \}+ \omega \re \int_{\mathbb{R}} u(x) \overline{v(x)} dx
	\\
	(u,v)_{\dot{H}_{\gamma}^{1}}&:= \re \int_{\mathbb{R}} \partial_{x}u(x) \overline{\partial_{x}v(x)} dx - \gamma \re \{ u(0)\overline{v(0)}\}
\end{align*}
for $f,g \in L^{2}(\mathbb{R})$ and $u,v \in H^{1}(\mathbb{R})$. 
The norms of $H_{\omega,\gamma}^{1}$ and $\dot{H}_{\gamma}^{1}$ are defined by the above inner products: 
\begin{align*}
	\|f\|_{\dot{H}_{\gamma}^{1}}^{2}&:= \|\partial_{x} f\|_{L^{2}}^{2} - \gamma |f(0)|^{2}
	=\|\sqrt{-\Delta_{\gamma}}f\|_{L^2}^2,
	\\
	\|f\|_{H_{\omega,\gamma}^{1}}^{2}&:= \|\partial_{x} f\|_{L^{2}}^{2} - \gamma |f(0)|^{2} +\omega\|f\|_{L^{2}}^{2}
	=\|\sqrt{\omega -\Delta_{\gamma}}f\|_{L^2}^2,
\end{align*}
for $f \in H^{1}(\mathbb{R})$. We note that $\dot{H}_{\gamma}^{1}=H_{0,\gamma}^{1}$, and that $H_{\omega,\gamma}^{1}$ is equivalent to $H^{1}$. 
Let $\mathbb{N}$ denote the nonnegative integers. 
We denote the set of smooth functions such that all the derivatives are bounded on $\mathbb{R}$ by
\begin{align*}
	BC^\infty(\mathbb{R}) := \{f \in C^\infty(\mathbb{R}): \partial_x^\alpha f \in L^\infty(\mathbb{R}), \forall \alpha \in \mathbb{N}\}.
\end{align*}
For a function space $X(\mathbb{R})$, we sometimes omit $\mathbb{R}$ when it is clear. For a function space $X$, we denote 
\begin{align*}
	X_0&:= \{f \in X:  f(x)=0 \text{ on an open interval including } 0\},
	\\
	X_c&:=  \{f \in X:  \supp f \text{ is compact } \},
	\\
	X_\even &:= \{f \in X:  f \text{ is even } \}.
\end{align*}
We may also mix these notations, e.g. $X_{c,0}$ and $X_{c,0,\even}$. 
Let $\mathcal{S}(\mathbb{R})$ be the Schwartz space. We define 
\begin{align*}
	\widetilde{\mathcal{S}}(\mathbb{R}):= \{f \in C^\infty(\mathbb{R}\setminus\{0\}): \varphi f \in \mathcal{S}(\mathbb{R}), \forall \varphi \in BC_0^\infty(\mathbb{R})\}.
\end{align*}
We set 
\begin{align*}
	\mathcal{D}:=\mathcal{D}(-\Delta_\gamma)=\{ f \in H^{1}(\mathbb{R}) \cap H^{2}(\mathbb{R}\setminus\{0\}): f'(0+) - f'(0-)=- \gamma f(0)\},
\end{align*}
and note that $\mathcal{D} \subset W^{1,\infty}$ holds by the Sobolev embedding. 

For a time interval $I$, we define space-time norms by 
\begin{align*}
	\|F\|_{L_t^qX(I)} := \|\|F\|_X \|_{L_t^q(I)},
\end{align*}
where $X$ denotes the function space for the spatial variable. 

\section{Preliminaries}
\label{sec2}

In this section, we give some lemmas which are used in the sequel. In what follows, $\omega>\gamma^{2}/4$ is always assumed. We assume that the initial data $u_{0} \in H_\even^{1}(\mathbb{R})$ satisfies \eqref{ME} in this section.  

\subsection{Variational argument}

In this section, we revisit the variational argument not only for the 
virial and Nehari functionals but also for more general functionals $K_{\omega,\gamma}^{\alpha,\beta}$, defined by
%
\begin{align*}
	K_{\omega,\gamma}^{\alpha,\beta}(f)
	&:=\partial_{\lambda}S_{\omega,\gamma}(e^{\alpha \lambda}f(e^{\beta\lambda}\cdot))|_{\lambda=0}
	\\
	&=\frac{2\alpha+\beta}{2}\|\partial_{x}f\|_{L^{2}}^{2}
	-\alpha\gamma|f(0)|^{2}
	+\omega \frac{2\alpha-\beta}{2}\|f\|_{L^{2}}^{2}
	-\frac{(p+1)\alpha-\beta}{p+1}\|f\|_{L^{p+1}}^{p+1}.
\end{align*}
By definition, $(\alpha,\beta)=(1/2,1)$ gives the virial functional, that is $K_{\gamma}(f)=K_{\omega,\gamma}^{1/2,1}(f)$, and $(\alpha,\beta)=(1,0)$ gives the Nehari functional, that is
\begin{align*}
	I_{\omega,\gamma}(f):=K_{\omega,\gamma}^{1,0}(f)
	=\|\partial_{x}f\|_{L^{2}}^{2} -\gamma|f(0)|^{2} +\omega\|f\|_{L^{2}}^{2} -\|f\|_{L^{p+1}}^{p+1}. 
\end{align*}

If  $(\alpha,\beta)$ satisfies
\begin{align}
\label{eq2.1}
	\alpha>0, 2\alpha-\beta\geq 0, 2\alpha+\beta\geq0,
\end{align}
we obtain the following lemma by \cite[Lemmas 2.7--2.10]{IkIn17}. 

\begin{lemma}
\label{lem2.3}
Let $(\alpha,\beta)$ satisfy \eqref{eq2.1}. 
We have
\begin{align*}
	S_{\omega,\gamma}(Q_{\omega,\gamma})
	=\inf \{S_{\omega,\gamma}(f):f\in H_{\even}^{1}(\mathbb{R})\setminus \{0\}, K_{\omega,\gamma}^{\alpha,\beta}(f)=0\}.
\end{align*}
\end{lemma}

We define 
\begin{align*}
	\mathcal{K}_{\omega,\gamma}^{\alpha,\beta,+}
	&:=\{f\in H_{\even}^{1}(\mathbb{R}): M(f)=M(Q_{\omega,\gamma}), E_{\gamma}(f)=E_{\gamma}(Q_{\omega,\gamma}), K_{\omega,\gamma}^{\alpha,\beta}(f)>0\},
	\\
	\mathcal{K}_{\omega,\gamma}^{\alpha,\beta,0}
	&:=\{f\in H_{\even}^{1}(\mathbb{R}): M(f)=M(Q_{\omega,\gamma}), E_{\gamma}(f)=E_{\gamma}(Q_{\omega,\gamma}), K_{\omega,\gamma}^{\alpha,\beta}(f)=0\},
	\\
	\mathcal{K}_{\omega,\gamma}^{\alpha,\beta,-}
	&:=\{f\in H_{\even}^{1}(\mathbb{R}): M(f)=M(Q_{\omega,\gamma}), E_{\gamma}(f)=E_{\gamma}(Q_{\omega,\gamma}), K_{\omega,\gamma}^{\alpha,\beta}(f)<0\}.
\end{align*}

We know that $\mathcal{K}_{\omega,\gamma}^{\alpha,\beta,0}=\{e^{i\theta}Q_{\omega,\gamma}: \theta \in \mathbb{R}\}$ by Proposition \ref{prop1.1} and Lemma \ref{lem2.3}.

\begin{lemma}
$\mathcal{K}_{\omega,\gamma}^{\alpha,\beta,+}$ is independent of $(\alpha,\beta)$; that is, $\mathcal{K}_{\omega,\gamma}^{\alpha,\beta,+}=\mathcal{K}_{\omega,\gamma}^{\alpha',\beta',+}$ for $(\alpha,\beta)$ and $(\alpha',\beta')$ satisfying \eqref{eq2.1}. The same holds for $\mathcal{K}_{\omega,\gamma}^{\alpha,\beta,-}$. 
\end{lemma}

\begin{proof}
Let $(\alpha,\beta)$ and $(\alpha',\beta')$ satisfy \eqref{eq2.1}. 
If $\mathcal{K}_{\omega,\gamma}^{\alpha,\beta,+}\subsetneq \mathcal{K}_{\omega,\gamma}^{\alpha',\beta',+}$, then there exists a function $f \in H_\even^1$ such that 
\begin{align*}
	M(f)=M(Q_{\omega,\gamma}),\ 
	E_{\gamma}(f)=E_{\gamma}(Q_{\omega,\gamma}),\ 
	K_{\omega,\gamma}^{\alpha,\beta}(f)\leq 0, \text{ and }
	K_{\omega,\gamma}^{\alpha',\beta'}(f)>0.
\end{align*}
If $K_{\omega,\gamma}^{\alpha,\beta}(f)= 0$, then $f=e^{i\theta}Q_{\omega,\gamma}$ and thus $K_{\omega,\gamma}^{\alpha',\beta'}(f)=0$. This is a contradiction. Hence $K_{\omega,\gamma}^{\alpha,\beta}(f)< 0$. Let $f_{\lambda}^{\alpha,\beta}(x)=e^{\alpha\lambda}f(e^{\beta\lambda}x)$. Then we have
\begin{align*}
	0>K_{\omega,\gamma}^{\alpha,\beta}(f) = \partial_{\lambda} S_{\omega,\gamma}(f_{\lambda}^{\alpha,\beta}) |_{\lambda=0}.
\end{align*}
This means that $S_{\omega,\gamma}(f_{\lambda}^{\alpha,\beta})$ is strictly decreasing in $\lambda$ near $\lambda=0$ and thus $S_{\omega,\gamma}(f_{\lambda}^{\alpha,\beta}) < S_{\omega,\gamma}(f)=S_{\omega,\gamma}(Q_{\omega,\gamma})$ for $\lambda>0$ sufficiently close to $0$. Since we have $K_{\omega,\gamma}^{\alpha,\beta}(f) <0$, it holds that $K_{\omega,\gamma}^{\alpha,\beta}(f_{\lambda}^{\alpha,\beta})<0$ for $\lambda>0$ sufficiently small. This means that $f_{\lambda}^{\alpha,\beta}$ belongs to 
\begin{align*}
	\mathcal{R}_{\omega,\gamma}^{\alpha,\beta,-}
	&:=\{f\in H_{\even}^{1}(\mathbb{R}): S_{\omega,\gamma}(f)<S_{\omega,\gamma}(Q_{\omega,\gamma}), K_{\omega,\gamma}^{\alpha,\beta}(f)<0\}.
\end{align*}
By Proposition 2.15 in \cite{IkIn17}, it holds that $\mathcal{R}_{\omega,\gamma}^{\alpha,\beta,-}=\mathcal{R}_{\omega,\gamma}^{\alpha',\beta',-}$ and thus we get $f_{\lambda}^{\alpha,\beta} \in \mathcal{R}_{\omega,\gamma}^{\alpha',\beta',-}$. In particular, $K_{\omega,\gamma}^{\alpha',\beta'}(f_{\lambda}^{\alpha,\beta})<0$ for $\lambda>0$ sufficiently close to $0$. Taking the limit $\lambda \searrow 0$, $K_{\omega,\gamma}^{\alpha',\beta'}(f)\leq 0$. This is a contradiction. Thus we get the statement. 
\end{proof}

By this lemma, we may omit $(\alpha,\beta)$ from $\mathcal{K}_{\omega,\gamma}^{\alpha,\beta,+}$ and so on. For example, we use $\mathcal{K}_{\omega,\gamma}^{+}$ instead of $\mathcal{K}_{\omega,\gamma}^{\alpha,\beta,+}$. By the variational structure of $Q_{\omega,\gamma}$, the sets are invariant by the flow.

\begin{lemma}
\label{lem2.5}
Let 
$u$ be the solution to \eqref{NLS} with $u(0)=u_{0}$.
\begin{enumerate}
\item If $K_{\gamma}(u_{0})>0$, then $u(t)\in \mathcal{K}_{\omega,\gamma}^{+}$ for the time of existence. 
In particular, the solution is global. 
\item If $K_{\gamma}(u_{0})=0$, then $u_{0}=e^{i\theta}Q_{\omega,\gamma}$ for some $\theta \in \mathbb{R}$. 
\item If $K_{\gamma}(u_{0})<0$, then $u(t)\in \mathcal{K}_{\omega,\gamma}^{-}$ for the time of existence. 
\end{enumerate}
\end{lemma}

\begin{proof}
Lemma \ref{lem2.3} and an argument similar to that in \cite[Lemma 2.17]{IkIn17} give the statement. 
\end{proof}

\begin{lemma}
\label{lem2.6}
Let $u$ be the solution to \eqref{NLS} with $u(0)=u_{0}$. Then
\begin{align*}
	K_{\omega,\gamma}^{\alpha,\beta}(u(t)) 
	&= \frac{(p-1)\alpha-2\beta}{2}\|\partial_{x}Q_{\omega,\gamma}\|_{L^{2}}^{2} +\frac{(p-1)\alpha-\beta}{2}(-\gamma)|Q_{\omega,\gamma}(0)|^{2}
	\\
	&\quad -\left(  \frac{(p-1)\alpha-2\beta}{2}\|\partial_{x}u(t)\|_{L^{2}}^{2} +\frac{(p-1)\alpha-\beta}{2}(-\gamma)|u(t,0)|^{2}\right)
\end{align*}
for the time of existence. 
\end{lemma}

\begin{proof}
This is a direct calculation using \eqref{ME}.
\end{proof}

We define 
\begin{align*}
	\mu(f):= \|Q_{\omega,\gamma}\|_{\dot{H}_{\gamma}^{1}}^{2} - \|f\|_{\dot{H}_{\gamma}^{1}}^{2}.
\end{align*}
This is nothing but the Nehari functional under \eqref{ME}: $I_{\omega,\gamma}(u(t))= \frac{p-1}{2} \mu(u(t))$. Therefore, Lemma \ref{lem2.5} implies that the sign of $\mu$ is conserved by the flow.
This functional $\mu$ is used to control the modulation parameters. 

The following proposition allows for control of the virial functional $K_\gamma$ by $\mu$. 
\begin{proposition}
\label{prop2.7}
We have the following statements.
\begin{enumerate}
\item 
If $K_{\gamma}(u_{0})>0$, then there exists $c=c(p,u_{0})>0$ such that $K_{\gamma}(u(t))\geq  c \mu(u(t))>0$ for $t \in \mathbb{R}$. 
\item If $K_{\gamma}(u_{0})<0$, then there exists $c=c(p,u_{0})>0$ such that $K_{\gamma}(u(t)) \leq c \mu(u(t))<0$ for the time of existence.
\end{enumerate}
\end{proposition}

To show this, we prove the following two lemmas. 

\begin{lemma}
\label{lem2.8}
We have $K_{\gamma}(f) - c \mu(f)=K_{\omega,\gamma}^{\frac{1}{2}-\frac{2c}{p-1},1}(f)$ for $f \in H^1$ satisfying \eqref{ME} 
and $c\in \mathbb{R}$. 
\end{lemma}

\begin{proof}
This follows from direct calculation and Lemma \ref{lem2.6}.
\end{proof}

We note that $(\frac{1}{2}-\frac{2c}{p-1},1)$ does not satisfy \eqref{eq2.1}. However, the functional $K_{\omega,\gamma}^{\frac{1}{2}-\frac{2c}{p-1},1}$ still 
characterizes the ground state if we take $c$ small depending on $p$. 

\begin{remark}
Let $(\alpha,\beta)$ not satisfying \eqref{eq2.1} be given. Then, it is shown by \cite[Proposition A.1]{IMN11} that there exists $p>5$ such that 
\begin{align*}
	\inf \{S_{\omega,0}(\varphi): \varphi\in H^{1}(\mathbb{R})\setminus\{0\}, K_{\omega,0}^{\alpha,\beta}(\varphi)=0\}
	=-\infty.
\end{align*}
Our statement is different. Indeed, for given $p>5$, there exists $c=c(p)$ such that the infimum with $(\alpha,\beta)=(\frac{1}{2}-\frac{2c}{p-1},1)$ is positive, as we will see. 
\end{remark}

\begin{lemma}
\label{lem2.9}
Let $0<c<(p-5)/4$. 
If $f \in H_\even^1$ satisfies \eqref{ME} and $K_{\omega,\gamma}^{\frac{1}{2}-\frac{2c}{p-1},1}(f)=0$, then $f=e^{i\theta}Q_{\omega,\gamma}$ for some $\theta$. 
\end{lemma}

\begin{proof}
We consider the following minimizing problem:
\begin{align*}
	m_{\omega,\gamma}:= \inf \{S_{\omega,\gamma}(f) : f\in H_\even^{1}(\mathbb{R})\setminus\{0\}, K_{\omega,\gamma}^{\frac{1}{2}-\frac{2c}{p-1},1}(f)=0\}.
\end{align*}
In general, we have
\begin{align}
\notag
	T_{\omega,\gamma}^{\alpha,\beta}(f)&:=S_{\omega,\gamma}(f) - \frac{1}{(p+1)\alpha-\beta}K_{\omega,\gamma}^{\alpha,\beta}(f)
	\\ \notag
	&= \frac{(p-1)\alpha-2\beta}{2\{(p+1)\alpha-\beta\}} \|\partial_{x}f\|_{L^{2}}^{2}
	-\gamma \frac{p\alpha-\beta}{2\{(p+1)\alpha-\beta\}} |f(0)|^{2}
	\\ \label{eq2.2}
	&\quad 
	+ \frac{(p-1)\alpha}{2\{(p+1)\alpha-\beta\}}\omega \|f\|_{L^{2}}^{2}
\end{align}
for $f \in H^{1}(\mathbb{R})$ and $(\alpha,\beta)\in \mathbb{R}^{2}$ with $(p+1)\alpha-\beta \neq 0$. When $(\alpha,\beta)=(\frac{1}{2}-\frac{2c}{p-1},1)$, 
all the coefficients of \eqref{eq2.2} are positive. Therefore, we get
\begin{align*}
	m_{\omega,\gamma}=\inf \{ T_{\omega,\gamma}^{\frac{1}{2}-\frac{2c}{p-1},1}(f): f \in H_\even^{1}(\mathbb{R}),K_{\omega,\gamma}^{\frac{1}{2}-\frac{2c}{p-1},1}(f)=0 \} \geq 0. 
\end{align*}
Define sets
\begin{align*}
	\mathcal{M}_{\omega,\gamma}&:=\{f \in H_\even^{1}(\mathbb{R}): K_{\omega,\gamma}^{\frac{1}{2}-\frac{2c}{p-1},1}(f)=0, S_{\omega,\gamma}(f)= m_{\omega,\gamma}\},
	\\
	\mathcal{A}_{\omega,\gamma}&:=\{f \in H_\even^{1}(\mathbb{R})\setminus \{0\}: S_{\omega,\gamma}'(f)=0\}.  
\end{align*}
Let $\varphi \in \mathcal{M}_{\omega,\gamma}$. By the Lagrange multiplier method, there exists a constant $\nu$ such that $S_{\omega,\gamma}'(\varphi)=\nu (K_{\omega,\gamma}^{\frac{1}{2}-\frac{2c}{p-1},1})'(\varphi)$. 
Therefore, letting $\mathscr{L}^{\alpha,\beta}f(x)=\partial_\lambda \{e^{\alpha\lambda}f(e^{\beta\lambda}x)\}|_{\lambda=0}$, we obtain 
\begin{align*}
	0=(K_{\omega,\gamma}^{\frac{1}{2}-\frac{2c}{p-1},1})(\varphi) 
	=\langle S_{\omega,\gamma}'(\varphi) , \mathscr{L}^{\frac{1}{2}-\frac{2c}{p-1},1}\varphi \rangle
	=\nu \langle (K_{\omega,\gamma}^{\frac{1}{2}-\frac{2c}{p-1},1})'(\varphi) , \mathscr{L}^{\frac{1}{2}-\frac{2c}{p-1},1}\varphi \rangle
\end{align*}
Now we have
$ \langle (K_{\omega,\gamma}^{\frac{1}{2}-\frac{2c}{p-1},1})'(\varphi) , \mathscr{L}^{\frac{1}{2}-\frac{2c}{p-1},1}\varphi \rangle < 0$ since $K_{\omega,\gamma}^{\frac{1}{2}-\frac{2c}{p-1},1}(\varphi)=0$. 
This implies $\nu=0$. Thus, $S_{\omega,\gamma}'(f)=0$. This means that $\mathcal{M}_{\omega,\gamma} \subset \mathcal{A}_{\omega,\gamma}$. It is known by Proposition \ref{prop1.1} that $\mathcal{A}_{\omega,\gamma}=\{e^{i\theta}Q_{\omega,\gamma}: \theta \in \mathbb{R}\}$. Hence, the mass-energy condition \eqref{ME} and $K_{\omega,\gamma}^{\frac{1}{2}-\frac{2c}{p-1},1}(f)=0$ give us $f=e^{i\theta}Q_{\omega,\gamma}$. 
\end{proof}

We are ready to prove Proposition \ref{prop2.7}. 

\begin{proof}[Proof of Proposition \ref{prop2.7}]
Let $K_{\gamma}(u_{0})>0$. 
By Lemma \ref{lem2.8}, we have
\begin{align*}
	K_{\gamma}(u(t)) -c \mu (u(t)) = K_{\omega,\gamma}^{\frac{1}{2}-\frac{2c}{p-1},1}(u(t)).
\end{align*}
Suppose that there exists a time $t_{*}$ such that $K_{\omega,\gamma}^{\frac{1}{2}-\frac{2c}{p-1},1}(u(t_{*}))=0$. Then we find that $u(t_{*})=e^{i\theta}Q_{\omega,\gamma}$ by Lemma \ref{lem2.9}. It follows that $K_{\gamma}(u(t_{*}))=0$, contradicting $K_{\gamma}(u_{0})>0$. Therefore, $K_{\omega,\gamma}^{\frac{1}{2}-\frac{2c}{p-1},1}(u(t))\neq 0$  for the time of existence. Now, we have
\begin{align*}
	K_{\omega,\gamma}^{\frac{1}{2}-\frac{2c}{p-1},1}(u_{0}) \to K_{\gamma}(u_{0})>0 \text{ as } c \to 0.
\end{align*}
Hence $K_{\omega,\gamma}^{\frac{1}{2}-\frac{2c}{p-1},1}(u_{0})>0$ if $c$ is sufficiently small depending on $u_{0}$. Therefore, we obtain $K_{\omega,\gamma}^{\frac{1}{2}-\frac{2c}{p-1},1}(u(t))>0$ for the time of existence, if $c=c(p,u_{0})>0$ is sufficiently small. So there is $c=c(p,u_{0})>0$ such that $K_{\gamma}(u(t))\geq  c \mu(u(t))>0$ for $t \in \mathbb{R}$. 
The case $K_{\gamma}(u_{0})<0$ is proved in the same way.
\end{proof}

\subsection{The Gagliardo--Nirenberg inequality}

A Gagliardo--Nirenberg type inequality follows easily from the usual one: for $\gamma<0$,
\begin{align*}
	\|f\|_{L^{p+1}}^{p+1} \leq C_{GN} \|f\|_{L^{2}}^{\frac{p+3}{2}} \|f\|_{\dot{H}_{\gamma}^{1}}^{\frac{p-1}{2}}. 
\end{align*}
The optimal constant $C_{GN}$ is same as for the usual case: 
\begin{align*}
	C_{GN}^{-1}=\frac{\|Q_{1,0}\|_{L^{2}}^{\frac{p+3}{2}} \|Q_{1,0}\|_{\dot{H}^{1}}^{\frac{p-1}{2}}}{\|Q_{1,0}\|_{L^{p+1}}^{p+1}}.
\end{align*}
See e.g. \cite{Inu21pre,HIIS22pre}. 
This is not suitable to our problem since we consider a higher threshold than $Q_{1,0}$. We give here a Gagliardo--Nirenberg-type inequality 
with respect to the Dirac delta potential.

\begin{lemma}[Gagliardo--Nirenberg type inequality w.r.t $\delta$ potential]
\label{GN}
\begin{align*}
	\|f\|_{L^{p+1}}^{2} \leq C_{\omega,\gamma} \|f\|_{{H}_{\omega,\gamma}^{1}}^{2}.
\end{align*}
Moreover, the optimal constant $C_{\omega,\gamma}$ is attained by $Q_{\omega,\gamma}$ and 
\begin{align*}
	C_{\omega,\gamma}^{-1}
	=\frac{\|Q_{\omega,\gamma}\|_{{H}_{\omega,\gamma}^{1}}^{2}}{\|Q_{\omega,\gamma}\|_{L^{p+1}}^{2}}=\left\{\frac{2(p+1)}{p-1}S_{\omega,\gamma}(Q_{\omega,\gamma})\right\}^{\frac{p-1}{p+1}}=\|Q_{\omega,\gamma}\|_{L^{p+1}}^{p-1}.
\end{align*}
\end{lemma}

\begin{proof}
This estimate follows directly from the characterization of $Q_{\omega,\gamma}$ by the Nehari functional $I_{\omega,\gamma}$,
\begin{align*}
	S_{\omega,\gamma}(Q_{\omega,\gamma})
	=\inf\{S_{\omega,\gamma}(f): f \in H_\even^1(\mathbb{R}) \setminus\{0\}, I_{\omega,\gamma}(f)=0\},
\end{align*}
so we omit the proof. 
\end{proof}

\subsection{Localized virial indentity}

Let $u$ be a solution to \eqref{NLS}. We define
\begin{align*}
	J(t)=J_{\infty}(t) := \int_{\mathbb{R}} |x|^{2}|u(t,x)|^{2} dx
\end{align*}
and its localized version by
\begin{align}
\label{J_R}
	J_{R}(t) := \int_{\mathbb{R}} R^{2} \varphi\left(\frac{x}{R}\right) |u(t,x)|^{2} dx,
\end{align}
where $\varphi$ is an even function in $C_{c}^{\infty}(\mathbb{R})$ such that
\begin{align*}
	\varphi(x) =
	\begin{cases}
	x^{2}, & (|x|<1),
	\\
	0, & (|x|>2).
	\end{cases}
\end{align*}
Then we have 
\begin{align*}
	J'(t)&=4 \im \int_{\mathbb{R}} x \overline{u(t,x)} \partial_{x}u(t,x) dx,
	\\
	J''(t)&=8K_{\gamma}(u(t)),
\end{align*}
and for the localized version,
\begin{align*}
	J_{R}'(t)&=2R\im \int_{\mathbb{R}} (\partial_{x}\varphi)\left(\frac{x}{R}\right) \overline{u(t,x)} \partial_{x}u(t,x)dx
	\\
	J_{R}''(t)
	&=F_R(u(t)),
\end{align*}
where we set
\begin{align*}
	F_{R}(f)&:=4\int_{\mathbb{R}} (\partial_{x}^{2}\varphi)\left(\frac{x}{R}\right) |\partial_{x}f(x)|^{2} dx 
	-4\gamma |f(0)|^{2} 
	\\
	&\quad -\frac{2(p-1)}{p+1} \int_{\mathbb{R}} (\partial_{x}^{2}\varphi)\left(\frac{x}{R}\right) |f(x)|^{p+1} dx
	-\frac{1}{R^{2}} \int_{\mathbb{R}}(\partial_{x}^{4}\varphi)\left(\frac{x}{R}\right) |f(x)|^{2} dx
	\\
	&=A_{R}(f) + 8K_{\gamma}(f),
\end{align*}
and 
\begin{align*}
	A_{R}(f) &:=4\int_{\mathbb{R}} \left\{ (\partial_{x}^{2}\varphi)\left(\frac{x}{R}\right)-2\right\} |\partial_{x}f(x)|^{2} dx 
	\\
	&\quad -\frac{2(p-1)}{p+1} \int_{\mathbb{R}} \left\{ (\partial_{x}^{2}\varphi)\left(\frac{x}{R}\right)-2\right\}  |f(x)|^{p+1} dx
	-\frac{1}{R^{2}} \int_{\mathbb{R}}(\partial_{x}^{4}\varphi)\left(\frac{x}{R}\right) |f(x)|^{2} dx.
\end{align*}

\begin{lemma}
\label{lem2.11}
For any $\theta:\mathbb{R} \to \mathbb{R}$, we have
\begin{align*}
	F_{R}(e^{i\omega t }e^{i\theta(t)}Q_{\omega,\gamma})=A_{R}(e^{i\omega t }e^{i\theta(t)}Q_{\omega,\gamma})=0
\end{align*}
for all $t\in \mathbb{R}$. 
\end{lemma}

\begin{proof}
It follows from the definition of $F_{R}$ and $A_{R}$ that 
\begin{align*}
	F_{R}(e^{i\omega t }e^{i\theta(t)}Q_{\omega,\gamma})=F_{R}(e^{i\omega t }Q_{\omega,\gamma}),
	\quad 
	A_{R}(e^{i\omega t }e^{i\theta(t)}Q_{\omega,\gamma})=A_{R}(e^{i\omega t }Q_{\omega,\gamma}).
\end{align*}
Since $e^{i\omega t }Q_{\omega,\gamma}$ is a solution to \eqref{NLS}, $Q_{\omega,\gamma}$ is real-valued, and $K_{\gamma}(e^{i\omega t }Q_{\omega,\gamma})=0$, we get
\begin{align*}
	\frac{d}{dt}J_{R}(e^{i\omega t }Q_{\omega,\gamma})=0
\end{align*}
and 
\begin{align*}
	0&= \frac{d^{2}}{dt^{2}}J_{R}(e^{i\omega t }Q_{\omega,\gamma})
	=F_{R}(e^{i\omega t }Q_{\omega,\gamma})
	\\
	&=A_{R}(e^{i\omega t }Q_{\omega,\gamma}) +8K_{\gamma}(e^{i\omega t }Q_{\omega,\gamma})
	=A_{R}(e^{i\omega t }Q_{\omega,\gamma})
\end{align*}
for all $t\in\mathbb{R}$. This concludes the proof. 
\end{proof}

%
%

\subsection{Strichartz estimates}

Though we do not use the Strichartz norms explicitly, we give the definition. The Strichartz estimates are used to obtain the solution with compact orbit in Lemma \ref{lem4.6.0}, whose proof is omitted in the present paper.  
\begin{definition}
Let $I$ be a (possibly unbounded) time interval. We define function spaces by
\begin{align*}
	S(I):=L_{t}^{\frac{2(p^{2}-1)}{p+3}} L_{x}^{p+1}(I),
	\quad
	W(I):=L_{t}^{\frac{2(p^{2}-1)}{p^{2}-3p-2}} L_{x}^{p+1}(I),
	\quad
	X(I):=L_{t}^{p-1} L_{x}^{\infty}(I).
\end{align*}
Moreover, $W'$ denotes the dual space of $W$, that is, $W'(I)=L_{t}^{\frac{2(p^{2}-1)}{p(p+3)}} L_{x}^{\frac{p+1}{p}}(I)$.
\end{definition}

\begin{lemma}[Strichartz estimates]
The following estimates are valid: 
\begin{align*}
	&\|e^{it\Delta_{\gamma}}f\|_{S(I)}+\|e^{it\Delta_{\gamma}}f\|_{X(I)} \lesssim \|f\|_{H^{1}}
	\\
	&\|\int_{t_{0}}^{t} e^{i(t-s)\Delta_{\gamma}}F(s)ds\|_{S(t_{0},t_{1})}
	+\|\int_{t_{0}}^{t} e^{i(t-s)\Delta_{\gamma}}F(s)ds\|_{X(t_{0},t_{1})}
	\lesssim \|F\|_{W'(t_{0},t_{1})},
\end{align*}
where $I$ is a (possibly unbounded) time interval and $t_{1}>t_{0}$. 
\end{lemma}

\begin{proof}
See \cite[Section 3.1]{BaVi16}. 
\end{proof}

\subsection{The linearized operators}

\subsubsection{Eigenvalues and eigenfunctions}

Here we consider the linearized operators around $Q_{\omega,\gamma}$. 
Define
\begin{align*}
	L^{-}_{\omega,\gamma} 
	&:= -\partial_{x}^{2} -\gamma \delta + \omega -Q_{\omega,\gamma}^{p-1},
	\\
	L^{+}_{\omega,\gamma} 
	&:= -\partial_{x}^{2}-\gamma \delta + \omega -pQ_{\omega,\gamma}^{p-1}.
\end{align*}
More precisely, $L^{-}_{\omega,\gamma}f:= -\partial_{x}^{2}f  + \omega f -Q_{\omega,\gamma}^{p-1}f$, $L^{+}_{\omega,\gamma}f := -\partial_{x}^{2} f+ \omega f -pQ_{\omega,\gamma}^{p-1}f$
with domains $\mathcal{D}(L^{-}_{\omega,\gamma})=\mathcal{D}(L^{+}_{\omega,\gamma})=\mathcal{D}$.
It is known by \cite{LFFKS08} (see also \cite{FuJe08}) that the linearized operator $L^{+}_{\omega,\gamma}$ on general functions has two negative eigenvalues, with one even and one odd eigenfunction. In what follows, we consider $L^{\pm}_{\omega,\gamma}$ as operators defined on $\mathcal{D}_\even$.

\begin{lemma}[\cite{FuJe08}]  \label{specinfo}
We have the following properties.
\begin{enumerate}
\item $L_{\omega,\gamma}^{-}$ and $L_{\omega,\gamma}^{+}$ are self-adjoint on the real Hilbert space $L^{2}(\mathbb{R})$. 
\item $\ker L_{\omega,\gamma}^{-}= \spn \{Q_{\omega,\gamma}\}$.
\item $\ker L_{\omega,\gamma}^{+}=\{0\}$ if $\gamma < 0$. 
\item $L_{\omega,\gamma}^{-} \geq 0$.
\item $L_{\omega,\gamma}^{-}|_{(\spn \{Q_{\omega,\gamma}\})^{\perp}} >0$.
\item $L_{\omega,\gamma}^{+}$ has one negative eigenvalue
(on $\mathcal{D}_\even$).
\end{enumerate}
\end{lemma}

Let $u$ satisfy the equation \eqref{NLS} and $u=e^{i\omega t}(Q_{\omega,\gamma}+v)$. Then $v$ satisfies 
\begin{align*}
	i\partial_{t}v +\partial_{x}^{2}v +\gamma \delta v -\omega v +P(v)
	+R(v)=0,
\end{align*}
where 
\begin{align*}
	P(v)&:= p Q_{\omega,\gamma}^{p-1}v_{1} +i Q_{\omega,\gamma}^{p-1}v_{2}
	=\frac{p+1}{2}Q_{\omega,\gamma}^{p-1}v+\frac{p-1}{2}Q_{\omega,\gamma}^{p-1}\overline{v},
	\\
	R(v)&:=|Q_{\omega,\gamma}+v|^{p-1}(Q_{\omega,\gamma}+v)
	-Q_{\omega,\gamma}^{p} - p Q_{\omega,\gamma}^{p-1}v_{1} -i Q_{\omega,\gamma}^{p-1}v_{2}
\end{align*}
for $v=v_{1}+iv_{2}$. 
In vector notation, i.e. $v = \begin{bmatrix} v_{1} \\ v_{2} \end{bmatrix}$, it can be also written as
\begin{align*}
	\partial_t v + \mathcal{L}_{\omega,\gamma}v=iR(v),
\end{align*}
where 
\begin{align*}
	\mathcal{L}_{\omega,\gamma}f= 
	\begin{bmatrix}
	0 & -L^{-}_{\omega,\gamma} 
	\\
	L^{+}_{\omega,\gamma} & 0
	\end{bmatrix}
	\begin{bmatrix}
	f_{1} \\ f_{2}
	\end{bmatrix} 
\end{align*}
for $f=f_1 + if_2$. 

\begin{proposition}
Let $\sigma(\mathcal{L}_{\omega,\gamma})$ be the spectrum of the operator $\mathcal{L}_{\omega,\gamma}$, defined on $(L_{\even}^{2}(\mathbb{R}))^{2}$ and let $\sigma_{\ess} (\mathcal{L}_{\omega,\gamma})$ be its essential spectrum.
We have
\begin{align*}
	&\sigma_{\ess} (\mathcal{L}_{\omega,\gamma})=\{i\eta :\eta \in \mathbb{R}, |\eta|\geq \omega\}, 
	\\
	&\sigma(\mathcal{L}_{\omega,\gamma}) \cap \mathbb{R} =\{-e_{\omega},0,e_{\omega}\} \text{ with } e_{\omega}>0.
\end{align*}
Moreover, $e_{\omega}$ and $-e_{\omega}$ are simple eigenvalues of $\mathcal{L}_{\omega,\gamma}$ with eigenfunctions $\mathcal{Y}_{+}$ and $\mathcal{Y}_{-}=\overline{\mathcal{Y}_{+}}$, respectively. 
\end{proposition}

\begin{proof}
Following Chang et.al. \cite[(2.19)]{CGNT07}, we consider 
\begin{align*}
	\mu_{1}:=\inf_{u \in H^1_{\even}, u\perp Q_{\omega,\gamma}} \frac{\langle L_{\omega,\gamma}^{+}u,u\rangle}{\langle (L_{\omega,\gamma}^{-})^{-1}u,u\rangle}.
\end{align*}
Then this is negative (by Lemma~\ref{specinfo} (6)),
and attained by some function $\xi$. By the Lagrange 
multiplier theorem, we get
\begin{align*}
	L_{\omega,\gamma}^{+} \xi =\beta (L_{\omega,\gamma}^{-})^{-1}\xi +\alpha Q_{\omega,\gamma}
\end{align*}
for some $\alpha, \beta \in \mathbb{R}$. Multiplying by $\xi$ and integrating, we obtain $\beta=\mu_1$. Thus $\pm e_\omega$, where $e_{\omega}:=\sqrt{-\mu_{1}}>0$, are eigenvalues of $\mathcal{L}_{\omega,\gamma}$ with eigenfunctions  
\begin{align*}
	\begin{bmatrix}
	\mp \xi \\ e_{\omega} (L_{\omega,\gamma}^{-})^{-1}\xi- e_{\omega}^{-1} \alpha Q_{\omega,\gamma}
	\end{bmatrix}.
\end{align*}
We note that $\xi$ belongs to the domain of $L_{\omega,\gamma}^{+}$,
i.e. $\xi \in \mathcal{D}_\even$. 
The uniqueness follows from \cite[Theorem 5.8]{GSS90}. 
\end{proof}


Set $\mathcal{Y}_{1}:=\re \mathcal{Y}_{+}$ and $\mathcal{Y}_{2}:=\im \mathcal{Y}_{+}$. 
Then we have
\begin{align*}
	L_{\omega,\gamma}^{+}\mathcal{Y}_{1} = e_{\omega}\mathcal{Y}_{2}, 
	\quad
	L_{\omega,\gamma}^{-}\mathcal{Y}_{2}=-e_{\omega}\mathcal{Y}_{1}.
\end{align*}

The eigenfunctions are not smooth on $\mathbb{R}$ because of the delta interaction at the origin. However, they are smooth away from the origin and decay fast at infinity.

\begin{lemma}[Smoothness of eigenfunctions]
\label{lem2.16}
We have $\mathcal{Y}_{1},\mathcal{Y}_{2} \in C^{\infty}(\mathbb{R}\setminus\{0\})$. 
\end{lemma}

\begin{proof}
Let $s \in \mathbb{N}$. Assume that $\psi \mathcal{Y}_{1},\psi \mathcal{Y}_{2} \in H^{s}(\mathbb{R})$ for an arbitrary function $\psi \in BC_0^{\infty}(\mathbb{R})$. 
By the equation $L_{\omega,\gamma}^{+}\mathcal{Y}_{1} = e_{\omega}\mathcal{Y}_{2}$, we have
\begin{align*}
	\psi\partial_{x}^{2} \mathcal{Y}_{1}=\psi (\omega \mathcal{Y}_{1} - pQ_{\omega,\gamma}^{p-1}\mathcal{Y}_{1}-e_{\omega}\mathcal{Y}_{2}) \in H^{s}(\mathbb{R}), 
\end{align*}
where we note that $\psi Q_{\omega,\gamma}^{p-1}\in BC_0^{\infty}(\mathbb{R})$. 
Since $\mathcal{Y}_{1} \in H^{1}$, we have $\psi \partial_{x} \mathcal{Y}_{1} \in L^{2}$. Thus, it holds that
\begin{align*}
	\partial_{x} (\psi \partial_{x} \mathcal{Y}_{1})
	=(\partial_{x} \psi ) \partial_{x} \mathcal{Y}_{1} + \psi \partial_{x}^{2} \mathcal{Y}_{1} \in H^{\min\{0,s\}}
\end{align*}
and thus $\psi \partial_{x} \mathcal{Y}_{1} \in H^{\min\{1,s+1\}}$. Since $\psi$ is arbitrary, we get
\begin{align*}
	\partial_{x} (\psi \partial_{x} \mathcal{Y}_{1})
	=(\partial_{x} \psi ) \partial_{x} \mathcal{Y}_{1} + \psi \partial_{x}^{2} \mathcal{Y}_{1} \in H^{\min\{1,s\}}
\end{align*}
and thus $\psi \partial_{x} \mathcal{Y}_{1} \in  H^{\min\{2,s+1\}}$. Iterating this procedure, we get $\psi \partial_{x} \mathcal{Y}_{1} \in  H^{s+1}$. 
\begin{align}
\label{eqq2.6}
	\partial_{x}(\psi \mathcal{Y}_{1})=\partial_{x} \psi \mathcal{Y}_{1} + \psi \partial_{x}\mathcal{Y}_{1} \in H^{s}
\end{align}
Thus $\psi \mathcal{Y}_{1} \in H^{s+1}$. By this, $\psi \partial_{x} \mathcal{Y}_{1} \in  H^{s+1}$, and \eqref{eqq2.6} again, we get $\psi \mathcal{Y}_{1} \in H^{s+2}$. Similarly, $\psi \mathcal{Y}_{2} \in H^{s+2}$. We start this argument with $s=0$ and thus we get $\psi \mathcal{Y}_{1},\psi \mathcal{Y}_{2} \in H^{m}$ for any integer $m>0$. This shows that $\mathcal{Y}_{1},\mathcal{Y}_{2} \in C^{\infty}(\mathbb{R}\setminus\{0\})$ since $\psi$ is arbitrary. 
\end{proof}

Since we have $L_{\omega,\gamma}^{+}\mathcal{Y}_{1} = e_{\omega}\mathcal{Y}_{2}$ and $L_{\omega,\gamma}^{-}\mathcal{Y}_{2}=-e_{\omega}\mathcal{Y}_{1}$, we get
\begin{align}
	\notag
	&\{(-\partial_{x}^{2}+\omega)^{2} + e_{\omega}^{2}\}(\psi\mathcal{Y}_{1})
	\\ \notag
	&=L_{\omega,\gamma}^{-}\{-2\partial_{x}\psi \partial_{x}\mathcal{Y}_{1}-(\partial_{x}^{2}\psi)\mathcal{Y}_{1}\}
	+e_{\omega}\{-2\partial_{x}\psi \partial_{x}\mathcal{Y}_{2} - (\partial_{x}^{2}\psi)\mathcal{Y}_{2}\}
	\\ \label{eqq2.7}
	&\quad +p(-\partial_{x}^{2}+\omega)(V\psi\mathcal{Y}_{1})
	+V(-\partial_{x}^{2}+\omega)(\psi\mathcal{Y}_{1})-pV^{2}(\psi\mathcal{Y}_{1}),
\end{align}
where $V:=Q_{\omega,\gamma}^{p-1}$, for $\psi \in BC_{0,\even}^{\infty}(\mathbb{R})$. We note that the above calculation is valid since $\psi$ vanishes near the origin and so the functions 
are in the domain $\mathcal{D}$. 

\begin{lemma}
\label{lem2.17}
For any function $\varphi \in C_{c,0,\even}^{\infty}(\mathbb{R})$ and $k,s\in \mathbb{N}$, there exists $C_{s,k,\varphi}>0$ such that 
\begin{align*}
	\|\varphi(x/R)\mathcal{Y}_{1}\|_{H^{s}}<\frac{C_{s,k,\varphi}}{R^{k}}
\end{align*}
for any $R\geq1$. A similar estimate for $\mathcal{Y}_{2}$ also holds.
\end{lemma}

\begin{proof}
For $j=1,2$, we say that $P_{s,k}^{j}$ is true if there exists a constant $C=C_{s,k,\varphi}>0$ such that
\begin{align*}
	\|\varphi(x/R)\mathcal{Y}_{j}\|_{H^{s}}<\frac{C}{R^{k}}.
\end{align*}
holds. We note that $P_{\sigma,k}^{j}$ is true for $\sigma\leq s$ if $P_{s,k}^{j}$ is valid. 
Now, $P_{s,0}^{1}$ and $P_{s,0}^{2}$ are true for all $s \in \mathbb{N}$ since we have $\varphi(x/R)\mathcal{Y}_{j} \in H^{s}(\mathbb{R})$ by the proof of Lemma \ref{lem2.16} for $j=1,2$. Let $k\geq 0$. We will show that $P_{s+1,k+1}^{1}$ is valid if $P_{s,k}^{1}$ is true for all $s\in \mathbb{N}$. Assume that $P_{s,k}^{1}$ is true for all $s\in \mathbb{N}$. 
For given $\varphi$, take a function $\widetilde{\varphi} \in C_{c,0,\even}^{\infty}(\mathbb{R})$ satisfying $\widetilde{\varphi}(x)=1$ on the support of $\varphi$. 
Then we have
\begin{align*}
	\|\varphi(x/R)\mathcal{Y}_{1}\|_{H^{s+1}}
	&=\|\varphi(x/R)\widetilde{\varphi}(x/R) \mathcal{Y}_{1}\|_{H^{s+1}}\\
	&\approx \|\{ (-\partial_{x}^{2}+\omega)^{2}+e_{\omega}^{2} \}\varphi(x/R)\widetilde{\varphi}(x/R) \mathcal{Y}_{1}\|_{H^{s-3}}\\
	&\leq \| [\{ (-\partial_{x}^{2}+\omega)^{2}+e_{\omega}^{2} \},\varphi(x/R)]\widetilde{\varphi}(x/R) \mathcal{Y}_{1}\|_{H^{s-3}} \\
	&\quad + \|\varphi(x/R) \{ (-\partial_{x}^{2}+\omega)^{2}+e_{\omega}^{2} \}\widetilde{\varphi}(x/R) \mathcal{Y}_{1}\|_{H^{s-3}}.
\end{align*}
Now we know that $\|[\{ (-\partial_{x}^{2}+\omega)^{2}+e_{\omega}^{2} \},\varphi(x/R)]\|_{H^{s}\to H^{s-3}} \lesssim R^{-1}$. 

Thus we get
\begin{align*}
	\|\varphi(x/R)\mathcal{Y}_{1}\|_{H^{s+1}}
	&\leq \frac{C}{R}\| \widetilde{\varphi}(x/R) \mathcal{Y}_{1}\|_{H^{s}}
	+ \|\varphi(x/R) \{ (-\partial_{x}^{2}+\omega)^{2}+e_{\omega}^{2} \}\widetilde{\varphi}(x/R) \mathcal{Y}_{1}\|_{H^{s-3}}.
\end{align*}
The second term can be estimated by using \eqref{eqq2.7} as $\psi=\widetilde{\varphi}(\cdot/R)$. Indeed, by $\varphi\widetilde{\varphi}=\varphi$, $\varphi\partial_x^\alpha\widetilde{\varphi}=0$ for $\alpha \geq 1$, and \eqref{eqq2.7}, we obtain 
\begin{align*}
	&\|\varphi(x/R) \{(-\partial_{x}^{2}+\omega)^{2} + e_{\omega}^{2}\}(\widetilde{\varphi}(x/R)\mathcal{Y}_{1})\|_{H^{s-3}}
	\\ 
	&\lesssim \|\varphi(x/R) V \mathcal{Y}_{1}\|_{H^{s-3}}
	+\|\varphi(x/R) \partial_xV \partial_x\mathcal{Y}_{1}\|_{H^{s-3}}
	\\
	&\quad +\|\varphi(x/R) \partial_x^2V\mathcal{Y}_{1}\|_{H^{s-3}}
	+\|\varphi(x/R) V\partial_x^2\mathcal{Y}_{1}\|_{H^{s-3}}
	+\|\varphi(x/R) V^2 \mathcal{Y}_{1}\|_{H^{s-3}}
	\\
	&\lesssim \sum_{\alpha+\beta+\gamma\leq s-1}\|\partial_x^\alpha \varphi(x/R) \partial_x^\beta V \partial_x^\gamma \mathcal{Y}_{1}\|_{L^2}.
\end{align*}
Since $V=Q_{\omega,\gamma}^{p-1}$ decays exponentially, we have $\partial_x^\alpha \varphi(x/R) \partial_x^\beta V \lesssim R^{-1}\partial_x^\alpha \varphi(x/R)$. 
Thus, it holds that 
\begin{align*}
	\sum_{\alpha+\beta+\gamma\leq s-1}\|\partial_x^\alpha \varphi(x/R) \partial_x^\beta V \partial_x^\gamma \mathcal{Y}_{1}\|_{L^2}
	&\lesssim R^{-1} \sum_{\alpha+\gamma\leq s-1}\|\partial_x^\alpha \varphi(x/R) \partial_x^\gamma \mathcal{Y}_{1}\|_{L^2}
	\\
	&\lesssim R^{-1} \|\varphi(x/R) \mathcal{Y}_{1}\|_{H^{s-1}}.
\end{align*}
This means that 
\begin{align*}
	\|\varphi(x/R) \{ (-\partial_{x}^{2}+\omega)^{2}+e_{\omega}^{2} \}\widetilde{\varphi}(x/R) \mathcal{Y}_{1}\|_{H^{s-3}}
\lesssim R^{-1} \|\varphi(x/R) \mathcal{Y}_{1}\|_{H^{s-1}}.
\end{align*}
Therefore, since $P_{s,k}^{1}$ is true, we get 
\begin{align*}
	\|\varphi(x/R)\mathcal{Y}_{1}\|_{H^{s+1}}
	\leq  \frac{C}{R}\| \widetilde{\varphi}(x/R) \mathcal{Y}_{1}\|_{H^{s}}
	+\frac{C}{R} \|\varphi(x/R) \mathcal{Y}_{1}\|_{H^{s-1}}
	\leq  \frac{C_{s,k,\varphi}}{R^{k+1}}.
\end{align*}
This means that $P_{s+1,k+1}^{1}$ is true. In the same way, we also can show that  $P_{s+1,k+1}^{2}$ is true. 
\end{proof}

\begin{lemma}[Decay of eigenfunctions]
\label{lem2.18}
We have $\mathcal{Y}_{1}, \mathcal{Y}_{2} \in \widetilde{\mathcal{S}}(\mathbb{R})$. 
\end{lemma}

\begin{proof}
We only consider $\mathcal{Y}_{1}$. 
Let $\psi \in C_{0,\even}^{\infty}$ satisfy
\begin{align*}
	\psi(x)=
	\begin{cases}
	0, & |x|\leq 1,
	\\
	1, & |x|\geq 2. 
	\end{cases}
\end{align*}
Let $s,l \in \mathbb{N}$. 
Now, $x^{l} \psi \mathcal{Y}_{1}$ is even or odd. Thus we have
\begin{align*}
	\|x^{l} \psi \mathcal{Y}_{1}\|_{H^{s}(\mathbb{R})}^{2}
	=2\|x^{l} \psi \mathcal{Y}_{1}\|_{H^{s}(1,\infty)}^{2}
	=2\sum_{n=0}^{\infty}\|x^{l} \psi \mathcal{Y}_{1}\|_{H^{s}([2^{n},2^{n+1}])}^{2}.
\end{align*}
Let $\varphi \in C_{c,0,\even}^{\infty}(\mathbb{R})$ satisfy 
\begin{align*}
	\varphi(x)=\begin{cases}
	1 & \text{ if } 1\leq |x|\leq 2,
	\\
	0 & \text{ if } |x|>4,|x|<1/2.
	\end{cases}
\end{align*}
We have
\begin{align*}
	\|x^{l} \psi \mathcal{Y}_{1}\|_{H^{s}([2^{n},2^{n+1}])}
	= \|x^{l} \psi \varphi(\cdot/2^{n}) \mathcal{Y}_{1}\|_{H^{s}([2^{n},2^{n+1}])}
	\lesssim 2^{nl} \|\varphi(\cdot/2^{n}) \mathcal{Y}_{1}\|_{H^{s}}.
\end{align*}
By Lemma \ref{lem2.17} with $k=l+1$, we get
\begin{align*}
	2^{nl} \|\varphi(\cdot/2^{n}) \mathcal{Y}_{1}\|_{H^{s}([2^{n},2^{n+1}])}
	\lesssim 2^{nl} \frac{C_{s,l,\varphi}}{2^{n(l+1)}} \lesssim 2^{-n}.
\end{align*}
Therefore, it holds that
\begin{align*}
	\sum_{n=0}^{\infty}\|x^{l} \psi \mathcal{Y}_{1}\|_{H^{s}([2^{n},2^{n+1}])}^{2}
	\leq C\sum_{n=0}^{\infty}2^{-2n} <C
\end{align*}
and thus $\|x^{l} \psi \mathcal{Y}_{1}\|_{H^{s}(\mathbb{R})}<\infty$. This and the Sobolev embedding imply $\psi \mathcal{Y}_{1} \in \mathcal{S}(\mathbb{R})$. 
Thus we have $\mathcal{Y}_1 \in \widetilde{\mathcal{S}}(\mathbb{R})$. 
A similar argument gives us $\mathcal{Y}_2 \in \widetilde{\mathcal{S}}(\mathbb{R})$. 
\end{proof}

\subsubsection{Estimates of the ground state and eigenfunctions}

In this section, we give estimates for the ground state and eigenfunctions. They will be used to construct special solutions in Section \ref{sec3}. 

\begin{lemma}
\label{lem2.19}
Let $\alpha=0,1$. 
We have
\begin{align*}
	\|Q_{\omega,\gamma}^{-1}\partial_{x}^{\alpha} Q_{\omega,\gamma}\|_{L^{\infty}} < \infty. 
\end{align*}
\end{lemma}

\begin{proof}
We have $Q_{\omega,\gamma}(x)>0$ for all $x\in \mathbb{R}$. The case of $\alpha=0$ is trivial. 
The derivative $\partial_{x} Q_{\omega,\gamma}$ is well-defined except for the origin. 
Since we have $\|Q_{\omega,0}^{-1}\partial_{x}Q_{\omega,0}\|_{L^{\infty}}< \infty$ by \cite[Corollary 3.8]{CFR20} and $Q_{\omega,\gamma}(x)=Q_{\omega,0}(x+\xi)$ on $(0,\infty)$ for some $\xi=\xi(\omega,\gamma)$ (see Lemma \ref{lemA.1}), we get
\begin{align*}
	\|Q_{\omega,\gamma}^{-1}\partial_{x}  Q_{\omega,\gamma}\|_{L^{\infty}}
	=\|Q_{\omega,0}^{-1}\partial_{x} Q_{\omega,0}\|_{L^{\infty}(\xi,\infty)}
	<\infty.
\end{align*}
This completes the proof.
\end{proof}

\begin{lemma}
\label{lem2.20}
Let $f \in \mathcal{S}(\mathbb{R})$ and $\lambda \in \mathbb{R}$. If $f$ satisfies $(1-\partial_{x}^{2}+i\lambda)f=G$ with $|G(x)| \lesssim e^{-a|x|}$ for $0<a \neq \re \sqrt{1+i\lambda}$, then we have
\begin{align*}
	|f(x)| \lesssim e^{-\min\{a, \re \sqrt{1+i\lambda}\}|x|}.
\end{align*}
\end{lemma}

\begin{proof}
See \cite[Lemma 3.7]{CFR20}. 
\end{proof}

\begin{lemma}
\label{lem2.21}
For $\alpha=0,1$, the following estimates hold: 
\begin{align*}
	\|Q_{\omega,\gamma}^{-1}e^{\eta|x|}\partial_{x}^{\alpha}\mathcal{Y}_{\pm}\|_{L^{\infty}}<\infty
\end{align*} 
for some $0<\eta \ll 1$. 
\end{lemma}

\begin{proof}
It is enough to show the statement for $\mathcal{Y}_{1}$. 
Let $\psi \in C_{0,\even}^\infty(\mathbb{R})$ satisfy
\begin{align*}
	\psi(x) =
	\begin{cases}
	0, & |x|\leq 1,
	\\
	1, & |x|\geq 2.
	\end{cases}
\end{align*}
Now we have
\begin{align*}
	\|Q_{\omega,\gamma}^{-1}e^{\eta|x|}\partial_{x}^{\alpha}\mathcal{Y}_{1}\|_{L^{\infty}}
	\leq \|Q_{\omega,\gamma}^{-1}e^{\eta|x|}\partial_{x}^{\alpha}(\psi\mathcal{Y}_{1})\|_{L^{\infty}}
	+\|Q_{\omega,\gamma}^{-1}e^{\eta|x|}\partial_{x}^{\alpha}\{(1-\psi)\mathcal{Y}_{1}\}\|_{L^{\infty}}.
\end{align*}
Since $1-\psi \in C_c^\infty(\mathbb{R})$ and $\mathcal{Y}_1 \in \mathcal{D} \subset W^{1,\infty}$, we have
\begin{align*}
	\|Q_{\omega,\gamma}^{-1}e^{\eta|x|}\partial_{x}^{\alpha}\{(1-\psi)\mathcal{Y}_{1}\}\|_{L^{\infty}}
	\lesssim \|\partial_{x}^{\alpha}\{(1-\psi)\mathcal{Y}_{1}\}\|_{L^{\infty}}
	\lesssim \|\mathcal{Y}_{1}\|_{W^{1,\infty}}<\infty.
\end{align*}
Thus, it is sufficient to prove that the first term is finite. As seen in Lemma \ref{lem2.18}, $\psi \mathcal{Y}_{1}$ is of Schwartz class. Set $g:=(-\partial_{x}^{2}+\omega - ie_{\omega})(\psi\mathcal{Y}_{1}) \in \mathcal{S}(\mathbb{R})$. Then we have
\begin{align*}
	\begin{cases}
	(-\partial_{x}^{2}+\omega + ie_{\omega})g=G(\mathcal{Y}_{1},\mathcal{Y}_{2}),
	\\
	(-\partial_{x}^{2}+\omega - ie_{\omega})(\psi\mathcal{Y}_{1})=g,
	\end{cases}
\end{align*}
where
\begin{align*}
	G(\mathcal{Y}_{1},\mathcal{Y}_{2})
	&:= \{(-\partial_{x}^{2})^{2}+\omega)^{2} +e_{\omega}^{2}\}(\psi\mathcal{Y}_{1})
	\\
	&=L_{\omega,\gamma}^{-}\{-2\partial_{x}\psi \partial_{x}\mathcal{Y}_{1}-(\partial_{x}^{2}\psi)\mathcal{Y}_{1}\}
	+e_{\omega}\{-2\partial_{x}\psi \partial_{x}\mathcal{Y}_{2} - (\partial_{x}^{2}\psi)\mathcal{Y}_{2}\}
	\\
	&\quad +p(-\partial_{x}^{2}+\omega)(V\psi\mathcal{Y}_{1})
	+V(-\partial_{x}^{2}+\omega)(\psi\mathcal{Y}_{1})-pV^{2}(\psi\mathcal{Y}_{1})
\end{align*}
and $V=Q_{\omega,\gamma}^{p-1}$. Since $\partial_{x}^{\beta}\psi=0$ on $(-1,1)\cup (\mathbb{R}\setminus (-2,2))$ for $\beta\geq 1$ and $\partial_{x}^{\alpha}\mathcal{Y}_{1} \in L^{\infty}(\mathbb{R})$, we have
\begin{align*}
	|G(\mathcal{Y}_{1},\mathcal{Y}_{2})| \lesssim V \lesssim e^{-(p-1)\omega^{\frac{1}{2}}|x|}.
\end{align*}
Lemma \ref{lem2.20} yields that 
\begin{align*}
	|g(x)| \lesssim e^{-\min\{(p-1)\omega^{1/2}, \re \sqrt{\omega + ie_{\omega}}\}|x|}. 
\end{align*}
Set $a:=\min\{(p-1)\omega^{1/2}, \re \sqrt{\omega + ie_{\omega}}\}$. 
By Lemma \ref{lem2.20} again, we get
\begin{align*}
	|\psi\mathcal{Y}_{1}| \lesssim e^{-\min\{a,\sqrt{\omega-ie_{\omega}}\}|x|}.
\end{align*}
Since $(p-1)\omega^{1/2},\re \sqrt{\omega + ie_{\omega}}, \re \sqrt{\omega - ie_{\omega}} \geq \omega^{1/2}$ and $Q_{\omega,\gamma}\lesssim e^{-\omega^{1/2}|x|}$, we obtain
\begin{align*}
	\|Q_{\omega,\gamma}^{-1}e^{\eta|x|}(\psi\mathcal{Y}_{1})\|_{L^{\infty}} < \infty,
\end{align*}
where $\eta>0$ is smaller than $\min\{(p-1)\omega^{1/2},\re \sqrt{\omega + ie_{\omega}}, \re \sqrt{\omega - ie_{\omega}}\} - \omega^{1/2}$. For the estimate of the derivative $\partial_{x}\mathcal{Y}_{1}$, we argue similarly, using
\begin{align*}
	\begin{cases}
	(-\partial_{x}^{2}+\omega + ie_{\omega}) \partial_{x}g=\partial_{x}G(\mathcal{Y}_{1},\mathcal{Y}_{2}),
	\\
	(-\partial_{x}^{2}+\omega - ie_{\omega})\partial_{x}(\psi\mathcal{Y}_{1})=\partial_{x}g.
	\end{cases}
\end{align*}
This completes the proof. 
\end{proof}

\begin{lemma}
\label{lem2.22}
Let $\lambda \in \mathbb{R} \setminus \sigma(\mathcal{L}_{\omega,\gamma})$ and let $f \in L^2_\even \cap \widetilde{\mathcal{S}}$. 
If $f$ satisfies 
\begin{align*}
	\|Q_{\omega,\gamma}^{-1}e^{\eta|x|}\partial_{x}^{\alpha}f\|_{L^{\infty}} < \infty
\end{align*}
for some $0<\eta < \re \sqrt{\omega + i \lambda}$ and for $\alpha=0,1$, then 
$f$ also satisfies 
\begin{align*}
	\|Q_{\omega,\gamma}^{-1}e^{\eta|x|}\partial_{x}^{\beta}\{ (\mathcal{L}_{\omega,\gamma}-\lambda)^{-1}f\}\|_{L^{\infty}} < \infty
\end{align*}
for $\beta=0,1$. 
\end{lemma}

\begin{proof}
We have 
\begin{align*}
	&\|Q_{\omega,\gamma}^{-1}e^{\eta|x|}\partial_{x}^{\beta}\{ (\mathcal{L}_{\omega,\gamma}-\lambda)^{-1}f\}\|_{L^{\infty}} 
	\\
	&\leq \|Q_{\omega,\gamma}^{-1}e^{\eta|x|}\partial_{x}^{\beta}\{\psi (\mathcal{L}_{\omega,\gamma}-\lambda)^{-1}f\}\|_{L^{\infty}} 
	\\
	&\quad +\|Q_{\omega,\gamma}^{-1}e^{\eta|x|}\partial_{x}^{\beta}\{ (1-\psi)(\mathcal{L}_{\omega,\gamma}-\lambda)^{-1}f\}\|_{L^{\infty}}.
\end{align*}
Since we have $1-\psi \in C_c^\infty$ and $(\mathcal{L}_{\omega,\gamma}-\lambda)^{-1}f \in \mathcal{D} \subset W^{1,\infty}$, the second term in the right hand side is finite. It is enough to estimate the first term. 

Let $g=(\mathcal{L}_{\omega,\gamma}-\lambda)^{-1}f$. Then we have $g \in \mathcal{D}\subset W^{1,\infty}$ and $(\mathcal{L}_{\omega,\gamma}-\lambda)g=f$, that is,
\begin{align*}
	\begin{cases}
	-(-\partial_{x}^{2} + \omega -Q_{\omega,\gamma}^{p-1})g_{2} -\lambda g_{1}=f_{1},
	\\
	(-\partial_{x}^{2} + \omega -pQ_{\omega,\gamma}^{p-1})g_{1} -\lambda g_{2}=f_{2}.
	\end{cases}
\end{align*}
From this, we obtain
\begin{align*}
	(-\partial_{x}^{2} + \omega -Q_{\omega,\gamma}^{p-1})(\psi g_{2})
	&= -2\partial_{x}\psi \partial_{x}g_{2} -\partial_{x}^{2}\psi g_{2} 
	+\psi(-\partial_{x}^{2} + \omega -Q_{\omega,\gamma}^{p-1}) g_{2}
	\\
	&= 2\partial_{x}\psi \partial_{x}g_{2} + \partial_{x}^{2}\psi g_{2} 
	-\lambda \psi g_{1} +\psi f_{1}.
\end{align*}
Now, in the same way as in Lemmas \ref{lem2.16}--\ref{lem2.18}, we have $\psi g\in \mathcal{S}(\mathbb{R})$. 
Thus we obtain
\begin{align*}
	&(-\partial_{x}^{2} + \omega -pQ_{\omega,\gamma}^{p-1})(-\partial_{x}^{2} + \omega -Q_{\omega,\gamma}^{p-1})(\psi g_{2})
	\\
	&=(-\partial_{x}^{2} + \omega -pQ_{\omega,\gamma}^{p-1})( -2\partial_{x}\psi \partial_{x}g_{2} - \partial_{x}^{2}\psi g_{2})
	\\
	&\quad -\lambda ( -2\partial_{x}\psi \partial_{x}g_{1} -\partial_{x}^{2}\psi g_{1} +\lambda \psi g_{2} + \psi f_{2})
	\\
	&\quad + (-\partial_{x}^{2} + \omega -pQ_{\omega,\gamma})(\psi f_{1}).
\end{align*}
This implies that
\begin{align*}
	\{(-\partial_{x}^{2} + \omega)^{2}+\lambda^{2}\} (\psi g_{2})
	&=pQ_{\omega,\gamma}^{p-1}(-\partial_{x}^{2} + \omega)(\psi g_{2})+ (-\partial_{x}^{2} + \omega) (Q_{\omega,\gamma}^{p-1}\psi g_{2})
	\\
	&\quad +(-\partial_{x}^{2} + \omega -pQ_{\omega,\gamma}^{p-1})( -2\partial_{x}\psi \partial_{x}g_{2} - \partial_{x}^{2}\psi g_{2})
	\\
	&\quad -\lambda ( -2\partial_{x}\psi \partial_{x}g_{1} -\partial_{x}^{2}\psi g_{1} + \psi f_{2})
	\\
	&\quad + (-\partial_{x}^{2} + \omega -pQ_{\omega,\gamma})(\psi f_{1})
	\\&=:\widetilde{G}(f,g).
\end{align*}
We set $F:=(-\partial_{x}^{2} + \omega+i\lambda) (\psi g_{2})$. 
Then we have
\begin{align*}
	\begin{cases}
	(-\partial_{x}^{2} + \omega-i\lambda) F=\widetilde{G}(f,g),
	\\
	(-\partial_{x}^{2} + \omega+i\lambda) (\psi g_{2})=F,
	\end{cases}
\end{align*}
where we note that $F$ and $\psi g_{2}$ are of Schwartz class. 
Since we have $|\widetilde{G}(f,g)| \lesssim e^{-\omega^{1/2}|x|}$ 
from the assumption, we get the conclusion by the similar argument to above. 
We also have
\begin{align*}
	\begin{cases}
	(-\partial_{x}^{2} + \omega-i\lambda) \partial_x^\beta F=\partial_x^\beta \widetilde{G}(f,g),
	\\
	(-\partial_{x}^{2} + \omega+i\lambda) \partial_x^\beta (\psi g_{2})=\partial_x^\beta F,
	\end{cases}
\end{align*}
and thus the estimate of the derivative is also obtained. In particular, we get
\begin{align*}
	\|Q_{\omega,\gamma}^{-1}e^{\eta|x|}\partial_{x}^{\alpha}\{\psi (\mathcal{L}_{\omega,\gamma}-\lambda)^{-1}f\}\|_{L^{\infty}} < \infty
\end{align*}
for any $\alpha \in \mathbb{N}$. 
\end{proof}

\subsubsection{Coercivity}
In this section, we show the coercivity of the linearized operators on sets with orthogonality conditions.

We define bilinear operators $B_{\omega,\gamma}^{\pm}$ by 
\begin{align*}
	B_{\omega,\gamma}^{+}(f,g)
	:= \langle L_{\omega,\gamma}^{+} f , g \rangle,
	\quad 
	B_{\omega,\gamma}^{-}(f,g)
	:= \langle L_{\omega,\gamma}^{-} f , g \rangle
\end{align*}
for $f,g \in H^{1}(\mathbb{R};\mathbb{R})$. More precisely, for example, the first one is defined by 
\begin{align*}
	B_{\omega,\gamma}^{+}(f,g)
	=\int_{\mathbb{R}} \partial_{x}f \partial_{x}g  dx
	+\omega \int_{\mathbb{R}} f  g dx
	-\gamma f(0) g(0) 
	-p\int_{\mathbb{R}}Q_{\omega,\gamma}^{p-1} f  g  dx.
\end{align*}
We also define 
\begin{align*}
	B_{\omega,\gamma}(f,g):= \frac{1}{2}B_{\omega,\gamma}^{+}(f_{1},g_{1})
	+\frac{1}{2}B_{\omega,\gamma}^{-}(f_{2},g_{2})
\end{align*}
for $f=f_{1}+if_{2},g=g_{1}+ig_{2} \in H^{1}(\mathbb{R})$ and 
\begin{align*}
	\Phi(f):=B_{\omega,\gamma}(f,f)
\end{align*}
for $f \in H^{1}(\mathbb{R})$.

We consider the following orthogonality conditions:
\begin{align}
	\label{eq2.8}
	&\int_{\mathbb{R}} f_{2}(x)Q_{\omega,\gamma}(x)dx=
	\int_{\mathbb{R}} f_{1}(x)Q_{\omega,\gamma}^{p}(x)dx=0,
	\\
	\label{eq2.9}
	&\int_{\mathbb{R}} f_{2}(x)Q_{\omega,\gamma}(x)dx=
	\int_{\mathbb{R}} f_{2}(x)\mathcal{Y}_{1}(x)dx=
	\int_{\mathbb{R}} f_{1}(x)\mathcal{Y}_{2}(x)dx=0,
\end{align}
where $f_{1}=\re f$ and $f_{2}=\im f$. 
The sets in $H_{\even}^{1}(\mathbb{R})$ satisfying \eqref{eq2.8} and \eqref{eq2.9} are denoted by $G^{\perp}$ and $\widetilde{G}^{\perp}$, respectively. That is, we set 
\begin{align*}
	&G^{\perp}:=\{f \in H_{\even}^{1}(\mathbb{R}): ( iQ, f )_{L^{2}} =( Q^{p} , f )_{L^{2}} =0\},
	\\
	&\widetilde{G}^{\perp}:=\{f \in H_{\even}^{1}(\mathbb{R}): ( iQ, f )_{L^{2}}=( \mathcal{Y}_{1}, f_{2} )_{L^{2}} = ( \mathcal{Y}_{2} , f_{1} )_{L^{2}} =0\}.
\end{align*}

Then we have the following coercivity of $\Phi$ on $G^{\perp}$ and $\widetilde{G}^{\perp}$. 

\begin{lemma}[Coercivity]
\label{coercivity}
There exists a positive constant $c$ such that 
\begin{align*}
	\Phi(f) \geq c\|f\|_{H^{1}}^{2}
\end{align*}
for all $f \in G^{\perp} \cup \widetilde{G}^{\perp}$. 
\end{lemma}

\begin{proof}
Let $f=f_{1}+i f_{2} \in G^{\perp}$. We show that 
$\langle L_{\omega,\gamma}^{+}f_{1} , f_{1} \rangle \geq c\|f_{1}\|_{H^{1}}^{2}$ 
for some constant $c>0$. For given small $\varepsilon>0$, we set 
\begin{align}
\label{eq2.10}
	u_{\varepsilon}:=\lambda_{\varepsilon} (Q_{\omega,\gamma}+\varepsilon f_{1})
\end{align} with $\lambda_{\varepsilon} \in \mathbb{R}$ chosen such that 
\begin{align}
\label{eq2.11}
	0=I_{\omega,\gamma}(u_{\varepsilon})=\lambda_{\varepsilon}^{2}\left\{\|Q_{\omega,\gamma}+\varepsilon f_{1}\|_{H_{\omega,\gamma}^{1}}^{2} -\lambda_{\varepsilon}^{p-1}\|Q_{\omega,\gamma}+\varepsilon f_{1}\|_{L^{p+1}}^{p+1}\right\}.
\end{align}
Since it follows from the equation \eqref{eleq} and the orthogonality assumption \eqref{eq2.8} that 
\begin{align*}
	(Q_{\omega,\gamma},f_{1})_{H_{\omega,\gamma}^{1}}
	=(Q_{\omega,\gamma}^{p},f_{1})_{L^{2}}
	=0,
\end{align*}
we get
\begin{align*}
	\|Q_{\omega,\gamma}+\varepsilon f_{1}\|_{H_{\omega,\gamma}^{1}}^{2}
	=\|Q_{\omega,\gamma}\|_{H_{\omega,\gamma}^{1}}^{2}+\varepsilon^{2}\|f_{1}\|_{H_{\omega,\gamma}^{1}}^{2}.
\end{align*}
Moreover, we also have
\begin{align*}
	\|Q_{\omega,\gamma}+\varepsilon f_{1}\|_{L^{p+1}}^{p+1}
	&=\|Q_{\omega,\gamma}\|_{L^{p+1}}^{p+1}
	+\varepsilon (p+1)(Q^{p},f_{1})_{L^{2}} + O(\varepsilon^{2})
	\\
	&=\|Q_{\omega,\gamma}\|_{L^{p+1}}^{p+1}
	+ O(\varepsilon^{2}).
\end{align*}
Thus \eqref{eq2.11} and these equalities yield that
\begin{align*}
	\lambda_{\varepsilon}^{p-1}
	=\frac{\|Q_{\omega,\gamma}+\varepsilon f_{1}\|_{H_{\omega,\gamma}^{1}}^{2}}{\|Q_{\omega,\gamma}+\varepsilon f_{1}\|_{L^{p+1}}^{p+1}}
	=\frac{\|Q_{\omega,\gamma}\|_{H_{\omega,\gamma}^{1}}^{2}+O(\varepsilon^{2})}{\|Q_{\omega,\gamma}\|_{L^{p+1}}^{p+1}
	+ O(\varepsilon^{2})}
	=1+O(\varepsilon^{2}),
\end{align*}
where we have used $I_{\omega,\gamma}(Q_{\omega,\gamma})=0$, that is, $\|Q_{\omega,\gamma}\|_{H_{\omega,\gamma}^{1}}^{2}=\|Q_{\omega,\gamma}\|_{L^{p+1}}^{p+1}$. This means that
\begin{align*}
	\lambda_{\varepsilon}=1+O(\varepsilon^{2}).
\end{align*}
Substituting this into \eqref{eq2.10}, we get
\begin{align*}
	u_{\varepsilon}=Q_{\omega,\gamma} + \varepsilon f_{1} +O(\varepsilon^{2}).
\end{align*}
Since $u_{\varepsilon}$ satisfies $I_{\omega,\gamma}(u_{\varepsilon})=0$ by \eqref{eq2.11} and $S_{\omega,\gamma}(Q_{\omega,\gamma})=\min\{S(f): f \in H_{\even}^{1}(\mathbb{R})\setminus\{0\}, I_{\omega,\gamma}(f)=0\}$,  we get
\begin{align*}
	0 \leq S_{\omega,\gamma}(u_{\varepsilon}) -S_{\omega,\gamma}(Q_{\omega,\gamma})
	&=\langle S_{\omega,\gamma}'(Q_{\omega,\gamma}),\varepsilon f_{1} +O(\varepsilon^{2})\rangle +\frac{1}{2} \varepsilon^{2} \langle L_{\omega,\gamma}^{+}f_{1}, f_{1}\rangle + O(\varepsilon^{3})
	\\
	&=\frac{1}{2} \varepsilon^{2} \langle L_{\omega,\gamma}^{+}f_{1}, f_{1}\rangle + O(\varepsilon^{3})
\end{align*}
by the Taylor expansion and the equation \eqref{eleq}.  By taking $\varepsilon>0$ sufficiently small, this means that
\begin{align*}
	\langle L_{\omega,\gamma}^{+}f_{1}, f_{1}\rangle \geq 0. 
\end{align*}
In the similar way to Step 2 in the proof of Lemma 3.5 (intercritical case) in \cite{CFR20}, we get the coercivity on $G^{\perp}$. We can also get the coercivity on $\widetilde{G}^{\perp}$ in the same way as in \cite{DuRo10}. We omit the details of the proof. 
\end{proof}

\subsubsection{Estimates on the linearized equation}

In this section, we give estimates of $P$ and $R$.

\begin{lemma}
\label{lem2.23}
Let $I$ be a bounded time interval. 
For $f \in L_{t}^{\infty}H^{1}$, we have
\begin{align*}
	\|P(f)\|_{L_{t}^{1}H^{1}(I)} \lesssim |I|\|f\|_{L_{t}^{\infty}H^{1}(I)}.
\end{align*}
\end{lemma}

\begin{proof}
Since we have $P(v)=\frac{p+1}{2}Q_{\omega,\gamma}^{p-1}v+\frac{p-1}{2}Q_{\omega,\gamma}^{p-1}\overline{v}$,
it follows from the H\"{o}lder inequality that
\begin{align*}
	\|P(f)\|_{L_{t}^{1}L^{2}(I)} \lesssim  \|Q_{\omega,\gamma}^{p-1}f\|_{L^{1}L^{2}(I)}
	\lesssim \|f\|_{L_{t}^{1}L^{2}(I)}
	\lesssim  |I| \|f\|_{L_{t}^{\infty}L^{2}(I)}.
\end{align*}
Moreover, we also have
\begin{align*}
	\|\partial_{x} P(f)\|_{L_{t}^{1}L^{2}(I)} 
	&\lesssim  \|Q_{\omega,\gamma}^{p-2}\partial_{x}Q_{\omega,\gamma}f\|_{L_{t}^{1}L^{2}(I)}
	+  \|Q_{\omega,\gamma}^{p-1}\partial_{x}f\|_{L_{t}^{1}L^{2}(I)} 
	\\
	&\lesssim |I|\|f\|_{L_{t}^{\infty}L^{2}(I)}+|I|\|\partial_{x}f\|_{L_{t}^{\infty}L^{2}(I)}
\end{align*}
since $\partial_{x}Q_{\omega,\gamma}\in L^{\infty}$. Thus we obtain the desired estimate. 
\end{proof}

\begin{lemma}
\label{lem2.24}
Let $I$ be a time interval with $|I|\leq 1$. 
For $f,g \in L_{t}^{\infty}H^{1}$, we have
\begin{align*}
	&\|R(f)-R(g)\|_{L_{t}^{1}H^{1}(I)} 
	\\
	&\lesssim \|f-g\|_{L^{\infty}H^{1}}(\|f\|_{L^{\infty}H^{1}} +\|g\|_{L^{\infty}H^{1}} + \|f\|_{L^{\infty}H^{1}}^{p-1}+ \|g\|_{L^{\infty}H^{1}}^{p-1}).
\end{align*}
\end{lemma}

\begin{proof}
Since $R(f)=Q_{\omega,\gamma}^{p}N(Q_{\omega,\gamma}^{-1}f)$, where $N(z)=|1+z|^{p-1}(1+z)-1-\frac{p+1}{2}z-\frac{p-1}{2}\overline{z}$, we get
\begin{align*}
	&R(f)-R(g)=Q_{\omega,\gamma}^{p} ( N(Q_{\omega,\gamma}^{-1}f)-N(Q_{\omega,\gamma}^{-1}g))
	\\
	&=Q_{\omega,\gamma}^{p-1} \int_{0}^{1} N_{z}(Q_{\omega,\gamma}^{-1}(sf+(1-s)g))(f-g) + N_{\overline{z}}(Q_{\omega,\gamma}^{-1}(sf+(1-s)g))(\overline{f-g}) ds.
\end{align*}
Now it holds that
\begin{align}
\label{eq2.12}
	| N_{z}(z_{1})- N_{z}(z_{2})| + |N_{\overline{z}}(z_{1})-N_{\overline{z}}(z_{2})|
	\lesssim |z_{1}-z_{2}|(1+|z_{1}|^{p-2}+|z_{2}|^{p-2}).
\end{align}
Thus we have
\begin{align}
	\notag
	&|R(f)-R(g)|
	\\ \notag
	&\leq Q_{\omega,\gamma}^{p-1}|f-g|\int_{0}^{1} |N_{z}(Q_{\omega,\gamma}^{-1}(sf+(1-s)g))| + |N_{\overline{z}}(Q_{\omega,\gamma}^{-1}(sf+(1-s)g))| ds
	\\ \notag
	&\lesssim Q_{\omega,\gamma}^{p-1}|f-g|Q_{\omega,\gamma}^{-1}(|f|+|g|)\{1+Q_{\omega,\gamma}^{-(p-2)}(|f|^{p-2}+|g|^{p-2})\}
	\\ \label{eq2.13}
	&\lesssim |f-g|(Q_{\omega,\gamma}^{p-2}|f|+Q_{\omega,\gamma}^{p-2}|g|+|f|^{p-1}+|g|^{p-1}).
\end{align}
Thus, by the Sobolev embedding $H^{1}(\mathbb{R}) \hookrightarrow L^{\infty}(\mathbb{R})$ and $|I|\leq 1$, we get
\begin{align*}
	&\|R(f)-R(g)\|_{L_{t}^{1}L^{2}}
	\\
	&\lesssim \|f-g\|_{L_{t}^{\infty}L^{2}}
	(\|Q_{\omega,\gamma}\|_{L^{p-2}L^{\infty}}^{p-2}\|f\|_{L_{t}^{\infty}L^{\infty}}
	+\|Q_{\omega,\gamma}\|_{L^{p-2}L^{\infty}}^{p-2}\|g\|_{L_{t}^{\infty}L^{\infty}}
	\\
	&\quad +\|f\|_{L^{p-2}L^{\infty}}^{p-2}\|f\|_{L_{t}^{\infty}L^{\infty}}
	+\|g\|_{L^{p-2}L^{\infty}}^{p-2}\|g\|_{L_{t}^{\infty}L^{\infty}})
	\\
	&\lesssim \|f-g\|_{L_{t}^{\infty}L^{2}} 
	(\|f\|_{L_{t}^{\infty}H^{1}}
	+\|g\|_{L_{t}^{\infty}H^{1}}
	+\|f\|_{L^{\infty}H^{1}}^{p-1}
	+\|g\|_{L^{\infty}H^{1}}^{p-1}).
\end{align*}
Moreover, by a direct calculation, we have
\begin{align*}
	\partial_{x}(R(f)-R(g))
	&=pQ^{p-1}\partial_{x}Q(N(Q^{-1}f) - N(Q^{-1}g))
	\\
	&\quad -Q^{p-2}(N_{z}(Q^{-1}f)(\partial_{x}Q)f - N_{z}(Q^{-1}g)(\partial_{x}Q)g)
	\\
	&\quad +Q^{p-1}(N_{z}(Q^{-1}f)\partial_{x}f - N_{z}(Q^{-1}g)\partial_{x}g)
	\\
	&\quad -Q^{p-2}(N_{\overline{z}}(Q^{-1}f)(\partial_{x}Q)\overline{f} - N_{\overline{z}}(Q^{-1}g)(\partial_{x}Q)\overline{g})
	\\
	&\quad +Q^{p-1}(N_{\overline{z}}(Q^{-1}f)\overline{\partial_{x}f} - N_{\overline{z}}(Q^{-1}g)\overline{\partial_{x}g})
	\\
	&=(i) + (ii) + (iii) + (iv)+(v),
\end{align*}
where we denote $Q=Q_{\omega,\gamma}$ for simplicity. It holds that 
\begin{align*}
	|(i)|
	\lesssim |\partial_{x}Q| Q^{-1} |R(f)-R(g)|
	\lesssim |f-g|(Q^{p-2}(|f|+|g|)+|f|^{p-1}+|g|^{p-1}),
\end{align*}
where we used $|\partial_{x}Q| Q^{-1} \lesssim 1$ by Lemma \ref{lem2.19}. We have 
\begin{align*}
	|(ii)|
	&\lesssim Q^{p-2}|\partial_{x}Q||N_{z}(Q^{-1}f)f - N_{z}(Q^{-1}g)g|
	\\ 
	&= Q^{p-2}|\partial_{x}Q| |N_{z}(Q^{-1}f)f - N_{z}(Q^{-1}f)g  + N_{z}(Q^{-1}f)g - N_{z}(Q^{-1}g)g|
	\\ 
	&\lesssim Q^{p-2}|\partial_{x}Q|\{ |N_{z}(Q^{-1}f)||f -g|  + |N_{z}(Q^{-1}f) - N_{z}(Q^{-1}g)||g|\}.
\end{align*}
From \eqref{eq2.12}, this can be estimated by 
\begin{align*}
	&Q^{p-2}|\partial_{x}Q|\{ |N_{z}(Q^{-1}f)||f -g|  + |N_{z}(Q^{-1}f) - N_{z}(Q^{-1}g)||g|\}
	\\
	&\lesssim Q^{p-2}|\partial_{x}Q|\{ (|Q^{-1}f| + |Q^{-1}f|^{p-1})|f -g|  
	+ Q^{-1}|f-g|(1+|Q^{-1}f|^{p-2}+|Q^{-1}g|^{p-2})|g|\}
	\\
	&\approx Q^{p-2}|\partial_{x}Q||f-g|\{ |Q^{-1}f| + |Q^{-1}f|^{p-1} 
	+ Q^{-1}|g|+Q^{-1}|g||Q^{-1}f|^{p-2}+|Q^{-1}g|^{p-1})\}.
\end{align*}
By the Young inequality, we have
\begin{align*}
	Q^{-1}|g||Q^{-1}f|^{p-2}
	\lesssim (Q^{-1}|g|)^{p-1} + |Q^{-1}f|^{p-1}.
\end{align*}
Thus we obtain 
\begin{align*}
	 &Q^{p-2}|\partial_{x}Q||f-g|\{ |Q^{-1}f| + |Q^{-1}f|^{p-1} 
	+ Q^{-1}|g|+Q^{-1}|g||Q^{-1}f|^{p-2}+|Q^{-1}g|^{p-1})\}
	\\
	&\lesssim  Q^{p-2}|\partial_{x}Q||f-g|\{ |Q^{-1}f| 
	+ Q^{-1}|g|+|Q^{-1}f|^{p-1}+|Q^{-1}g|^{p-1})\}
\end{align*}
Combining these estimates give us that 
\begin{align*}
	|(ii)|
	&\lesssim Q^{p-2}|\partial_{x}Q||f-g|\{ |Q^{-1}f| 
	+ Q^{-1}|g|+|Q^{-1}f|^{p-1}+|Q^{-1}g|^{p-1})\}
	\\
	&\lesssim |f-g|(Q^{p-2}|f| +Q^{p-2}|g|+|f|^{p-1}+|g|^{p-1}),
\end{align*}
where we used $Q^{-1}|\partial_{x}Q|\lesssim 1$. We can estimate (iv) similarly. 
\begin{align*}
	|(iii)|
	&\lesssim Q^{p-1}\{ |N_{z}(Q^{-1}f)||\partial_{x}f-\partial_{x}g|
	+|N_{z}(Q^{-1}f)- N_{z}(Q^{-1}g)||\partial_{x}g|\}
	\\
	&\lesssim (Q^{p-2}|f| + |f|^{p-1})|\partial_{x}f-\partial_{x}g|
	+|f-g|(Q^{p-2}+|f|^{p-2}+|g|^{p-2}))|\partial_{x}g|.
\end{align*}
(v) can be estimated in the similar way to (iii). We omit the calculation.  

As a consequence, we get
\begin{align*}
	\|(i)\|_{L_{t}^{1}L^{2}}
	\lesssim  \|f-g\|_{L^{\infty}L^{2}}(\|f\|_{L^{\infty}H^{1}}+\|g\|_{L^{\infty}H^{1}}+\|f\|_{L^{\infty}H^{1}}^{p-1}+\|g\|_{L^{\infty}H^{1}}^{p-1}),
\end{align*}
\begin{align*}
	\|(ii)\|_{L_{t}^{1}L^{2}}+\|(iv)\|_{L_{t}^{1}L^{2}}
	\lesssim \|f-g\|_{L^{\infty}L^{2}}(\|f\|_{L^{\infty}H^{1}} 
	+ \|g\|_{L^{\infty}H^{1}} + \|f\|_{L^{\infty}H^{1}}^{p-1}+ \|g\|_{L^{\infty}H^{1}} ^{p-1}),
\end{align*}
and
\begin{align*}
	\|(iii)\|_{L_{t}^{1}L^{2}}+\|(v)\|_{L_{t}^{1}L^{2}}
	&\lesssim  (\|f\|_{L^{\infty}H^{1}} + \|f\|_{L^{\infty}H^{1}}^{p-1})\|\partial_{x}f-\partial_{x}g\|_{L^{\infty}L^{2}}
	\\
	&\quad+\|f-g\|_{L^{\infty}L^{\infty}}(1+\|f\|_{L^{\infty}H^{1}}^{p-2}+\|g\|_{L^{\infty}H^{1}}^{p-2})\|\partial_{x}g\|_{L^{\infty}L^{2}}
	\\
	&\lesssim \|f-g\|_{L^{\infty}H^{1}}(\|f\|_{L^{\infty}H^{1}} + \|f\|_{L^{\infty}H^{1}}^{p-1}+\|g\|_{L^{\infty}H^{1}} + \|g\|_{L^{\infty}H^{1}}^{p-1}).
\end{align*}
Combining all the above estimates, we get
\begin{align*}
	&\|\partial_{x}(R(f)-R(g))\|_{L^{1}L^{2}}
	\\
	&\lesssim  \|f-g\|_{L^{\infty}H^{1}}(\|f\|_{L^{\infty}H^{1}} + \|f\|_{L^{\infty}H^{1}}^{p-1}+\|g\|_{L^{\infty}H^{1}} + \|g\|_{L^{\infty}H^{1}}^{p-1}).
\end{align*}
This estimate and the estimate for the $L_{t}^{1}L^{2}$-norm finish the proof. 
\end{proof}

\section{Construction of the family of special solutions}
\label{sec3}

In this section, we prove the existence of the family of the special solutions. We first construct an approximate solution. Secondly, we show the existence of the special solutions by contraction mapping principle. 

\subsection{Approximate solutions}
\label{sec3.1}

To estimate the error term, we define the following notion.


%

\begin{definition}
We say that $f(t,x)=O(g(t))$ as $t \to \infty$ in $\mathscr{S}^{1}(\mathbb{R})$ if it satisfies 
\begin{align*}
	\limsup_{t\to \infty} \frac{\|x^\beta \partial_x^\alpha f(t)\|_{L^\infty}}{g(t)} < C_{\alpha,\beta}
\end{align*}
for $\alpha=0,1$ and $\beta \in \mathbb{N}$. 
\end{definition}

We construct an approximate solution. 

\begin{lemma}
\label{lem3.1}
Let $A \in \mathbb{R}$. There exists a sequence $\{Z_{k}^{A}\}_{k\in\mathbb{N}}$ of functions in $\mathcal{D}_{\even}\cap \widetilde{\mathcal{S}}$ such that $Z_{1}^{A}=A\mathcal{Y}_{+}$ and, if $k \geq 1$ and $\mathcal{V}_{k}^{A}=\sum_{j=1}^{k}e^{-je_{\omega}t}Z_{j}^{A}$, then we have
\begin{align}
\label{eq3.1}
	\partial_{t} \mathcal{V}_{k}^{A} + \mathcal{L}_{\omega,\gamma}\mathcal{V}_{k}^{A} =iR(\mathcal{V}_{k}^{A})+O(e^{-(k+1)e_{\omega}t})
\end{align}
as $t \to \infty$ in $\mathscr{S}^{1}(\mathbb{R})$. 
\end{lemma}

\begin{proof}
We use an induction argument. First, we define $Z_{1}^{A}:=A \mathcal{Y}_{+}\in \mathcal{D}_\even \cap \widetilde{\mathcal{S}}$ and $\mathcal{V}_{1}^{A}:=e^{-e_{\omega}t}Z_{1}^{A}$. Then it clearly holds that
\begin{align*}
	\partial_{t} \mathcal{V}_{1}^{A} + \mathcal{L}_{\omega,\gamma}\mathcal{V}_{1}^{A} -iR(\mathcal{V}_{1}^{A})
	= -e_{\omega} e^{-e_{\omega}t}A \mathcal{Y}_{+} + A e^{-e_{\omega}t} \mathcal{L}_{\omega,\gamma} \mathcal{Y}_{+} -iR(\mathcal{V}_{1}^{A})
	=-iR(\mathcal{V}_{1}^{A})
\end{align*}
by $\mathcal{L}_{\omega,\gamma}\mathcal{Y}_{+}=e_{\omega}\mathcal{Y}_{+}$. Therefore, it is enough to show $iR(\mathcal{V}_{1}^{A})=O(e^{-2e_{\omega}t})$ as $t\to \infty$ in $\mathscr{S}^{1}$. 
Now, we have $R(f)=Q_{\omega,\gamma}^{p}N(Q_{\omega,\gamma}^{-1}f)$, where $N(z)=|1+z|^{p-1}(1+z)-1 - \frac{p+1}{2}z -\frac{p-1}{2}\overline{z}$. Since $N$ is real-analytic in the open disc $\{z:|z|<1\}$ and $N(0)=\partial_{z}N(0)=\partial_{\overline{z}}N(0)=0$, it follows from the Taylor expansion that 
\begin{align*}
	N(z)=\sum_{l+m\geq 2} a_{lm}z^{l}\overline{z}^{m}.
\end{align*}
The convergence of the series and the derivatives is uniform on the compact disc $\{z:|z|\leq 1/2\}$. We know that $\|Q_{\omega,\gamma}^{-1}e^{\eta|x|}\partial_{x}^{\alpha}\mathcal{Y}_{+}\|_{L^{\infty}}<\infty$ for $\alpha=0,1$ by Lemma \ref{lem2.21}. We note that we also have $\|Q_{\omega,\gamma}^{-1}e^{\eta|x|}\partial_{x}^{\alpha}\mathcal{L}_{\omega,\gamma}\mathcal{Y}_{+}\|_{L^{\infty}}<\infty$ for $\alpha=0,1$. 
By taking large $t_{1}>0$, we get 
\begin{align*}
	\|Q_{\omega,\gamma}^{-1}e^{\eta|x|}\partial_{x}^{\alpha}\mathcal{V}_{1}^{A}\|_{L^{\infty}}
	\leq \frac{1}{2},
\end{align*}
for $t \geq t_{1}$. This ensures the above Taylor expansion and thus we obtain
\begin{align*}
	|N(Q_{\omega,\gamma}^{-1}\mathcal{V}_{1}^{A})|
	&\leq  \sum_{l+m\geq 2} |a_{lm}|
	|Q_{\omega,\gamma}^{-1}\mathcal{V}_{1}^{A}|^{l+m}
	\\
	&\leq  \left(\sum_{l+m\geq 2} 
	\frac{|a_{lm}|}{2^{l+m-2}}\right)
	|Q_{\omega,\gamma}^{-1}\mathcal{V}_{1}^{A}|^{2}
	\\
	&\leq C |Q_{\omega,\gamma}^{-1}\mathcal{V}_{1}^{A}|^{2}
	\\
	&\leq Ce^{-2e_{\omega}t}A^{2}e^{-2\eta|x|}\|Q_{\omega,\gamma}^{-1}e^{\eta|x|}\mathcal{Y}_{+}\|_{L^\infty}^{2}
	\\
	&\leq C e^{-2e_{\omega}t}e^{-2\eta|x|}.
\end{align*}
Thus we get
\begin{align*}
	|R(\mathcal{V}_{1}^{A})| \leq \|Q_{\omega,\gamma}\|_{L^\infty}^{p} |N(Q_{\omega,\gamma}^{-1}\mathcal{V}_{1}^{A})|
	\lesssim e^{-2e_{\omega}t}e^{-2\eta|x|}.
\end{align*}
Next, we consider the estimate of the derivative. For $x \neq 0$, we have
\begin{align*}
	\partial_{x} R(\mathcal{V}_{1}^{A})
	&=p(\partial_{x}Q_{\omega,\gamma})Q_{\omega,\gamma}^{p-1}N(Q_{\omega,\gamma}^{-1}\mathcal{V}_{1}^{A})
	\\
	&\quad +Q_{\omega,\gamma}^{p}\{N_{z}(Q_{\omega,\gamma}^{-1}\mathcal{V}_{1}^{A}) + N_{\overline{z}}(Q_{\omega,\gamma}^{-1}\mathcal{V}_{1}^{A})\}((\partial_{x}Q_{\omega,\gamma})Q_{\omega,\gamma}^{-2}\mathcal{V}_{1}^{A}+ Q_{\omega,\gamma}^{-1}\partial_{x}\mathcal{V}_{1}^{A})
\end{align*}
and thus, by Lemma \ref{lem2.19}, 
\begin{align*}
	|\partial_{x} R(\mathcal{V}_{1}^{A})|
	&\lesssim \|Q_{\omega,\gamma}^{-1}\partial_{x}Q_{\omega,\gamma}\|_{L^{\infty}}\|Q_{\omega,\gamma}\|_{L^{\infty}}^{p} |N(Q_{\omega,\gamma}^{-1}\mathcal{V}_{1}^{A})|
	\\
	&+\|Q_{\omega,\gamma}\|_{L^{\infty}}^{p} \{|N_{z}(Q_{\omega,\gamma}^{-1}\mathcal{V}_{1}^{A})| + |N_{\overline{z}}(Q_{\omega,\gamma}^{-1}\mathcal{V}_{1}^{A})|\}
	\\
	&\quad \times e^{-\eta|x|}(\|Q_{\omega,\gamma}^{-1}\partial_{x}Q_{\omega,\gamma}\|_{L^{\infty}}\|Q_{\omega,\gamma}^{-1}e^{\eta|x|}\mathcal{V}_{1}^{A}\|_{L^{\infty}}+\|Q_{\omega,\gamma}^{-1}e^{\eta|x|}\partial_{x}\mathcal{V}_{1}^{A}\|_{L^{\infty}}).
\end{align*}
Now we have
\begin{align*}
	|N_{z}(Q_{\omega,\gamma}^{-1}\mathcal{V}_{1}^{A})| + |N_{\overline{z}}(Q_{\omega,\gamma}^{-1}\mathcal{V}_{1}^{A})|
	\lesssim |Q_{\omega,\gamma}^{-1}\mathcal{V}_{1}^{A}|
	\lesssim e^{-e_{\omega}t}e^{-\eta|x|}
\end{align*}
and thus 
\begin{align*}
	|\partial_{x} R(\mathcal{V}_{1}^{A})|\lesssim e^{-2e_{\omega}t}e^{-2\eta|x|}.
\end{align*}
These show that $\|x^\beta \partial_x^\alpha R(\mathcal{V}_{1}^{A})\|_{L^{\infty}} \lesssim e^{-2e_{\omega}t}$ for $\alpha=0,1$ and $\beta\in \mathbb{N}$. Hence, we get \eqref{eq3.1} for $k=1$. 

Secondly, let $k \geq 1$ and we assume that $Z_{j}^{A} \in \mathcal{D}_\even \cap \widetilde{\mathcal{S}}$ ($j=1,...,k$) are defined and $\mathcal{V}_{j}^{A}=\sum_{l=1}^{j}e^{-le_{\omega}t}Z_{l}^{A}$ satisfies \eqref{eq3.1} for any $j \leq k$. Moreover, we assume that
\begin{align}
\label{eq3.a}
	&\|Q_{\omega,\gamma}^{-1}e^{\eta|x|}\partial_{x}^{\alpha}Z_{j}^{A}\|_{L^{\infty}} < \infty,
	\\
\label{eq3.b}
	&\|Q_{\omega,\gamma}^{-1}e^{\eta|x|}\partial_{x}^{\alpha} \mathcal{L}_{\omega,\gamma}Z_{j}^{A}\|_{L^{\infty}} < \infty,
\end{align}
for $\alpha=0,1$ and all $j \leq k$. 
We define 
\begin{align*}
	\varepsilon_{k}:=\partial_{t} \mathcal{V}_{k}^{A} + \mathcal{L}_{\omega,\gamma} \mathcal{V}_{k}^{A} -iR(\mathcal{V}_{k}^{A}).
\end{align*}
Since we have 
\begin{align*}
	\partial_{t} \mathcal{V}_{k}^{A} 
	=\partial_{t} \left( \sum_{j=1}^{k}e^{-je_{\omega}t}Z_{j}^{A} \right)
	=\sum_{j=1}^{k} (-je_{\omega}) e^{-je_{\omega}t}Z_{j}^{A},
\end{align*}
we obtain
\begin{align*}
	\varepsilon_{k}
	= \sum_{j=1}^{k}  e^{-je_{\omega}t}( -je_{\omega}Z_{j}^{A} 
	+\mathcal{L}_{\omega,\gamma}Z_{j}^{A})
	-iR(\mathcal{V}_{k}^{A}).
\end{align*}
Since it holds that $\mathcal{V}_{k}^{A}=\sum_{j=1}^{k}e^{-je_{\omega}t}Z_{j}^{A}$ and $\|Q_{\omega,\gamma}^{-1}e^{\eta|x|}\partial_{x}^{\alpha}Z_{j}^{A}\|_{L^{\infty}} < \infty$ for $\alpha=0,1$ and $j \leq k$, we have
\begin{align*}
	\|Q_{\omega,\gamma}^{-1}e^{\eta|x|}\partial_{x}^{\alpha}\mathcal{V}_{k}^{A}\|_{L^{\infty}}
	\leq \sum_{j=1}^{k}e^{-je_{\omega}t} \|Q_{\omega,\gamma}^{-1}e^{\eta|x|}\partial_{x}^{\alpha}Z_{j}^{A}\|_{L^{\infty}}
	\leq \frac{1}{2}
\end{align*}
for $t \geq t_{k}$, where $t_{k}$ is sufficiently large. 
\begin{align*}
	R(\mathcal{V}_{k}^{A})
	&=Q_{\omega,\gamma}^{p}  \sum_{l+m\geq 2} a_{lm} Q_{\omega,\gamma}^{-(l+m)} (\mathcal{V}_{k}^{A})^{l} \overline{\mathcal{V}_{k}^{A}}^{m}
	\\
	&=Q_{\omega,\gamma}^{p}  \sum_{l+m \geq 2} a_{lm} Q_{\omega,\gamma}^{-(l+m)} \\
	&\quad \times \sum_{\mathit{\Delta}} c_{\mathit{\Delta}} e^{-(\sum_{j=1}^{k}jl_{j})e_{\omega}t}e^{-(\sum_{j=1}^{k}jm_{j})e_{\omega}t} (Z_{1}^{A})^{l_{1}} \cdots (Z_{k}^{A})^{l_{k}}\overline{(Z_{1}^{A})^{m_{1}} \cdots (Z_{k}^{A})^{m_{k}}},
\end{align*}
where $\sum_{\mathit{\Delta}}$ denotes the summation on the divisions of $l$ and $m$, that is, $\sum_{\mathit{\Delta}}=\sum_{l_{1}+...+l_{k}=l, m_{1}+...+m_{k}=m}$, and $c_{\mathit{\Delta}} =\frac{l!}{l_{1}!...l_{k}!} \frac{m!}{m_{1}!...m_{k}!}$.

Now we have $Z_j^A \in \mathcal{D}_\even \cap \widetilde{\mathcal{S}}$ and $\mathcal{L}_{\omega,\gamma}Z_j^A \in L^2_\even \cap \widetilde{\mathcal{S}}$ for any $j=1,...,k$. We also have $Q_{\omega,\gamma}^p Q_{\omega,\gamma}^{-(l+m)} (Z_{1}^{A})^{l_{1}} \cdots (Z_{k}^{A})^{l_{k}}\overline{(Z_{1}^{A})^{m_{1}} \cdots (Z_{k}^{A})^{m_{k}}} \in L^2_\even \cap \widetilde{\mathcal{S}}$, where $l_1+\cdots+l_k=l$, $m_1+\cdots+m_k=m$, and $l,m\in \mathbb{N}$. 
Moreover,  by the induction's assumption, we have \eqref{eq3.a}, \eqref{eq3.b}, and 
\begin{align*}
	\|Q_{\omega,\gamma}^{-1}e^{\eta(l+m)|x|} \partial_x^\alpha \{Q_{\omega,\gamma}^p Q_{\omega,\gamma}^{-(l+m)} (Z_{1}^{A})^{l_{1}} \cdots (Z_{k}^{A})^{l_{k}}\overline{(Z_{1}^{A})^{m_{1}} \cdots (Z_{k}^{A})^{m_{k}}}\}\|_{L^\infty}<\infty. 
\end{align*}
These give us that
\begin{align}
\label{eq3.e}
	\varepsilon_{k}=\sum_{j=1}^{k+1} e^{-je_{\omega}t} F_{j}(x)
	+O(e^{-(k+2)e_{\omega}t})
\end{align}
as $t\to \infty$ in $\mathscr{S}^{1}$, where $F_{j} \in L^2_\even  \cap \widetilde{\mathcal{S}}$ and it satisfies
\begin{align}
\label{eq3.d}
	\|Q_{\omega,\gamma}^{-1}e^{\eta|x|}\partial_{x}^{\alpha}F_{j}\|_{L^{\infty}} < \infty.
\end{align} Since $\mathcal{V}_{k}^{A}$ satisfies \eqref{eq3.1} by assumption, we must have $F_{j}=0$ 
for $j=1,...,k$. Hence we have
\begin{align*}
	\varepsilon_{k}= e^{-(k+1)e_{\omega}t} F_{k+1}(x)
	+O(e^{-(k+2)e_{\omega}t}).
\end{align*}
Since $(k+1)e_{\omega}$ is not the spectrum of $\mathcal{L}_{\omega,\gamma}$, $Z_{k+1}^{A}:=(-\mathcal{L}_{\omega,\gamma}+(k+1)e_{\omega})^{-1}F_{k+1}$ is well-defined. We have $Z_{k+1}^{A} \in \mathcal{D}_\even \cap \widetilde{\mathcal{S}}$ by definition. By Lemma \ref{lem2.22}, we also have
\begin{align}
\label{eq3.c}
	\|Q_{\omega,\gamma}^{-1}e^{\eta|x|}\partial_{x}^{\alpha}Z_{k+1}^{A}\|_{L^{\infty}} < \infty.
\end{align}
Moreover, since $\mathcal{L}_{\omega,\gamma}Z_{k+1}^{A} = (k+1)e_\omega Z_{k+1}^{A}-F_{k+1}$, by \eqref{eq3.c} and \eqref{eq3.d}, we also have
\begin{align*}
	\|Q_{\omega,\gamma}^{-1}e^{\eta|x|}\partial_{x}^{\alpha} \mathcal{L}_{\omega,\gamma}Z_{k+1}^{A}\|_{L^{\infty}} < \infty.
\end{align*}
We define $\mathcal{V}_{k+1}^{A}:=\mathcal{V}_{k}^{A}+e^{-(k+1)e_{\omega}t}Z_{k+1}^{A}$. Then, by the definition, we have
\begin{align*}
	\varepsilon_{k+1}
	&=\partial_{t} \mathcal{V}_{k}^{A} + \mathcal{L}_{\omega,\gamma} \mathcal{V}_{k}^{A} -e^{-(k+1)e_{\omega}t}F_{k+1} -iR(\mathcal{V}_{k+1}^{A})
	\\
	&=\varepsilon_{k}-e^{-(k+1)e_{\omega}t}F_{k+1} +i\{ R(\mathcal{V}_{k}^{A}) -R(\mathcal{V}_{k+1}^{A})\}.
\end{align*}
By \eqref{eq3.e}, we have $\varepsilon_{k}-e^{-(k+1)e_{\omega}t}F_{k+1}=O(e^{-(k+2)e_{\omega}t})$ as $t\to \infty$ in $\mathscr{S}^{1}(\mathbb{R})$. We can show that $R(\mathcal{V}_{k}^{A}) -R(\mathcal{V}_{k+1}^{A})=O(e^{-(k+2)e_{\omega}t})$ as $t\to \infty$ in $\mathscr{S}^{1}(\mathbb{R})$. Indeed, in a similar way to Lemma \ref{lem2.24}, it holds that
\begin{align*}
	\|x^\beta ( R(f)-R(g))\|_{L^\infty}
	&\leq  C(\|f\|_{W^{1,\infty}},\|g\|_{W^{1,\infty}}) \|x^\beta (f-g)\|_{L^\infty},
	\\
	\|x^\beta \partial_x ( R(f)-R(g))\|_{L^\infty}
	&\leq  C(\|f\|_{W^{1,\infty}},\|g\|_{W^{1,\infty}})
	\\
	&\quad \times ( \|x^\beta (f-g)\|_{L^\infty}+\|x^\beta \partial_x (f-g)\|_{L^\infty}),
\end{align*}
where $C(a,b)$ is a $C^{1}$-function of $a,b$ and $C(0,0)=0$. Since $\mathcal{V}_{k+1}^{A} - \mathcal{V}_{k}^{A}= e^{-(k+1)e_\omega t}Z_{k+1}^A$ and $\|\mathcal{V}_{k}^{A}\|_{W^{1,\infty}}+\|\mathcal{V}_{k+1}^{A}\|_{W^{1,\infty}} \lesssim e^{-e_\omega t}$, we have $R(\mathcal{V}_{k}^{A}) -R(\mathcal{V}_{k+1}^{A})=O(e^{-(k+2)e_{\omega}t})$. This completes the proof. 
\end{proof}


\subsection{Existence by contraction mapping theorem}

Let $k \in \mathbb{N}$ and let $t_k >0$ be sufficiently large. 
We define a norm 
\begin{align*}
	\|h\|_{\mathcal{E}}&:=\sup_{t \in [t_{k},\infty)} e^{(k+\frac{1}{2})e_{\omega}t} \|h\|_{L_{t}^{\infty}H^{1}(t,\infty)}
\end{align*}
and function spaces
\begin{align*}
	\mathcal{E}&=\mathcal{E}(k,t_{k}):=\{h \in C ( [t_{k},\infty);H^{1}(\mathbb{R})): \|h\|_{\mathcal{E}}<\infty\},
	\\
	\mathcal{B}&=\mathcal{B}(k,t_{k}):=\{h \in \mathcal{E}, \|h\|_{\mathcal{E}} \leq 1\},
\end{align*}
equipped with the metric
\begin{align*}
	d(u,v):=\sup_{t \in [t_{k},\infty)} e^{(k+\frac{1}{2})e_{\omega}t} \|u-v\|_{L_{t}^{\infty}H^{1}(t,\infty)}.
\end{align*}
Then $(\mathcal{B},d)$ is a complete metric space. We will consider a contraction map on $\mathcal{B}$ and construct the solution to \eqref{NLS} going to the ground state. Before discussing the construction, we prepare the following estimates.

\begin{lemma}
\label{lem3.2}
For any $\eta>0$, there exists $\widetilde{k}=\widetilde{k}(\eta)\in \mathbb{N}$ such that if $k\geq \widetilde{k}\geq 1$ then for $g,h \in \mathcal{B}$ the following estimates are true for $t\geq t_{k}$.
\begin{enumerate}
\item $\|P(h)\|_{L_{t}^{1}H^{1}(t,\infty)} \leq \eta e^{-(k+\frac{1}{2})e_{\omega}t} \|h\|_{\mathcal{E}}$,
\item $\|R(\mathcal{V}_k^A+h)-R(\mathcal{V}_k^A)\|_{L_{t}^{1}H^{1}(t,\infty)} \leq C_{k} e^{-e_{\omega}t}e^{-(k+\frac{1}{2})e_{\omega}t} \|h\|_{\mathcal{E}}$,
\item $\|R(\mathcal{V}_k^A+g)-R(\mathcal{V}_k^A+h)\|_{L_{t}^{1}H^{1}(t,\infty)} \leq C_{k} e^{-e_{\omega}t}e^{-(k+\frac{1}{2})e_{\omega}t}d(g,h)$,
\item $\|\varepsilon_{k}\|_{L_{t}^{1}H^{1}(t,\infty)} \leq C_{k} e^{-(k+1)e_{\omega}t}$,
\end{enumerate}
where $\mathcal{V}_k^A$ and $\varepsilon_k$ are defined in Section \ref{sec3.1} for $A \in \mathbb{R}$. 
\end{lemma}

\begin{proof}
(1). Let $\tau>0$ be fixed later and let $t_{j}:=t+j\tau$ for a nonnegative integer $j$. By Lemma \ref{lem2.23}, we have
\begin{align*}
	\|P(h)\|_{L_{t}^{1}H^{1}(t,\infty)}
	\leq C \tau \sum_{j=0}^{\infty}  \|h\|_{L_{t}^{\infty}H^{1}(t_{j},t_{j+1})}
	\leq C\tau \frac{e^{-(k+\frac{1}{2})e_{\omega}t} }{1-e^{-(k+\frac{1}{2})e_{\omega}\tau} } \|h\|_{\mathcal{E}}.
\end{align*}
For $\eta>0$, we take $\tau:=\eta/(2C)$ and take $k$ large enough (depending on $\tau$, i.e., $\eta$) such that $e^{-(k+\frac{1}{2})e_{\omega}\tau} \leq1/2$. Then the desired estimate is obtained. 

(2). Let $t_j:=t+j$ for $j\in \mathbb{N}$. Since $\|\mathcal{V}_{k}^A\|_{L^{\infty}H^{1}(t_j,t_{j+1})} \lesssim e^{-e_\omega t_j}$ by the definition of $\mathcal{V}_k^A$, it follows from Lemma \ref{lem2.24} that
\begin{align*}
	\|R(\mathcal{V}_{k}^A+h)-R(\mathcal{V}_{k}^A)\|_{L_{t}^{1}H^{1}(t_j,t_{j+1})}
	&\lesssim e^{-e_\omega t_j} \|h\|_{L_{t}^{\infty}H^{1}(t_j,t_{j+1})}
	\\
	&\lesssim e^{-e_\omega t_j} e^{-(k+1/2)e_\omega t_j}\|h\|_{\mathcal{E}},
\end{align*}
where we also used  and $\|h\|_{L^{\infty}H^{1}(t_j,t_{j+1})} \leq e^{-(k+1/2)e_\omega t_j}\|h\|_{\mathcal{E}}\leq e^{-(k+1/2)e_\omega t_j}$.
Therefore, we get
\begin{align*}
	\|R(\mathcal{V}_{k}^A+h)-R(\mathcal{V}_{k}^A)\|_{L_{t}^{1}H^{1}(t,\infty)}
	&\lesssim \sum_{j=0}^{\infty} \|R(\mathcal{V}_{k}^A+h)-R(\mathcal{V}_{k}^A)\|_{L_{t}^{1}H^{1}(t_j,t_{j+1})}
	\\
	&\lesssim e^{-e_\omega t} e^{-(k+1/2)e_\omega t}\|h\|_{\mathcal{E}}.
\end{align*}

(3). This follows from the similar argument to (2). 

(4). As seen in the proof of Lemma \ref{lem3.1}, we have $\|x^\beta \partial_x^\alpha \varepsilon_{k}\|_{L^\infty} \leq C_{k} e^{-(k+1)e_{\omega}t}$ for $t>t_k$, $\alpha=0,1$, and $\beta \in \mathbb{N}$. By the H\"{o}lder inequality, it holds that
\begin{align*}
	\|\varepsilon_{k}\|_{L_{t}^{1}H^{1}(t,\infty)} 
	&\lesssim  \| \langle x \rangle \varepsilon_{k}\|_{L_{t}^{1}L^\infty (t,\infty)}
	+\|  \langle x \rangle \partial_x \varepsilon_{k}\|_{L_{t}^{1}L^\infty(t,\infty)} 
	\\
	&\lesssim \| e^{-(k+1)e_{\omega}t} \|_{L_{t}^{1}(t,\infty)}
	\\
	&\leq C_{k} e^{-(k+1)e_{\omega}t}.
\end{align*}
This completes the proof. 
\end{proof}

The following proposition gives the family $\{U^A\}_{A\in \mathbb{R}}$ of the solutions going to the ground state in positive time. The behavior of $U^A$ in the negative time will be investigated later (see the proof of Theorem \ref{thm1.5}). 

\begin{proposition}
\label{prop3.3}
Let $A \in \mathbb{R}$. There exists $k_{0} \in \mathbb{N}$ such that for any $k \geq k_{0}$ there exists $t_{k}>0$ and a solution $U^{A}$ to \eqref{NLS} such that 
\begin{align}
\label{eq00}
	\|U^{A}(t) - e^{i\omega t}Q_{\omega,\gamma} - e^{i\omega t}\mathcal{V}_{k}^A\|_{L^{\infty}(t,\infty;H^{1})} \leq e^{-(k+1/2)e_{\omega}t}
\end{align}
for all $t \geq t_{k}$. 
Moreover, $U^{A}$ is the unique solution satisfying \eqref{eq00} for $t \geq \widetilde{t}$, where $\widetilde{t}$ is larger than $ t_{k}$. 
Furthermore, $U^{A}$ is independent of $k$ and satisfies 
\begin{align}
\label{eq003}
	\|U^{A}(t) - e^{i\omega t}Q_{\omega,\gamma} - Ae^{-e_{\omega}t + i\omega t}\mathcal{Y}_{+}\|_{H^{1}} \leq e^{-2e_{\omega}t}. 
\end{align}
\end{proposition}

\begin{proof}
Let $A \in \mathbb{R}$ be fixed.  We define
\begin{align*}
	h:=U^{A}-e^{i\omega t}Q_{\omega,\gamma} - e^{i\omega t}\mathcal{V}_{k}^{A}.
\end{align*}
The function $U^{A}$ is a solution if and only if $h$ satisfies
\begin{align*}
	i\partial_{t} h + \Delta_{\gamma} h =-P(h) -(R(\mathcal{V}_{k}^A+h)-R(\mathcal{V}_{k}^A)) +i \varepsilon_k
\end{align*}
where $\varepsilon_k =O(e^{-(k+1)e_{\omega}t})$ in $\mathscr{S}^{1}$ for all $k$. Thus it is enough to solve the fixed-point problem 
\begin{align*}
	h(t)=\Psi(h)(t),
\end{align*}
where
\begin{align*}
	\Psi(h):=-i \int_{t}^{\infty} e^{i(t-s)\Delta_{\gamma}} \{-P(h) -(R(\mathcal{V}_{k}^A+h)-R(\mathcal{V}_{k}^A)) +i \varepsilon_k \} ds.
\end{align*}
We show that $\Psi$ is a contraction map on $\mathcal{B}$. 
Since $\|f\|_{H^1} \approx \|\sqrt{1-\Delta_\gamma}f\|_{L^2}$, we have
\begin{align*}
	\|\Psi(h)\|_{L_{t}^{\infty}H^{1}(t,\infty)} 
	&\lesssim  \| P(h)\|_{L_{t}^{1}H^{1}(t,\infty)}  
	\\
	&\quad +\|R(\mathcal{V}_{k}^A+h)-R(\mathcal{V}_{k}^A)\|_{L_{t}^{1}H^{1}(t,\infty)}  +\|\varepsilon_k \|_{L_{t}^{1}H^{1}(t,\infty)}.
\end{align*}
Let $\eta>0$ be sufficiently small and let $k\geq \widetilde{k}(\eta)$ be fixed, where $\widetilde{k}(\eta)$ is defined in Lemma \ref{lem3.2}. 
It holds from Lemma \ref{lem3.2} that
\begin{align*}
	\|\Psi(h)\|_{L_{t}^{\infty}H^{1}(t,\infty)} 
	\lesssim \eta e^{-(k+\frac{1}{2})e_\omega t} +e^{-e_\omega t} e^{-(k+\frac{1}{2})e_\omega t}
	\leq e^{-(k+\frac{1}{2})e_\omega t}
\end{align*}
for $t >t_k$ by taking $t_k$ large. Hence we get $\|\Psi(h)\|_{\mathcal{E}} \leq 1$. This means $\Psi: \mathcal{B} \to \mathcal{B}$ is well defined. Noting that $P$ is linear, we have
\begin{align*}
	\|\Psi(g)-\Psi(h)\|_{L_{t}^{\infty}H^{1}(t,\infty)}
	&\lesssim \|P(g-h)\|_{L_{t}^{1}H^{1}(t,\infty)} 
	\\
	&\quad + \|R(\mathcal{V}_{k}^A+h)-R(\mathcal{V}_{k}^A+g)\|_{L_{t}^{1}H^{1}(t,\infty)}.
\end{align*}
By Lemma \ref{lem3.2}, we also obtain
\begin{align*}
	\|\Psi(g)-\Psi(h)\|_{L_{t}^{\infty}H^{1}(t,\infty)}
	&\lesssim \eta e^{-(k+\frac{1}{2})e_{\omega}t} d(g,h) + e^{-e_{\omega}t} e^{-(k+\frac{1}{2})e_{\omega}t} d(g,h)
	\\
	&\leq \frac{1}{2} e^{-(k+\frac{1}{2})e_{\omega}t} d(g,h)
\end{align*}
for $t >t_k$ by taking $\eta$ small and $t_k$ large. Thus we get 
\begin{align*}
	d(\Psi(g),\Psi(h))=\|\Psi(g)-\Psi(h)\|_{\mathcal{E}}
	\leq  \frac{1}{2} d(g,h).
\end{align*}
The function $\Psi:\mathcal{B} \to \mathcal{B}$ is a contraction map. By the contraction mapping theorem, there exists unique $h\in \mathcal{B}$ such that $\Psi(h)=h$. 
We note that the fixed point argument holds for $\widetilde{t}$ larger than $t_k$, so that the uniqueness is also true for any $\widetilde{t} \geq t_k$ in the class of solutions to \eqref{NLS} in $C([\widetilde{t},\infty),H^1(\mathbb{R}))$ satisfying \eqref{eq00} for any $t \geq \widetilde{t}$. 

Let $k_1>k$. 
By the construction, we also have $t_{k_1}>0$ and the solution $\widetilde{U}^A \in C([t_{k_1},\infty),H^1(\mathbb{R}))$ to \eqref{NLS} satisfying
\begin{align*}
	\|\widetilde{U}^{A}(t) - e^{i\omega t}Q_{\omega,\gamma} - e^{i\omega t}\mathcal{V}_{k_1}^A\|_{L^{\infty}(t,\infty;H^{1})} \leq e^{-(k_1+1/2)e_{\omega}t}
\end{align*}
for $t > t_{k_1}$. Then it holds by the triangle inequality that
\begin{align*}
	&\|\widetilde{U}^{A}(t) - e^{i\omega t}Q_{\omega,\gamma} - e^{i\omega t}\mathcal{V}_{k}^A\|_{L^{\infty}(t,\infty;H^{1})} 
	\\
	&\leq \|\widetilde{U}^{A}(t) - e^{i\omega t}Q_{\omega,\gamma} - e^{i\omega t}\mathcal{V}_{k_1}^A\|_{L^{\infty}(t,\infty;H^{1})} 
	+\|\mathcal{V}_{k_1}^A-\mathcal{V}_{k}^A\|_{L^{\infty}(t,\infty;H^{1})} 
	\\
	&\leq e^{-(k_1+1/2)e_{\omega}t} + \|\mathcal{V}_{k_1}^A-\mathcal{V}_{k}^A\|_{L^{\infty}(t,\infty;H^{1})}.
\end{align*}
Now, we have
\begin{align*}
	\|\mathcal{V}_{k_1}^A-\mathcal{V}_{k}^A\|_{L^{\infty}(t,\infty;H^{1})}
	\lesssim e^{-(k+1)e_\omega t}
\end{align*}
by the definition of $\mathcal{V}_{k}^A$. Since $k_1>k$, we obtain
\begin{align*}
	\|\widetilde{U}^{A}(t) - e^{i\omega t}Q_{\omega,\gamma} - e^{i\omega t}\mathcal{V}_{k}^A\|_{L^{\infty}(t,\infty;H^{1})} 
	&\leq e^{-(k_1+1/2)e_{\omega}t} +Ce^{-(k+1)e_\omega t}
	\\
	&\leq e^{-(k+1/2)e_\omega t} 
\end{align*}
for any sufficiently large $t$. By the uniqueness, we get $U^A = \widetilde{U}^A$. This shows that $U^A$ is independent of $k$. Since $\mathcal{V}_k^A$ satisfies 
\begin{align*}
	\mathcal{V}_k^A = e^{-e_\omega t}\mathcal{Y}_+ + O(e^{-2e_\omega t}),
\end{align*}
 as $t \to \infty$ in $\mathscr{S}^1$, $U^A$ fulfilled 
\begin{align*}
	\|U^{A}(t) - e^{i\omega t}Q_{\omega,\gamma} - Ae^{-e_{\omega}t + i\omega t}\mathcal{Y}_{+}\|_{H^{1}} \leq e^{-2e_{\omega}t}. 
\end{align*}
This completes the proof. 
\end{proof}

\section{Non-scattering solutions converge to the ground state}
\label{sec4}

In this section, we show that a global non-scattering solution on the threshold goes to the ground state solution.

\subsection{Modulation}
\label{sec4.1}

\begin{lemma}
\label{modulation}
There exist $\mu_{0}>0$ and a function $\varepsilon:(0,\mu_{0}) \to (0,\infty)$ with $\varepsilon(\mu) \to 0$ as $\mu \to 0$ such that, for all $f \in H_{\even}^{1}(\mathbb{R})$ satisfying $M(f)=M(Q_{\omega,\gamma})$, $E_{\gamma}(f)=E_{\gamma}(Q_{\omega,\gamma})$ and $|\mu(f)|\leq \mu_0$, there exists a $C^{1}$-function $\theta=\theta(f)$ such that 
\begin{align*}
	\|e^{-i\theta}f - Q_{\omega,\gamma}\|_{H^{1}} \leq \varepsilon(\mu)
\end{align*}
and 
\begin{align*}
	\im \int_{\mathbb{R}} g Q_{\omega,\gamma}dx=0,
\end{align*}
where $g=e^{-i\theta}f -Q_{\omega,\gamma}$. 
\end{lemma}

\begin{proof}
First, we will show the statement when $f$ is close to $Q_{\omega,\gamma}$. 
We define 
\begin{align*}
	G(\theta,v):= \im \int_{\mathbb{R}} (e^{-i\theta}v -Q_{\omega,\gamma})Q_{\omega,\gamma} dx
\end{align*}
for $\theta \in \mathbb{R}$ and $v \in H_{\even}^{1}(\mathbb{R})$. Then we have $G(0,Q_{\omega,\gamma})=0$. Moreover, we also have
\begin{align*}
	\frac{\partial G}{\partial \theta}(0,Q_{\omega,\gamma})
	=-\|Q_{\omega,\gamma}\|_{L^{2}}^{2} \neq 0
\end{align*}
Therefore, by the implicit function theorem, we get a $C^{1}$ function $\theta$ defined on a neighborhood of $Q_{\omega,\gamma}$ such that 
\begin{align*}
	\im \int_{\mathbb{R}} (e^{-i\theta(v)}v -Q_{\omega,\gamma})Q_{\omega,\gamma} dx=0.
\end{align*}
If $f$ is a function satisfying the assumption, we can take $\theta_{0}$ such that $\widetilde{f}:=e^{-i\theta_{0}}f - Q_{\omega,\gamma}$ and $\|\widetilde{f}\|_{H^{1}} \leq \widetilde{\varepsilon}(\mu)$, where $\widetilde{\varepsilon}(\mu) \to 0$ as $\mu \to 0$. Namely, $e^{-i\theta_{0}}f$ is close to $Q_{\omega,\gamma}$. The implicit function theorem ensures the uniqueness of $\theta$ and the regularity of the map. We get the statement. 
\end{proof}

Let $u$ be an even solution satisfying the mass-energy condition 
\begin{align*}
	M(u)=M(Q_{\omega,\gamma}) \text{ and } E_{\gamma}(u)=E_{\gamma}(Q_{\omega,\gamma}).
\end{align*}
We set $I_{\mu_{0}}:=\{ t \in I_{\max}: |\mu(u(t))| < \mu_{0}\}$, where $I_{\max}$ denotes the maximal existence time interval. By Lemma \ref{modulation}, we obtain a $C^{1}$-function $\widetilde{\theta}=\widetilde{\theta}(t)$ satisfying the statement of Lemma \ref{modulation} for $t \in I_{\mu_{0}}$. Set $\theta:=\widetilde{\theta}-\omega$ and 
\begin{align*}
	u(t,x)&=e^{i\theta(t)+i\omega t}(Q_{\omega,\gamma} + g(t,x))
	\\
	&=e^{i\theta(t)+i\omega t}(Q_{\omega,\gamma} +\rho(t)Q_{\omega,\gamma} +h(t,x)),
\end{align*}
where 
\begin{align*}
	\rho(t):=\frac{\int_{\mathbb{R}} (e^{-i\theta(t)-i\omega t}u - Q_{\omega,\gamma})Q_{\omega,\gamma}^{p}dx }{\|Q_{\omega,\gamma}\|_{L^{p+1}}^{p+1}}
	=\frac{\int_{\mathbb{R}} gQ_{\omega,\gamma}^{p}dx }{\|Q_{\omega,\gamma}\|_{L^{p+1}}^{p+1}}.
\end{align*}
Then $h$ satisfies the following orthogonality condition:
\begin{align*}
	\im \int_{\mathbb{R}} hQ_{\omega,\gamma}dx 
	=\re  \int_{\mathbb{R}} hQ_{\omega,\gamma}^{p}dx 
	=0.
\end{align*}

We give estimates of the parameters $\theta,\rho, g, h$. 

\begin{proposition}[Estimates of parameters]
\label{prop4.2}
We have
\begin{align*}
	|\rho|\approx \|g\|_{H^{1}} \approx \|h\|_{H^{1}} \approx |\mu(u)|
\end{align*}
for $t \in I_{\mu_{0}}$. 
\end{proposition}

\begin{lemma}
\label{lem4.3}
We have
\begin{align*}
	&|\rho| \leq \|g\|_{L^{2}},
	\\
	&\|g\|_{H^{1}} \lesssim |\rho| + \|h\|_{H^{1}},
	\\
	&\|h\|_{H^{1}} \lesssim \|g\|_{H^{1}}+|\rho| \lesssim \|g\|_{H^{1}}.
\end{align*}
\end{lemma}

\begin{proof}
By the definition of $\rho$ and the H\"{o}lder inequality, the first inequality is obtained. 
The second inequality comes from $g=\rho Q_{\omega,\gamma}+h$.
The last one is obtained by $h=g-\rho Q_{\omega,\gamma}$ and the first one.
\end{proof}

From this lemma, we may assume that $|\rho|$ and $\|h\|_{H^{1}}$ are sufficiently small since $\|g\|_{H^{1}}$ is small for $t \in I_{\mu_{0}}$.

\begin{proof}[Proof of Proposition \ref{prop4.2}]
\textbf{Claim 1.} $|\rho| \lesssim \|h\|_{H^{1}}$.

Since $M(u)=M(Q_{\omega,\gamma})$, we have
\begin{align*}
	M(Q_{\omega,\gamma})
	=M(Q_{\omega,\gamma}) + 2\re \int_{\mathbb{R}} Q_{\omega,\gamma}(\rho Q_{\omega,\gamma}+h)dx +\|\rho Q_{\omega,\gamma}+h\|_{L^{2}}^{2}
\end{align*}
and thus 
\begin{align*}
	 2\re \int_{\mathbb{R}} Q_{\omega,\gamma}(\rho Q_{\omega,\gamma}+h)dx +\|\rho Q_{\omega,\gamma}+h\|_{L^{2}}^{2}=0.
\end{align*}
Thus this implies that
\begin{align}
\label{eqqq}
	 2 \rho \|Q_{\omega,\gamma}\|_{L^{2}}^{2} 
	 +2\re \int_{\mathbb{R}} Q_{\omega,\gamma}hdx 
	 +\|\rho Q_{\omega,\gamma}+h\|_{L^{2}}^{2}=0
	 \\ \notag
	 \Leftrightarrow \rho = -\frac{2\re \int_{\mathbb{R}} Q_{\omega,\gamma}hdx + \|\rho Q_{\omega,\gamma}+h\|_{L^{2}}^{2}}{2\|Q_{\omega,\gamma}\|_{L^{2}}^{2} }.
\end{align}
Therefore, we obtain
\begin{align*}
	|\rho| \lesssim \|h\|_{L^{2}} + O(|\rho|^{2} + \|h\|_{L^{2}}^{2}).
\end{align*}
Since $|\rho|$ and $\|h\|_{L^{2}}$ can be taken small, we get
\begin{align*}
	|\rho| \lesssim \|h\|_{L^{2}}.
\end{align*}
This finishes the proof of the claim. 

\textbf{Claim 2.} $\|h\|_{H^{1}} \lesssim |\rho|$. 

Since we have $S_{\omega,\gamma}(u)=S_{\omega,\gamma}(Q_{\omega,\gamma})$, by the Taylor expansion, we obtain
\begin{align*}
	0&=S_{\omega,\gamma}(Q_{\omega,\gamma}+g) - S_{\omega,\gamma}(Q_{\omega,\gamma})
	\\
	&=\langle S_{\omega,\gamma}'(Q_{\omega,\gamma}) , g \rangle
	+\frac{1}{2} \langle S_{\omega,\gamma}''(Q_{\omega,\gamma})g , g \rangle
	+o(\|g\|_{H^{1}}^{2}).
\end{align*}
Since $Q_{\omega,\gamma}$ is the solution to $-\partial_{x}^{2} Q -\gamma \delta Q +\omega Q =Q^{p}$, we find that $S_{\omega,\gamma}'(Q_{\omega,\gamma})=0$ and thus the first term in the right hand side disappears. 
For the second term, we have
\begin{align*}
	&\langle S_{\omega,\gamma}''(Q_{\omega,\gamma})g , g \rangle
	\\
	&=\rho^{2} \langle S_{\omega,\gamma}''(Q_{\omega,\gamma}) Q_{\omega,\gamma} ,  Q_{\omega,\gamma} \rangle
	+2\rho \langle S_{\omega,\gamma}''(Q_{\omega,\gamma}) Q_{\omega,\gamma} , h \rangle 
	+\langle S_{\omega,\gamma}''(Q_{\omega,\gamma})h  ,h\rangle.
\end{align*}
Thus we get
\begin{align*}
	\langle S_{\omega,\gamma}''(Q_{\omega,\gamma})h  ,h\rangle
	&\leq -(\rho^{2} \langle S_{\omega,\gamma}''(Q_{\omega,\gamma}) Q_{\omega,\gamma} ,  Q_{\omega,\gamma} \rangle
	\\
	&\quad +2\rho \langle S_{\omega,\gamma}''(Q_{\omega,\gamma}) Q_{\omega,\gamma} , h \rangle ) 
	+o(\|g\|_{H^{1}}^{2})
	\\
	&\lesssim \rho^{2} + |\rho| \|h\|_{H^{1}} +o(\|g\|_{H^{1}}^{2}).
\end{align*}
By Lemma \ref{coercivity}, we have $\langle S_{\omega,\gamma}''(Q_{\omega,\gamma})h  ,h\rangle \gtrsim c\|h\|_{H^{1}}^{2}$ 
and thus  we obtain
\begin{align*}
	\|h\|_{H^{1}}^{2} 
	\lesssim  \rho^{2} + |\rho| \|h\|_{H^{1}} +o(\|g\|_{H^{1}}^{2})
\end{align*}
Now Lemma \ref{lem4.3} and Claim 1 imply that $\|g\|_{H^{1}} \lesssim |\rho|+ \|h\|_{H^{1}} \lesssim  \|h\|_{H^{1}}$. Hence $o(\|g\|_{H^{1}}^{2})$ can be absorbed by the left hand side, that is, we get
\begin{align*}
	\|h\|_{H^{1}}^{2} \lesssim   \rho^{2}+  |\rho| \|h\|_{H^{1}}.
\end{align*}
It follows from the Young inequality that 
\begin{align*}
	\|h\|_{H^{1}}^{2} \leq  C\rho^{2} + \frac{1}{2}\|h\|_{H^{1}}^{2}.
\end{align*}
Thus it holds that $\|h\|_{H^{1}}^{2} \lesssim \rho^{2}$ and this completes the proof of Claim 2.

\textbf{Claim 3.} $|\mu(u)| \approx |\rho|$. 

We have
\begin{align*}
	\mu(u)
	=-2\rho \|Q_{\omega,\gamma}\|_{\dot{H}_{\gamma}^{1}}^{2}
	-2 ( Q_{\omega,\gamma} , h )_{\dot{H}_{\gamma}^{1}}
	+\|\rho Q_{\omega,\gamma}+h\|_{\dot{H}_{\gamma}^{1}}^{2}.
\end{align*}
Since $Q_{\omega,\gamma}$ is the solution of the elliptic equation and $\re \int Q_{\omega,\gamma}^{p} \overline{h}dx =0$, we get
\begin{align*}
	-2\langle Q_{\omega,\gamma} , h \rangle_{\dot{H}_{\gamma}^{1}}
	=2\omega \re   \int_{\mathbb{R}}  Q_{\omega,\gamma} h dx.
\end{align*}
By \eqref{eqqq}, we obtain 
\begin{align*}
	2\omega \re   \int_{\mathbb{R}}  Q_{\omega,\gamma} h dx
	=-2\rho \omega \|Q_{\omega,\gamma}\|_{L^{2}}^{2} - \omega \|\rho Q_{\omega,\gamma}+h\|_{L^{2}}^{2}.
\end{align*}
Therefore, it holds that
\begin{align*}
	\mu(u)
	= 2(-\|Q_{\omega,\gamma}\|_{\dot{H}_{\gamma}^{1}}^{2}
	-\omega \|Q_{\omega,\gamma}\|_{L^{2}}^{2})\rho  - \omega \|\rho Q_{\omega,\gamma}+h\|_{L^{2}}^{2}
	+\|\rho Q_{\omega,\gamma}+h\|_{\dot{H}_{\gamma}^{1}}^{2}
\end{align*}
and the coefficient  $-\|Q_{\omega,\gamma}\|_{\dot{H}_{\gamma}^{1}}^{2}
-\omega \|Q_{\omega,\gamma}\|_{L^{2}}^{2}$ of $\rho$ is not zero. 
This means that
\begin{align*}
	|\mu(u)| \lesssim |\rho| + o(|\rho|^{2}+\|h\|_{H^{1}}^{2})
\end{align*}
and thus $|\mu(u)| \lesssim |\rho|$ since we have $|\rho| \approx \|h\|_{H^{1}}$ and they are small. Moreover, we also have
\begin{align*}
	 |\rho| \lesssim |\mu(u)| + o(|\rho|^{2}+\|h\|_{H^{1}}^{2})
\end{align*}
and thus $ |\rho| \lesssim |\mu(u)|$. This completes the proof of Claim 3. 

By combining those claims, the proof is done. 
\end{proof}

We also have estimates of the derivatives of parameters.

\begin{lemma}
\label{lem4.4}
We have
\begin{align*}
	|\rho'(t)|+ |\theta'(t)| \lesssim |\mu(u(t))|
\end{align*}
for $t \in I_{\mu_{0}}$. 
\end{lemma}

\begin{proof}
Since $h(t)=e^{-i\theta(t)-i\omega t}u(t)-(1+\rho(t))Q_{\omega,\gamma}$, $u$ satisfies the equation \eqref{NLS}, and $Q_{\omega,\gamma}$ satisfies the elliptic equation, direct calculation shows that
\begin{align*}
	i \partial_{t} h +\partial_{x}^{2} h +\gamma \delta h
	&=\theta' (Q_{\omega,\gamma}+g) + \omega g 
	\\
	&\quad +|Q_{\omega,\gamma}+g|^{p-1}(Q_{\omega,\gamma}+g) - |Q_{\omega,\gamma}|^{p-1}Q_{\omega,\gamma}
	\\
	&\quad -i\rho' Q_{\omega,\gamma} - \rho (\omega Q_{\omega,\gamma} 
	-|Q_{\omega,\gamma}|^{p-1}Q_{\omega,\gamma}).
\end{align*}
Testing this with $Q_{\omega,\gamma}$ and taking the real part, we obtain
\begin{align*}
	|\theta'(t)| \lesssim |\rho(t)|+\|g(t)\|_{H^{1}}+\|h\|_{H^{1}}
	\lesssim |\mu(u(t))|
\end{align*}
since the orthogonality condition $\im \int_{\mathbb{R}}Q_{\omega,\gamma} h dx=0$ implies $\im \int_{\mathbb{R}}Q_{\omega,\gamma} \partial_{t}h dx=0$. 
Testing this with $Q_{\omega,\gamma}^{p}$ and taking the imaginary part, we obtain
\begin{align*}
	|\rho'(t)| \lesssim |\rho(t)|+\|g(t)\|_{H^{1}}+\|h\|_{H^{1}}
	\lesssim |\mu(u(t))|
\end{align*}
since the condition $\re \int_{\mathbb{R}}Q_{\omega,\gamma}^{p} h dx=0$ implies $\re \int_{\mathbb{R}}Q_{\omega,\gamma}^{p} \partial_{t} h dx=0$. 
\end{proof}


\subsection{Convergence to the ground state}
\label{sec4.2}

\subsubsection{Case of positive $K_{\gamma}$}
\label{sec4.2.1}

We first consider the case that $K_{\gamma}(u_{0})$ is positive. 
In this case, we have $\mu(u(t))>0$. 
We show that a solution which does not scatter converges to the ground state. 

\begin{proposition}
\label{prop4.5}
Let $u$ be the solution to \eqref{NLS} with $u(0)=u_0$. Assume that $u_0\in H_\even^1(\mathbb{R})$ satisfies \eqref{ME} and $K_{\gamma}(u_{0})>0$. 
Furthermore, we assume that the solution $u$ satisfies $\|u\|_{S(0,\infty)}=\infty$. 
Then there exist $\theta_{0} \in \mathbb{R}$ and $c>0$ such that 
\begin{align*}
	\|u(t) - e^{i\theta_{0}} e^{i\omega t}Q_{\omega,\gamma}\|_{H^{1}} \lesssim e^{-ct}
\end{align*}
for $t>0$. Moreover, $u$ scatters in the negative time direction. 
\end{proposition}

\begin{lemma}
\label{lem4.6.0}
Under the assumption of Proposition \ref{prop4.5}, 
$\{u(t): t >0\}$ is precompact in $H_\even^{1}(\mathbb{R})$. 
\end{lemma}

\begin{proof}
The proof is similar to that of Lemma 3.11 in Ikeda and Inui \cite{IkIn17} so we omit it.
\end{proof}

By the compactness, we find that, for any $\varepsilon>0$, there exists $R=R_{\varepsilon}>0$ such that 
\begin{align}
\label{eq4.a}
	\int_{|x|>R} |\partial_{x}u(t,x)|^{2} + |u(t,x)|^{2} dx < \varepsilon
\end{align}
for all $t>0$.

\begin{lemma}
Under the same assumption of Proposition \ref{prop4.5},
 we have 
\begin{align*}
	\lim_{T\to \infty} \frac{1}{T} \int_{0}^{T} \mu(u(t))dt =0.
\end{align*}
\end{lemma}

\begin{proof}
This can be shown in the same way as in \cite[Lemma 6.5]{DuRo10} and  \cite[Lemma A.11]{CFR20}, where we note that we use Proposition \ref{prop2.7}. We omit the proof. 
\end{proof}

By this lemma, we immediately obtain the following corollary.  

\begin{corollary}
\label{cor4.9}
Under the assumption of Proposition \ref{prop4.5}, there exists a time sequence $\{t_{n}\}_{n \in \mathbb{N}}$ with $t_{n} \to \infty$ such that
\begin{align*}
	\lim_{n\to \infty} \mu(u(t_{n})) =0. 
\end{align*}
\end{corollary}


\begin{lemma}
\label{lem4.10}
Under the assumption of Proposition \ref{prop4.5}, we have
\begin{align*}
	\int_{\tau_{1}}^{\tau_{2}} \mu(u(t)) dt \lesssim \mu(u(\tau_{1})) + \mu(u(\tau_{2}))
\end{align*}
for any $0 < \tau_{1} < \tau_{2}$. 
\end{lemma}

\begin{proof}
The proof is similar to \cite[Lemma 6.7]{DuRo10} and \cite[Lemma A.12]{CFR20}. However, we give the precise proof, which is easier here because of the lack of the spatial translation parameter. 

\textbf{Claim 1.} We have $|J_{R}'(t)| \lesssim R \mu(u(t))$ for all $t>0$. 

When $\mu(u(t))> \mu_{0}$, we obtain
\begin{align*}
	|J_{R}'(t)| \lesssim R\int_{\mathbb{R}} |u(t,x)||\partial_{x}u(t,x)|dx 
	\lesssim R \|u\|_{L_{t}^{\infty}H^{1}}^{2} \lesssim R \lesssim R \frac{\mu(u(t))}{\mu_{0}}. 
\end{align*}
Next we consider the case of $\mu(u(t))\leq \mu_{0}$. 
We have
\begin{align*}
	&|J_{R}'(t)|
	\\
	&= 2R \left| \im \int_{\mathbb{R}} \partial_{x}\varphi \left( \frac{x}{R}\right)  \{ \overline{u(t,x)} \partial_{x} u(t,x) - \overline{e^{i\theta(t)+i\omega t}Q_{\omega,\gamma}(x)}\partial_{x}(e^{i\theta(t)+i\omega t}Q_{\omega,\gamma}(x)) \}dx\right|
	\\
	&\lesssim R (\|u(t)\|_{H^{1}}+\|Q_{\omega,\gamma}\|_{H^{1}})\|u(t)-e^{i\theta(t)+i\omega t}Q_{\omega,\gamma}(x)\|_{H^{1}}
	\\
	&\lesssim R\mu(u(t)), 
\end{align*}
where we used the fact that $Q_{\omega,\gamma}$ is real valued and Proposition \ref{prop4.2}. This completes the proof of Claim 1. 

\textbf{Claim 2.} For sufficiently small $\varepsilon>0$, there exists $R>0$ such that $|A_{R}(u(t))|\lesssim \varepsilon \mu(u(t))$ for all $t>0$. 

Let $\varepsilon>0$ be smaller than $\min\{\mu_{0},1\}$. 
By the compactness \eqref{eq4.a}, we obtain $|A_{R}(u(t))| \lesssim \varepsilon^{2}$ for sufficiently large $R>0$. If $\mu(u(t)) > \varepsilon$ then we have
\begin{align*}
	|A_{R}(u(t))| \lesssim \varepsilon^{2} \frac{\mu(u(t))}{\varepsilon} =\varepsilon\mu(u(t)).
\end{align*}
Next we consider the case of $\mu(u(t)) \leq \varepsilon (< \mu_{0})$. 
We have
\begin{align*}
	A_{R}(u(t))=A_{R}(u(t)) - A_{R}(e^{i\omega t} e^{i\theta(t)}Q_{\omega,\gamma})
\end{align*}
since $A_{R}(e^{i\omega t} e^{i\theta(t)}Q_{\omega,\gamma})=0$ by Lemma \ref{lem2.11}. By direct calculations, we get
\begin{align*}
	&|A_{R}(u(t))| 
	\\
	& \lesssim (\|u(t)\|_{H^{1}(|x|>R)} + \|Q_{\omega,\gamma}\|_{H^{1}(|x|>R)} + \|u(t)\|_{H^{1}(|x|>R)}^{p} + \|Q_{\omega,\gamma}\|_{H^{1}(|x|>R)}^{p})
	\\
	&\quad \times \|u(t)-e^{i\omega t} e^{i\theta(t)}Q_{\omega,\gamma}\|_{H^{1}}
\end{align*}
By the compactness lemma \eqref{eq4.a} and exponential decay of $Q_{\omega,\gamma}$, there exists $R>0$ such that 
\begin{align*}
	\|u(t)\|_{H^{1}(|x|>R)} + \|Q_{\omega,\gamma}\|_{H^{1}(|x|>R)} + \|u(t)\|_{H^{1}(|x|>R)}^{p} + \|Q_{\omega,\gamma}\|_{H^{1}(|x|>R)}^{p} \lesssim \varepsilon
\end{align*}
for all $t>0$ and thus we get
\begin{align*}
	|A_{R}(u(t))| \lesssim \varepsilon \mu(u(t)). 
\end{align*}
for all $t>0$. Combining these estimates, in all cases, we have
\begin{align*}
	|A_{R}(u(t))| \lesssim \varepsilon \mu(u(t))
\end{align*}
by taking $R$ sufficiently large. This shows Claim 2. 

By Proposition \ref{prop2.7} and Claims 1 and 2, we obtain
\begin{align*}
	&\int_{\tau_{1}}^{\tau_{2}}\mu(u(t))dt 
	\lesssim \int_{\tau_{1}}^{\tau_{2}}( c \mu(u(t)) -\varepsilon \mu(u(t)) )dt \lesssim \int_{\tau_{1}}^{\tau_{2}}( 8K_{\gamma}(u(t))+A_{R}(u(t)))dt 
	\\
	&=\int_{\tau_{1}}^{\tau_{2}} J_{R}''(t)dt
	\leq |J_{R}'(\tau_{1})| +|J_{R}'(\tau_{2})|
	\lesssim R\{ \mu(u(\tau_{1})) + \mu(u(\tau_{2}))\}. 
\end{align*}
\end{proof}

\begin{lemma}
Under the assumption of Proposition \ref{prop4.5}, we have that there exists $c>0$ such that
\begin{align*}
	\int_{t}^{\infty} \mu(u(s)) ds \lesssim e^{-ct}
\end{align*}
\end{lemma}

\begin{proof}
This follows from Lemma \ref{lem4.10} and Corollary \ref{cor4.9}. See the end of the proof of Proposition 6.1 (below (6.28)) in \cite{DuRo10} for the detailed proof. 
\end{proof}

\begin{lemma}
\label{lem4.12}
We have
\begin{align*}
	\lim_{t\to \infty} \mu(u(t))=0
\end{align*}
\end{lemma}

\begin{proof}
Suppose that the statement fails. We know that there exists a time sequence $\{t_{n}\}$ such that $\lim_{n \to \infty} \mu(u(t_{n})) =0$ by Corollary \ref{cor4.9}. Thus, there exist a sequence $\{\widetilde{t}_{n}\}$ and $\varepsilon_{0}\in(0,\mu_{0})$ such that 
\begin{align*}
	t_{n} < \widetilde{t}_{n},
	\mu(u(\widetilde{t}_{n})) = \varepsilon_{0},
	\mu(u(t)) < \varepsilon_{0} \text{ for all } t \in [t_{n},\widetilde{t}_{n}).
\end{align*}
By $\varepsilon_{0}<\mu_{0}$, the parameters in modulation argument are well-defined in the interval $[t_{n},\widetilde{t}_{n}]$. Lemma \ref{lem4.4} implies that $|\rho'(t)|\lesssim \mu(u(t))$ and thus 
\begin{align*}
	|\rho(\widetilde{t}_{n})-\rho(t_{n})|=\int_{t_{n}}^{\widetilde{t}_{n}} |\rho'(t)| dt \lesssim \int_{t_{n}}^{\widetilde{t}_{n}}  \mu(u(t)) dt \lesssim e^{-ct_{n}}.
\end{align*}
Hence we get
\begin{align*}
	\lim_{n \to \infty} |\rho(\widetilde{t}_{n})-\rho(t_{n})|=0.
\end{align*}
Since we have $\mu(u(t))\approx|\rho(t)|$ by Proposition \ref{prop4.2}, we get
\begin{align*}
	\mu(u(\widetilde{t}_{n})) \approx |\rho(\widetilde{t}_{n})|
	\leq |\rho(\widetilde{t}_{n})-\rho(t_{n})| + |\rho(t_{n})|
	\lesssim |\rho(\widetilde{t}_{n})-\rho(t_{n})| + \mu(u(t_{n}))
	\to 0
\end{align*}
as $n \to \infty$. This contradicts $\mu(u(\widetilde{t}_{n})) = \varepsilon_{0}$. The proof is completed.
\end{proof}

We are ready to prove Proposition \ref{prop4.5}. The proof is similar to Lemma 4.4. in \cite{DuRo10} and Lemma 5.5 in \cite{CFR20}.  However, we give the proof for the reader's convenience. 

\begin{proof}[Proof of Proposition \ref{prop4.5}]
By Proposition \ref{prop4.2} and Lemma \ref{lem4.12}, we get $|\rho(t)| \approx \mu(u(t)) \to 0$ as $t\to \infty$. Thus it follows from this and Lemma \ref{lem4.4} that
\begin{align*}
	|\rho(t)| \leq \int_{t}^{\infty} |\rho'(s)| ds \lesssim \int_{t}^{\infty} \mu(u(s)) ds \lesssim e^{-ct}.
\end{align*}
Thus we get $\|h(t)\|_{H^{1}} \approx \|g(t)\|_{H^{1}} \to 0$ as $t \to \infty$ by Proposition \ref{prop4.2}. 
Moreover, Proposition \ref{prop4.2}, $\mu(u(t))\approx |\rho(t)|$, and Lemma \ref{lem4.4} give that
\begin{align*}
	|\theta'(t)|\lesssim \mu(u(t)) \lesssim  e^{-ct}
\end{align*}
This implies that there exists $\theta_{0}$ such that $\theta(t) \to \theta_{0}$ as $t \to \infty$. Namely we get
\begin{align*}
	\|u(t) - e^{i\theta_{0}} e^{i\omega t}Q_{\omega,\gamma}\|_{H^{1}}
	&\leq \|u(t) - e^{i\theta(t)} e^{i\omega t}Q_{\omega,\gamma}\|_{H^{1}}
	+\|(e^{i\theta(t)} -e^{i\theta_{0}})e^{i\omega t}Q_{\omega,\gamma}\|_{H^{1}}
	\\
	&=\|g(t)\|_{H^{1}}
	+|e^{i\theta(t)} -e^{i\theta_{0}}|\|Q_{\omega,\gamma}\|_{H^{1}}
	\to 0
\end{align*}
as $t \to \infty$. At last, we show that $u$ scatters in the negative time direction. If not, by the above argument and time-reversibility, we also have
\begin{align*}
	\lim_{t\to -\infty} \mu(u(t))=0.
\end{align*}
Thus, applying Lemma \ref{lem4.10} as $\tau_{1}=-n$ and $\tau_{2}=n$, we obtain
\begin{align*}
	0 < \int_{-\infty}^{\infty} \mu(u(t)) dt = \lim_{n \to \infty} \int_{-n}^{n} \mu(u(t)) dt \lesssim  \lim_{n \to \infty}\{ \mu(u(-n)) + \mu(u(n)) \}=0.
\end{align*}
since $\mu(u(t))>0$ for all $t\in \mathbb{R}$. This is a contradiction. 
\end{proof}


\subsubsection{Case of negative $K_{\gamma}$}
\label{sec4.2.2}

Next, we consider the case that $K_{\gamma}(u_{0})$ is negative. 
In this case, we have $\mu(u(t))<0$. 
We show that a global solution converges to the ground state. 

\begin{proposition}
\label{prop4.13}
Let $u$ be the solution to \eqref{NLS} with $u(0)=u_0$. Assume that $u_0 \in H_\even^1(\mathbb{R})$ satisfies \eqref{ME} and $K_{\gamma}(u_{0})<0$. Furthermore, we assume that the solution is global in positive time and $\int_{\mathbb{R}}|xu_{0}(x)|^{2}dx<\infty$. 
Then there exist $\theta_{0} \in \mathbb{R}$ and $c>0$ such that 
\begin{align*}
	\|u(t) - e^{i\theta_{0}} e^{i\omega t}Q_{\omega,\gamma}\|_{H^{1}} \lesssim e^{-ct}
\end{align*}
for $t>0$. Moreover, $u$ blows up in  finite negative time. 
\end{proposition}

\begin{lemma}
\label{lem4.14}
Let $\phi \in C^{1}(\mathbb{R})$ with $\phi(0)=0$ and $f\in H^{1}(\mathbb{R})$. We assume that $\int_{\mathbb{R}} |\partial_{x}\phi(x)|^{2}|f(x)|^{2}dx<\infty$, $M(f)=M(Q_{\omega,\gamma})$, and $E_{\gamma}(f)=E_{\gamma}(Q_{\omega,\gamma})$. Then we have the following inequality:
\begin{align*}
	\left(\im \int_{\mathbb{R}} \partial_{x}\phi(x)\partial_{x}f(x)\overline{f(x)} dx \right)^{2} \lesssim |\mu(f)|^{2} \int_{\mathbb{R}} |\partial_{x}\phi(x)|^{2} |f(x)|^{2} dx.
\end{align*}
\end{lemma}

\begin{proof}
The proof is similar to \cite[Claim 5.4]{DuRo10} (see also \cite[Lemma 2.1]{Ban04} and \cite[Claims A.2 and A.4]{CFR20}). However, we give the proof since we use a different type of Gagliardo--Nirenberg inequality. 

By the Gagliardo--Nirenberg type inequality with respect to $\delta$ potential, Lemma \ref{GN}, we have
\begin{align*}
	\|f\|_{L^{p+1}}^{2} = \|e^{ i \lambda \phi}f\|_{L^{p+1}}^{2} \leq C_{\omega,\gamma} \|e^{ i \lambda \phi}f\|_{H_{\omega,\gamma}^{1}}^{2}
	=C_{\omega,\gamma} (\|e^{ i \lambda \phi}f\|_{\dot{H}_{\gamma}^{1}}^{2}+\omega \|f\|_{L^{2}}^{2})
\end{align*}
for $\lambda \in \mathbb{R}$. 
Since we have
\begin{align*}
	\|e^{ i \lambda \phi}f\|_{\dot{H}_{\gamma}^{1}}^{2}
	&=
	\lambda^{2}\|(\partial_{x}\phi) f\|_{L^{2}}^{2} +2\lambda \im \int_{\mathbb{R}} (\partial_{x}\phi) \partial_{x}f \overline{f}dx
	+ \|f\|_{\dot{H}_{\gamma}^{1}}^{2},
\end{align*}
where we used $\phi(0)=0$, we get 
\begin{align*}
	\lambda^{2}\|(\partial_{x}\phi) f\|_{L^{2}}^{2} +2\lambda \im \int_{\mathbb{R}} (\partial_{x}\phi) \partial_{x}f \overline{f}dx
	+\|f\|_{H_{\omega,\gamma}^{1}}^{2} - C_{\omega,\gamma}^{-1} \|f\|_{L^{p+1}}^{2} \geq 0.
\end{align*}
Since $\lambda\in \mathbb{R}$ is arbitrary, it must hold that
\begin{align*}
	\left( \im \int_{\mathbb{R}} (\partial_{x}\phi) \partial_{x}f \overline{f}dx\right)^{2} \leq \|(\partial_{x}\phi) f\|_{L^{2}}^{2}
	( \|f\|_{H_{\omega,\gamma}^{1}}^{2} - C_{\omega,\gamma}^{-1} \|f\|_{L^{p+1}}^{2}).
\end{align*}
We show that $\|f\|_{H_{\omega,\gamma}^{1}}^{2} - C_{\omega,\gamma}^{-1} \|f\|_{L^{p+1}}^{2}=O(\mu(f)^{2})$. By the definition of $\mu$ and $M(f)=M(Q_{\omega,\gamma})$, we have
\begin{align*}
	 \|f\|_{H_{\omega,\gamma}^{1}}^{2} - C_{\omega,\gamma}^{-1} \|f\|_{L^{p+1}}^{2}
	=-\mu(f) + \|Q_{\omega,\gamma}\|_{H_{\omega,\gamma}^{1}}^{2} - C_{\omega,\gamma}^{-1} \|f\|_{L^{p+1}}^{2}.
\end{align*}
By the energy condition $E_{\gamma}(f)=E(Q_{\omega,\gamma})$, it holds that
\begin{align*}
	\|f\|_{L^{p+1}}^{p+1} 
	= \frac{p+1}{2} \|f\|_{\dot{H}_{\gamma}^{1}}^{2}
	-E(Q_{\omega,\gamma})
	=\|Q_{\omega,\gamma}\|_{L^{p+1}}^{p+1} - \frac{p+1}{2} \mu(f).
\end{align*}
Therefore, by the Taylor expansion $x^{\frac{2}{p+1}}=a^{\frac{2}{p+1}}+\frac{2}{p+1}a^{\frac{2}{p+1}-1}(x-a)+O(|x-a|^{2})$ with $a=\|Q_{\omega,\gamma}\|_{L^{p+1}}^{p+1}$ and $x-a=-\frac{p+1}{2} \mu(f)$, we obtain
\begin{align*}
	\|f\|_{L^{p+1}}^{2} &= \left(\|Q_{\omega,\gamma}\|_{L^{p+1}}^{p+1} - \frac{p+1}{2} \mu(f)\right)^{\frac{2}{p+1}}
	\\
	&=\|Q_{\omega,\gamma}\|_{L^{p+1}}^{2}- \|Q_{\omega,\gamma}\|_{L^{p+1}}^{-(p-1)} \mu(f)+O(|\mu(f)|^{2})
\end{align*}
From these equalities and $C_{\omega,\gamma}^{-1}=\|Q_{\omega,\gamma}\|_{L^{p+1}}^{p-1}$, it holds that
\begin{align*}
	&-\mu(f) + \|Q_{\omega,\gamma}\|_{H_{\omega,\gamma}^{1}}^{2}  - C_{\omega,\gamma}^{-1} \|f\|_{L^{p+1}}^{2}
	=O(|\mu(f)|^{2}).
\end{align*}
This completes the proof. 
\end{proof}

\begin{corollary}
\label{cor4.15}
Let $u_{0} \in H^{1}(\mathbb{R})$. We assume that $K_{\gamma}(u_{0})<0$, $\int_{\mathbb{R}} |x|^{2}|u_{0}(x)|^{2}dx<\infty$, $M(u_{0})=M(Q_{\omega,\gamma})$, and $E_{\gamma}(u_{0})=E_{\gamma}(Q_{\omega,\gamma})$. Then we have the following inequality:
\begin{align*}
	\left(\im \int_{\mathbb{R}} x\partial_{x}u(t,x)\overline{u(t,x)} dx \right)^{2} \lesssim |K_{\gamma}(u(t))|^{2} \int_{\mathbb{R}} |x|^{2} |u(t,x)|^{2} dx.
\end{align*}
\end{corollary}

\begin{proof}
It holds that $K_{\gamma}(u(t))< c \mu(u(t))<0$ by Proposition \ref{prop2.7}. By this and applying Lemma \ref{lem4.14} with $\phi(x)=x^{2}$ and $f=u$, we obtain the statement. 
\end{proof}

\begin{lemma}
\label{lem4.16}
Under the assumption of Proposition \ref{prop4.13}, 
we have
\begin{align*}
	\im \int_{\mathbb{R}} x\partial_{x}u(t,x) \overline{u(t,x)}dx>0
\end{align*}
for the time of existence and there exists $c>0$ such that 
\begin{align*}
	-\int_{t}^{\infty} \mu(u(s))ds \lesssim e^{-ct}
\end{align*}
for all $t >0$. 
\end{lemma}

\begin{proof}
The proof is same as proofs of \cite[Lemma 5.3]{DuRo10} and \cite[Lemma A.3]{CFR20} so we omit it.  
\end{proof}

\begin{lemma}
\label{lem4.17}
Under the assumption of Proposition \ref{prop4.13}, we have
\begin{align*}
	\lim_{t \to \infty} \mu(u(t))=0
\end{align*}
\end{lemma}

\begin{proof}
By Lemma \ref{lem4.16}, there exists a time sequence $\{t_{n}\}$ with $t_{n} \to \infty$ such that $\mu(u(t_{n})) \to 0$ as $n \to \infty$. The rest of the proof is same as in Lemma \ref{lem4.12} above. 
\end{proof}

We are ready to prove Proposition \ref{prop4.13}. 

\begin{proof}[Proof of Proposition \ref{prop4.13}]
The convergence can be shown in the same way as in the proof of Proposition \ref{prop4.5}. 
We show the blow-up in the negative time direction. Suppose that $u$ exists on $\mathbb{R}$. Let $v(t):=\overline{u(-t)}$. Then we have
\begin{align*}
	\im \int_{\mathbb{R}} x \partial_{x}u_{0}(x)\overline{u_{0}(x)}dx > 0
\end{align*}
by Lemma \ref{lem4.16}. Applying Lemma \ref{lem4.16} to the time reversed solution $v$, we obtain 
\begin{align*}
	\im \int_{\mathbb{R}} x \partial_{x}v_{0}(x)\overline{v_{0}(x)}dx > 0.
\end{align*}
On the other hand, by the definition of $v$, we have
\begin{align*}
	\im \int_{\mathbb{R}} x \partial_{x}u_{0}(x)\overline{u_{0}(x)}dx
	=- \im \int_{\mathbb{R}} x \partial_{x}v_{0}(x)\overline{v_{0}(x)}dx.
\end{align*}
This is a contradiction, completing the proof. 
\end{proof}


\section{Non-scattering solutions are the corresponding special solutions}
\label{sec5}

The aim of this section is to show the following statements.

\begin{proposition}
\label{prop5.1}
If a solution $u$ satisfies \eqref{ME} and
\begin{align*}
	\|u(t)-e^{i\omega t}Q_{\omega,\gamma}\|_{H^{1}} \lesssim e^{-ct}
\end{align*}
for some $c>0$, then there exists $A \in \mathbb{R}$ such that $u=U^{A}$, where $U^{A}$ is the solution constructed in Proposition \ref{prop3.3}. 
\end{proposition}

\begin{corollary}
\label{cor5.2}
For any $A \neq 0$, there exists $T_{A} \in \mathbb{R}$ such that 
\begin{align*}
	U^{A}(t)=Q^{+}(t-T_{A}) \text{ if }A>0 
\end{align*} 
or 
\begin{align*}
	U^{A}(t)=Q^{-}(t-T_{A})\text{ if }A<0,
\end{align*}
where $Q^{+}=U^{1}$ and $Q^{-}=U^{-1}$. 
\end{corollary}

\subsection{Estimates on exponential decaying solutions}

Let $h$ satisfy the following approximate equation:
\begin{align*}
	\partial_{t} h + \mathcal{L}_{\omega,\gamma} h = \varepsilon.
\end{align*}
We assume that $h$ and $\varepsilon$ satisfy
\begin{align}
\label{eq5.0}
	\|h(t)\|_{H^{1}} \lesssim e^{-c_{0}t}
	\text{ and }
	\|\varepsilon\|_{L^{1}((t,\infty);H^{1}(\mathbb{R}))}\lesssim e^{-c_{1}t},
\end{align}
where $c_{1}>c_{0}>0$. 
The following arguments rely on the arguments by Duyckaerts and Roudenko \cite[Section 7.1]{DuRo10} and Campos, Farah, and Roudenko \cite[Sections 8.2 and A.4]{CFR20}. Thus, we omit the proofs. 
We note that, though Campos, Farah, and Roudenko \cite{CFR20} used the Strichartz norm for the estimate of the error term $\varepsilon$, we use the $L_{t}^{1}H_{x}^{1}$-norm and thus we do not need the statement of Lemma 8.9 in \cite{CFR20}. 

In this section, the notation $c^{-}$ means an arbitrary real number strictly less than $c$.

\begin{lemma}
\label{lem5.3}
Under the assumption of Proposition \ref{prop5.1}, the following are valid.
\begin{enumerate}
\item If $e_{\omega} \not\in [c_{0},c_{1})$, then we have
\begin{align*}
	\|h(t)\|_{H^{1}}\lesssim e^{-c_{1}^{-}t}.
\end{align*}
\item If $e_{\omega} \in [c_{0},c_{1})$, then we have
\begin{align*}
	\|h(t)-Ae^{-e_{\omega}t}\mathcal{Y}_{+}\|_{H^{1}}\lesssim e^{-c_{1}^{-}t}.
\end{align*}
\end{enumerate}
\end{lemma}

\begin{proof}
See Lemma 8.7 in \cite{CFR20} and also Lemma 7.2 in \cite{DuRo10}. 
We note that the proof is easier than those in \cite{CFR20} and \cite{DuRo10} since we do not have spatial-transform invariance because of the delta potential (for instance, compare $G^{\perp}$ with that in \cite{CFR20}). 
\end{proof}

Let $k\in \mathbb{N}$ and $A \in \mathbb{R}$. 
We study the linearized operator of $Q_{\omega,\gamma}+\mathcal{V}_{k}^{A}$, where $\mathcal{V}_{k}^{A}$ is a function defined in Section \ref{sec3.1}. We define
\begin{align*}
	\widetilde{\mathcal{L}}_{\omega,\gamma}^{k}
	&:=\begin{pmatrix}
	0 & -(\omega -\Delta_{\gamma})
	\\
	\omega -\Delta_{\gamma} & 0
	\end{pmatrix}
	+\frac{p+1}{2}|Q_{\omega,\gamma}+\mathcal{V}_{k}^{A}|^{p-1}
	\begin{pmatrix}
	0 & 1
	\\
	-1 & 0
	\end{pmatrix}
	\\
	&\quad +\frac{p-1}{2}|Q_{\omega,\gamma}+\mathcal{V}_{k}^{A}|^{p-3}
	\begin{pmatrix}
	\im \{(Q+\mathcal{V}_{k}^{A})^{2}\} & -\re \{(Q+\mathcal{V}_{k}^{A})^{2}\}
	\\
	-\re \{(Q+\mathcal{V}_{k}^{A})^{2}\} & -\im \{(Q+\mathcal{V}_{k}^{A})^{2}\}
	\end{pmatrix}
\end{align*}
\begin{align*}
	\widetilde{P}_{k}^{A}(h):=\frac{p+1}{2}|Q_{\omega,\gamma}+\mathcal{V}_{k}^{A}|^{p-1}h
	+\frac{p-1}{2}|Q_{\omega,\gamma}+\mathcal{V}_{k}^{A}|^{p-3}(Q_{\omega,\gamma}+\mathcal{V}_{k}^{A})^{2}\overline{h}
\end{align*}
and 
\begin{align*}
	\widetilde{R}_{k}^{A}(h):=
	|Q_{\omega,\gamma}+\mathcal{V}_{k}^{A}|^{p-1}(Q_{\omega,\gamma}+\mathcal{V}_{k}^{A})N((Q_{\omega,\gamma}+\mathcal{V}_{k}^{A})^{-1}h),
\end{align*}
where we recall that $N(z):=|1+z|^{p-1}(1+z)-1 -\frac{p+1}{2}z - \frac{p-1}{2}\overline{z}$.

If $u=e^{i\omega t}(Q_{\omega,\gamma}+\mathcal{V}_{k}^{A}+h)$ is a solution to \eqref{NLS}, then $h$ satisfies
\begin{align*}
	\partial_{t}h +\widetilde{\mathcal{L}}_{\omega,\gamma}^{k} h =i \widetilde{R}_{k}^{A}(h)-\varepsilon_{k}
\end{align*}
or 
\begin{align*}
	i\partial_{t} h + \Delta_{\gamma} h -\omega h + \widetilde{P}_{k}^{A}(h) 
	=-\widetilde{R}_{k}^{A}(h)-i\varepsilon_{k}
\end{align*}
in the form of a Schr\"{o}dinger equation, where $\varepsilon_{k} = O(e^{-(k+1)e_{\omega}t})$ in $\mathscr{S}^{1}(\mathbb{R})$ (see Section \ref{sec3.1}).

\begin{lemma}
We have
\begin{align*}
	&|\mathcal{V}_{k}^{A}| \lesssim e^{-e_{\omega}t} |Q_{\omega,\gamma}|,
	\\
	&|\partial_{x}\mathcal{V}_{k}^{A}| \lesssim e^{-e_{\omega}t} |Q_{\omega,\gamma}|.
\end{align*}
\end{lemma}

\begin{proof}
This follows from the construction of $\mathcal{V}_{k}^{A}$. See the proof in Section \ref{sec3.1}. 
\end{proof}


\begin{lemma}
\label{lem5.5}
Let $I$ be a time interval and $f \in L_{t}^{\infty}H^{1}$. Then we have the following estimates.
\begin{enumerate}
\item $\|\widetilde{P}_{k}^{A}(f)\|_{L_{t}^{1}H^{1}(I)} \lesssim |I|^{a}\|f\|_{L_{t}^{\infty}H^{1}(I)}$ for some $a>0$.
\item $\|\widetilde{R}_{k}^{A}(f)\|_{L_{t}^{1}H^{1}(I)} \lesssim \|f\|_{L_{t}^{\infty}H^{1}(I)}(|I|^{a}\|f\|_{L_{t}^{\infty}H^{1}(I)} + \|f\|_{L_{t}^{\infty}H^{1}(I)}^{p-1})$ for some $a>0$.
\end{enumerate}
\end{lemma}

\begin{proof}
This follows from a similar argument to Lemmas \ref{lem2.23} and \ref{lem2.24}. See also Lemma 8.8 in \cite{CFR20}. 
\end{proof}

\begin{lemma}
\label{lem5.6}
Let $h$ be a solution to
\begin{align*}
	\partial_{t} h + \widetilde{\mathcal{L}}_{\omega,\gamma}^{k} h =\varepsilon,
\end{align*}
where $h$ and $\varepsilon$ satisfy
\begin{align*}
	\|h(t)\|_{H^{1}}\lesssim e^{-c_{0}t} 
	\text{ and }
	\|\varepsilon\|_{L_{t}^{1}H^{1}(t,\infty)} \lesssim e^{-c_{1}t}
\end{align*}
with $e_{\omega}<c_{0}<c_{1}<(k+1)e_{\omega}$. Then it holds that
\begin{align*}
	\|h(t)\|_{H^{1}}\lesssim e^{-c_{1}^{-}t}.
\end{align*}
\end{lemma}

\begin{proof}
See Lemma 8.10 and also Lemma 8.4 in \cite{CFR20}. 
\end{proof}

\begin{proof}[Proof of Proposition \ref{prop5.1}]
By the above lemmas and  the uniqueness of Proposition \ref{prop3.3}, we get the statement. See Lemmas 8.5 and 8.11 in \cite{CFR20} for the detailed proof. 
\end{proof}

\begin{proof}[Proof of Corollary \ref{cor5.2}]
This follows from Propositions \ref{prop3.3} and \ref{prop5.1}. See Corollaries 8.6 and 8.12 in \cite{CFR20} for the detail. 
\end{proof}

\subsection{Proof of Theorem \ref{thm1.5}}

The solution going to the ground state with negative $K_{\gamma}$ is of finite variance.

\begin{lemma}
\label{lem5.7}
Let $u$ be a solution to \eqref{NLS} such that $u$ exists globally in positive time and satisfies \eqref{ME} and $K_{\gamma}(u_{0})<0$. 
If there exists $c>0$ such that
\begin{align*}
	\|u(t)- e^{i\omega t}Q_{\omega,\gamma}\|_{H^{1}} \lesssim e^{-ct}
\end{align*}
for all $t>0$, then the solution $u$ must satisfy $\int |x|^{2}|u_{0}(x)|^{2}dx<\infty$.
\end{lemma}

\begin{proof}
The proof is similar to the proof of Lemma 5.7 in \cite{GuIn22pre}. 
\end{proof}

\begin{remark}
In the one dimensional case, the radial symmetry of $U^{1}$ does not imply the finite variance because the radial Sobolev inequality is not applicable. Thus, we need the above lemma to show the blow-up of $U^{1}$ even when $\gamma=0$. 
\end{remark}

Finally, we are ready to prove Theorem \ref{thm1.5}.

\begin{proof}[Proof of Theorem \ref{thm1.5}]
First, we prove that $(Q_{\omega,\gamma},\mathcal{Y}_{1})_{H_{\omega,\gamma}^{1}}\neq 0$. 
Indeed, if not,  by the equation of $Q_{\omega,\gamma}$, we have
\begin{align*}
	B_{\omega,\gamma}(Q_{\omega,\gamma},\mathcal{Y}_{\pm})
	&=\frac{p+1}{2} (Q_{\omega,\gamma},\mathcal{Y}_{1})_{H_{\omega,\gamma}^{1}} =0
\end{align*}
and thus we find $Q_{\omega,\gamma} \in \widetilde{G}^{\perp}$ since we have
\begin{align*}
	B_{\omega,\gamma}(Q_{\omega,\gamma},\mathcal{Y}_{\pm})= \frac{1}{2}(L_{\omega,\gamma}^{+}Q_{\omega,\gamma},\mathcal{Y}_{1})_{L^{2}}=\frac{1}{2}(Q_{\omega,\gamma},L_{\omega,\gamma}^{+}\mathcal{Y}_{1})_{L^{2}}=\frac{e_{\omega}}{2}(Q_{\omega,\gamma},\mathcal{Y}_{2})_{L^{2}}.
\end{align*} 
Now, we have
\begin{align*}
	\Phi(Q_{\omega,\gamma})=-\frac{p-1}{2}\|Q_{\omega,\gamma}\|_{L^{p+1}}^{p+1} <0.
\end{align*}
This contradicts the fact that $\Phi$ is nonnegative on $\widetilde{G}^{\perp}$. 
We may assume that $(Q_{\omega,\gamma},\mathcal{Y}_{1})_{H_{\omega,\gamma}^{1}} >0$ 
replacing $\mathcal{Y}_{1}$ by $-\mathcal{Y}_{1}$, that is, $\mathcal{Y}_{\pm}$ by $\mathcal{Y}_{\mp}$, if necessary. 

We define $U^{+}:=U^1$ and $U^{-}:=U^{-1}$.
Since $e^{-i\omega t}U^{\pm}(t) \to Q$ in $H^{1}(\mathbb{R})$, we have $M(U^{\pm})=M(Q_{\omega,\gamma})$ and $E_{\gamma}(U^{\pm})=E_{\gamma}(Q_{\omega,\gamma})$. Moreover, we also have
\begin{align*}
	\|U^{\pm}\|_{H_{\omega,\gamma}^{1}}^{2}
	=\|Q_{\omega,\gamma}\|_{H_{\omega,\gamma}^{1}}^{2}
	\pm 2 e^{-2e_{\omega}t}  (Q_{\omega,\gamma},\mathcal{Y}_{1})_{H_{\omega,\gamma}^{1}} 
	+O(e^{-4e_{\omega}t}).
\end{align*}
This shows that 
\begin{align*}
	\|U^{+}\|_{H_{\omega,\gamma}^{1}}^{2} - \|Q_{\omega,\gamma}\|_{H_{\omega,\gamma}^{1}}^{2}>0
\text{ and }
	\|U^{-}\|_{H_{\omega,\gamma}^{1}}^{2} - \|Q_{\omega,\gamma}\|_{H_{\omega,\gamma}^{1}}^{2}<0
\end{align*}
for large $t$. Since $\|U^{\pm}\|_{L^{2}}^{2}=\|Q_{\omega,\gamma}\|_{L^{2}}^{2}$, this means that
\begin{align*}
	\mu(U^{+}(t))<0 \text{ and } \mu(U^{-}(t))>0
\end{align*}
for large $t$. Since this sign is conserved by the flow, these inequalities hold for all existence time. These inequalities and the exponential convergence give that $U^{-}(t)$ exists globally and scatters in the negative time direction by Proposition \ref{prop4.5} and $U^{+}(t)$ blows up in the negative time direction by Proposition \ref{prop4.13}, where we note that $U^{+}$ has finite variance by Lemma \ref{lem5.7}.  This completes the proof. 
\end{proof}


\appendix

\section{Threshold condition}
\label{appA}

In this section, we consider the envelope of the family  of the lines $\{(M,E):E+\frac{\omega}{2}M=r_{\omega,\gamma}\}_{\omega>0}$. 
By the explicit formula of $Q_{\omega,\gamma}$, we have the following lemmas.

\begin{lemma}
\label{lemA.1}
Let $\omega>\gamma^2/4$. 
We have
\begin{align*}
	Q_{\omega,\gamma}(x) =\omega^{\frac{1}{p-1}} Q_{1,\frac{\gamma}{\sqrt{\omega}}}(\sqrt{\omega}x).
\end{align*}
Moreover, setting $\xi=\xi(\omega,\gamma):=\frac{2}{(p-1)\sqrt{\omega}} \tanh^{-1}\left( \frac{\gamma}{2\sqrt{\omega}} \right)$, we have
\begin{align*}
	Q_{\omega,\gamma}(x)=
	\begin{cases}
	Q_{\omega,0}(x+\xi) & \text{ if } x \geq 0,
	\\
	Q_{\omega,0}(x-\xi) & \text{ if } x < 0.
	\end{cases}
\end{align*}
\end{lemma}

Let $F(M,E,\omega):=E+\frac{\omega}{2}M-r_{\omega,\gamma}$ for $M>0$, $E\in \mathbb{R}$, and $\omega>0$.
We define a family $\{C_{\omega}\}_{\omega>0}$ of lines by 
\begin{align*}
	C_{\omega}:=\{(M,E)\in (0,\infty)\times \mathbb{R}: F(M,E,\omega)=0\}.
\end{align*} 
Let us define 
\begin{align*}
	\mathcal{M}(\omega)
	:=
	\begin{cases}
	M(Q_{\omega,\gamma}) & \text{if } \omega>\frac{\gamma^{2}}{4},
	\\
	2M(Q_{\omega,0}) & \text{if } 0<\omega\leq \frac{\gamma^{2}}{4},
	\end{cases}
	\text{ and }
	\mathcal{E}(\omega)
	:=
	\begin{cases}
	E_{\gamma}(Q_{\omega,\gamma}) & \text{if } \omega>\frac{\gamma^{2}}{4},
	\\
	2E_{0}(Q_{\omega,0}) & \text{if } 0<\omega\leq \frac{\gamma^{2}}{4}.
	\end{cases}
\end{align*}
Then we have $F(\mathcal{M}(\omega),\mathcal{E}(\omega),\omega)=0$ for all $\omega>0$. 

\begin{lemma}
The functions $\mathcal{M}$ and $\mathcal{E}$ are continuous on $(0,\infty)$. $\mathcal{M}$ and $\mathcal{E}$ are continuously differentiable on $(0,\gamma^{2}/4) \cup (\gamma^{2}/4,\infty)$. 
\end{lemma}

\begin{proof}
The continuous differentiability for $\omega>\gamma^{2}/4$ follows from the composition of differentiable functions and for $\omega<\gamma^{2}/4$ can be shown easily from the scaling. 

We give a proof of the continuity of $\mathcal{M}$ at $\omega=\gamma^{2}/4$. It is enough to show right continuity. For $\omega>\gamma^2/4$, we have
\begin{align*}
	M(Q_{\omega,\gamma})=\omega^{-\frac{p-5}{2(p-1)}}M(Q_{1,\frac{\gamma}{\sqrt{\omega}}})
\end{align*}
and
\begin{align*}
	M(Q_{1,\frac{\gamma}{\sqrt{\omega}}})
	=2\int_{0}^{\infty} \left|Q_{1,0}(x + \xi(1,\frac{\gamma}{\sqrt{\omega}}))\right|^{2}dx
	=2\int_{\xi(1,\frac{\gamma}{\sqrt{\omega}})}^{\infty} \left|Q_{1,0}(x)\right|^{2}dx.
\end{align*}
Thus, since $\xi(1,\frac{\gamma}{\sqrt{\omega}}) \to -\infty$ as $\omega \to \gamma^{2}/4$, it holds that
\begin{align*}
	M(Q_{\omega,\gamma}) \to
	\left(\frac{\gamma^{2}}{4}\right)^{-\frac{p-5}{2(p-1)}}  2\int_{-\infty}^{\infty}|Q_{1,0}(x)|^{2}dx=\mathcal{M}(\gamma^{2}/4) \text{ as } \omega \to \gamma^{2}/4.
\end{align*}
Noting $|Q_{\omega,\gamma}(0)|^2=|Q_{\omega,0}(\xi)|^2 \to 0$ as $\omega \to \gamma^2/4$, we can show the continuity of $\mathcal{E}$ by the same argument as above. 
\end{proof}

\begin{remark}
\label{rmkA.1}
$\mathcal{M}$ is not differentiable at $\omega=\gamma^{2}/4$. 
\end{remark}

\begin{proof}[Proof of Remark \ref{rmkA.1}]
We show that $\mathcal{M}$ is not differentiable at $\omega=\gamma^{2}/4$. Let $\omega>\gamma^{2}/4$. Setting
\begin{align*}
	\widetilde{\mathcal{M}}(\omega):=
	\begin{cases}
	M(Q_{1,\frac{\gamma}{\sqrt{\omega}}}) & \text{if } \omega>\gamma^{2}/4,
	\\
	2M(Q_{1,0}) & \text{if } 0<\omega\leq \gamma^{2}/4,
	\end{cases}
\end{align*}
we have $\mathcal{M}(\omega)=\omega^{-\frac{p-5}{2(p-1)}}  \widetilde{\mathcal{M}}(\omega)$. Thus it is enough to show that $\widetilde{\mathcal{M}}$ is not differentiable at $\omega=\gamma^{2}/4$. It holds that
\begin{align*}
	\frac{\widetilde{\mathcal{M}}(\omega)-\widetilde{\mathcal{M}}(\gamma^{2}/4)}{\omega-\gamma^{2}/4}
	=-2\frac{\int_{-\infty}^{\xi(1,\frac{\gamma}{\sqrt{\omega}})} |Q_{1,0}(x)|^{2}dx}{\omega-\gamma^{2}/4}.
\end{align*}
Since $Q_{1,0}(x)\approx e^{-|x|}$ and $\xi <0$, we get
\begin{align*}
	\frac{\int_{-\infty}^{\xi(1,\frac{\gamma}{\sqrt{\omega}})} |Q_{1,0}(x)|^{2}dx}{\omega-\gamma^{2}/4}
	\approx \frac{\int_{-\infty}^{\xi(1,\frac{\gamma}{\sqrt{\omega}})} e^{2x}dx}{\omega-\gamma^{2}/4}
	\approx \frac{e^{2\xi(1,\frac{\gamma}{\sqrt{\omega}})} }{4\omega-\gamma^{2}}.
\end{align*}
Here, we have
\begin{align*}
	e^{2\xi(1,\frac{\gamma}{\sqrt{\omega}})}
	=\left( \frac{1+\frac{\gamma}{2\sqrt{\omega}}}{1-\frac{\gamma}{2\sqrt{\omega}}}\right)^{\frac{2}{p-1}}
	=\left( \frac{4\omega-\gamma^{2}}{(2\sqrt{\omega}-\gamma)^{2}}\right)^{\frac{2}{p-1}}
\end{align*}
since $\xi(1,\frac{\gamma}{\sqrt{\omega}})=\frac{2}{p-1} \tanh^{-1}\left( \frac{\gamma}{2\sqrt{\omega}} \right)$ and $\tanh^{-1}x=\frac{1}{2}\log \left(\frac{1+x}{1-x}\right)$ for $x \in (-1,1)$. Therefore, we obtain
\begin{align*}
	\frac{\widetilde{\mathcal{M}}(\omega)-\widetilde{\mathcal{M}}(\gamma^{2}/4)}{\omega-\gamma^{2}/4}
	\approx (4\omega-\gamma^{2})^{\frac{2}{p-1}-1} \to \infty
\end{align*}
as $\omega \to \gamma^{2}/4$ when $p>5$. 
\end{proof}

\begin{lemma}
$\partial_{\omega}\mathcal{M}(\omega)<0$ for $\omega \in (0,\infty)\setminus\{\gamma^{2}/4\}$. 
\end{lemma}

\begin{proof}
The result in the case of $\omega>\gamma^{2}/4$ is already known by Fukuizumi and Jeanjean \cite{FuJe08}. In the case of $\omega<\gamma^{2}/4$, it is trivial by the scaling structure. 
\end{proof}

\begin{lemma}
$\{\mathcal{M}(\omega):\omega>0\}=(0,\infty)$.
\end{lemma}

\begin{proof}
We have $\mathcal{M}(\omega) \to \infty$ as $\omega \to 0$. Since it holds that
\begin{align*}
	M(Q_{\omega,\gamma})=\omega^{-\frac{p-5}{2(p-1)}}M(Q_{1,\frac{\gamma}{\sqrt{\omega}}})
	=2\omega^{-\frac{p-5}{2(p-1)}}\int_{\xi(1,\frac{\gamma}{\sqrt{\omega}})}^{\infty} \left|Q_{1,0}(x)\right|^{2}dx
\end{align*}
and $\xi(1,\frac{\gamma}{\sqrt{\omega}})=\frac{2}{p-1} \tanh^{-1}\left( \frac{\gamma}{2\sqrt{\omega}} \right)\to 0$ as $\omega \to \infty$, we get $\mathcal{M}(\omega) \to 0$. 
\end{proof}

We set $\nu(\omega):=(\mathcal{M}(\omega),\mathcal{E}(\omega))$ for simplicity. 

\begin{lemma}
We have $F(\nu(\omega),\omega)=0$  for $\omega>0$ and $(\partial_{\omega}F)(\nu(\omega),\omega)=0$ for $\omega \in (0,\infty) \setminus \{\gamma^{2}/4\}$. 
\end{lemma}

\begin{proof}
The first statement is trivial. We prove the second statement. When $\omega>\gamma^{2}/4$, it holds that
\begin{align*}
	(\partial_{\omega}F)(\nu(\omega),\omega)
	=\frac{1}{2}M(Q_{\omega,\gamma})-\partial_{\omega}r_{\omega,\gamma}=\frac{1}{2}M(Q_{\omega,\gamma}) -\partial_{\omega}\{ S_{\omega,\gamma}(Q_{\omega,\gamma})\}=0
\end{align*}
since $Q_{\omega,\gamma}$ is a solution of the elliptic equation $E_{\gamma}'(Q_{\omega,\gamma})+\frac{\omega}{2}M'(Q_{\omega,\gamma})=0$. When $\omega < \gamma^{2}/4$, it holds that
\begin{align*}
	(\partial_{\omega}F)(\nu(\omega),\omega)
	=2\left\{\frac{1}{2}M(Q_{\omega,0})- \partial_{\omega}(S_{\omega,0}(Q_{\omega,0}))\right\}
	=0
\end{align*}
since $Q_{\omega,0}$ is a solution of the elliptic equation with $\gamma=0$.
\end{proof}

\begin{lemma}
$\{\nu(\omega):\omega>0\}$ is the envelope of $\{C_{\omega}\}_{\omega>0}$ and is a continuous curve. The curve is $C^{1}$ except for $\omega=\gamma^{2}/4$. 
\end{lemma}

\begin{proof}
Let $G:=\begin{bmatrix}F(M,E,\omega) \\ \partial_{\omega}F(M,E,\omega) \end{bmatrix}$ for $\omega \neq \gamma^{2}/4$. We note that $G$ is $C^{1}$ in $(0,\infty)\times \mathbb{R} \times (0,\infty)\setminus\{\gamma^{2}/4\}$. Now we have
\begin{align*}
	\det \frac{\partial G}{\partial(M,E)}=\det
	\begin{bmatrix}
	\frac{\omega}{2} & 1
	\\
	\frac{1}{2} & 0
	\end{bmatrix}
	\neq 0
\end{align*}
and $G(\nu(\omega),\omega)=0$. By the implicit function theorem, $\nu(\omega)$ is a unique $C^{1}$-function in $\omega\neq\gamma^{2}/4$ satisfying $G(\nu(\omega),\omega)=0$. 
\end{proof}

Since $\partial_{\omega}\mathcal{M}<0$ for $\omega \neq \gamma^{2}/4$, there exists a function $\phi:(0,\infty)\to \mathbb{R}$ such that $\mathcal{E}(\omega)=\phi(\mathcal{M}(\omega))$. 

%
%
%

\begin{lemma}
The function $\phi$ is strictly convex.
\end{lemma}

\begin{proof}
Since $\nu(\omega)$ is a tangent point on $E=\phi(M)$ with the line $C_{\omega}$, we have $\phi'(\mathcal{M}(\omega))=-\frac{\omega}{2}$. This means that $\phi'$ is increasing in $M$ since $\partial_{\omega}\mathcal{M}<0$ and $\omega=\mathcal{M}^{-1}$ is decreasing in $M$. This implies that $\phi$ is strictly convex.
\end{proof} 


By this convexity, we know that the line $C_{\omega}$ is less than $E=\phi(M)$ except for the tangent point $(\mathcal{M}(\omega),\mathcal{E}(\omega))$. This gives us the following lemma:

\begin{lemma}
The following are equivalent.

(1). $f$ satisfies $S_{\omega,\gamma}(f) \leq r_{\omega,\gamma}$ for some $\omega > 0$.

(2). $f$ satisfies $S_{\omega',\gamma}(f) < S_{\omega',\gamma}(Q_{\omega',\gamma})$ for some $\omega' > 0$ or $M(f)=M_{\omega}$ and $E_{\gamma}(f)=E_{\omega}$ for some $\omega>0$.
\end{lemma}

\begin{corollary}
If $f$ satisfies $S_{\omega,\gamma}(f) = r_{\omega,\gamma}$ for some $\omega > \gamma^{2}/4$, then $f$ satisfies $S_{\omega',\gamma}(f) < r_{\omega',\gamma}$ for some $\omega' > 0$ or $M(f)=M(Q_{\omega,\gamma})$ and $E_{\gamma}(f)=E_{\gamma}(Q_{\omega,\gamma})$.
\end{corollary}

By this corollary, in order to show the main theorem, it is enough to consider $M(f)=M(Q_{\omega,\gamma})$ and $E_{\gamma}(f)=E_{\gamma}(Q_{\omega,\gamma})$ for some $\omega>\gamma^{2}/4$. 

\begin{remark}
When $0<\omega \leq \gamma^2/4$, i.e., $M>\mathcal{M}(\gamma^2/4)$, the envelope is explicitly given by $M^{1-s_c}E^{s_c} = 2 M(Q_{1,0})^{1-s_c}E(Q_{1,0})^{s_c}$ by the scaling structure. However, the formula when $\omega>\gamma^2/4$ is not clear. 
\end{remark}

\section{Cut-off Schwartz class}
\label{appB}


\begin{lemma}
Let $f \in C^{\infty}(\mathbb{R}\setminus\{0\})$. Then the following are equivalent:
\begin{enumerate}
\item There exists $\psi \in BC_0^{\infty}$ with $\psi(x)=1$ on $(-\infty,R) \cup (R,\infty)$ for some $R>0$ such that $\psi f \in \mathcal{S}(\mathbb{R})$.
\item $f \in \widetilde{\mathcal{S}}(\mathbb{R})$. 
\end{enumerate}
\end{lemma}

\begin{proof}
It is trivial that (2) implies (1) since $\psi \in BC_0^{\infty}(\mathbb{R})$. We show that (1) implies (2). Take an arbitrary function $\varphi \in BC_0^{\infty}(\mathbb{R})$. We denote the interval including $0$ and $\varphi=0$ (resp. $\psi$) by $I_\varphi$  (resp. $I_\psi$). 
We take a function $\widetilde{\psi}\in C_0^{\infty}(\mathbb{R})$ satisfying $\widetilde{\psi}(x)=1$ on an interval $(-\infty,R) \cup (R,\infty)$ including $\supp \varphi \cup \supp \psi$ and $\widetilde{\psi}(x)=0$ on an open interval included in $I_\varphi \cap I_\psi$. Then we have
\begin{align*}
	\widetilde{\psi}f &= \widetilde{\psi} \psi f +  \widetilde{\psi} (1-\psi) f
	= \psi f +  \widetilde{\psi} (1-\psi) f.
\end{align*}
Since $\psi f \in \mathcal{S}(\mathbb{R})$ by (1) and $\widetilde{\psi} (1-\psi) f$ is smooth and compactly supported, we get $\widetilde{\psi}f \in \mathcal{S}(\mathbb{R})$. It follows from this and $\varphi f = \varphi (\widetilde{\psi} f)$ 
that $\varphi f \in \mathcal{S}(\mathbb{R})$. Indeed, since $\varphi \in BC_0^{\infty}(\mathbb{R})$ and $\widetilde{\psi}f \in \mathcal{S}(\mathbb{R})$, we have
\begin{align*}
	\| x^{\beta} \partial_{x}^{\alpha} \{ \varphi (\widetilde{\psi} f)\}\|_{L^{\infty}}
	\lesssim 
	\sum_{k=0}^{\alpha}\| x^{\beta} \partial_{x}^{k} \varphi \partial_{x}^{\alpha-k}(\widetilde{\psi} f) \|_{L^{\infty}}
	\lesssim \sum_{k=0}^{\alpha}\| x^{\beta} \partial_{x}^{\alpha-k}(\widetilde{\psi} f) \|_{L^{\infty}}
	\leq C_{\alpha,\beta}. 
\end{align*}
This completes the proof. 
\end{proof}

\section*{Acknowledgement}

Research of the first author is partially supported
by an NSERC Discovery Grant.
The second author is supported by JSPS Overseas Research Fellowship and KAKENHI Grant-in-Aid for Early-Career Scientists No. JP18K13444. 


\end{document}